\documentclass{scrartcl}
\pdfoutput=1
\usepackage{imakeidx}
\usepackage[utf8]{inputenc}
\usepackage{lmodern} 
\usepackage{amsmath}
\usepackage{amssymb}
\usepackage{amsthm}
\usepackage{thmtools}
\usepackage{mathtools}
\usepackage{hyperref} 
\usepackage{enumitem} 
\usepackage{graphicx} 
\usepackage{dsfont} 
\usepackage{tikz} 
\usepackage[framemethod=TikZ]{mdframed}
\usepackage{ifthen}
\usepackage{supertabular} 
\usepackage[titletoc,title]{appendix}

\usepackage[
	backend=biber,
	style=alphabetic,
	citestyle=alphabetic 
]{biblatex}
\renewcommand{\subset}{\subseteq}
\newcommand{\lcb}{\left\lbrace} 
\newcommand{\rcb}{\right\rbrace} 
\newcommand{\cb}[1]{\lcb #1 \rcb} 
\newcommand{\cbOf}[1]{\mathopen{}\lcb #1 \rcb\mathclose{}} 
\newcommand{\lab}{\left[} 
\newcommand{\rab}{\right]} 
\newcommand{\ab}[1]{\lab #1 \rab} 
\newcommand{\abOf}[1]{\!\ab{#1}} 
\newcommand{\lb}{\left(} 
\newcommand{\rb}{\right)} 
\newcommand{\br}[1]{\lb #1 \rb} 
\newcommand{\brOf}[1]{\!\br{#1}} 
\newcommand{\abs}[1]{\left| #1 \right|} 
\newcommand*{\E}{\mathbb{E}} 
\newcommand*{\V}{\mathbb{V}} 
\let\Pr\relax
\newcommand*{\Pr}{\mathbb{P}} 
\newcommand{\sizedMid}[2]{#1 \, \kern-\nulldelimiterspace\mathopen{}\left| \vphantom{#1}\,#2\right.\mathclose{}\kern-\nulldelimiterspace}
\newcommand{\EOf}[1]{\E\abOf{#1}}

\newcommand{\VOf}[1]{\V\abOf{#1}}

\newcommand{\PrOf}[1]{\Pr\mathopen{}\lb #1 \rb\mathclose{}}
\newcommand{\Prof}[1]{\Pr(#1)}

\DeclareMathOperator{\diam}{\mathsf{diam}}

\DeclareMathOperator{\ball}{\mathrm{B}}

\DeclarePairedDelimiterX\Set[1]{\lbrace}{\rbrace}%
{  #1 }
\newcommand{\Ex}{\E\expectarg}
\DeclarePairedDelimiterX{\expectarg}[1]{[}{]}{%
	\ifnum\currentgrouptype=16 \else\begingroup\fi
	\activatebar#1
	\ifnum\currentgrouptype=16 \else\endgroup\fi
}
\newcommand{\innermid}{\nonscript\;\delimsize\vert\nonscript\;}
\newcommand{\activatebar}{%
	\begingroup\lccode`\~=`\|
	\lowercase{\endgroup\let~}\innermid 
	\mathcode`|=\string"8000
}
\newcommand*{\mc}[1]{\mathcal{#1}}
\newcommand*{\mb}[1]{\mathbb{#1}}
\newcommand*{\mr}[1]{\mathrm{#1}}
\newcommand*{\ms}[1]{\mathsf{#1}}
\newcommand*{\mo}[1]{\mathbf{#1}}
\newcommand*{\mf}[1]{\mathfrak{#1}}
\newcommand{\N}{\mathbb{N}}

\newcommand{\R}{\mathbb{R}}

\newcommand{\pr}{^\prime}
\newcommand{\prr}{^{\prime\prime}}
\newcommand{\prrr}{^{\prime\prime\prime}}
\def\integral from #1to #2of #3by #4;{\int_{#1}^{#2} \! #3 \mathrm{d}#4} %
\def\integralMeasure in #1of #2by #3of #4;{\int_{#1} \! #2{#4} #3{\mathrm{d}#4}} %
\def\mapping #1from #2to #3;{#1 \colon #2 \rightarrow #3}
\def\mappingDef #1from #2to #3maps #4to #5;{#1 \colon #2 \rightarrow #3,\ #4 \mapsto #5}
\def\seq #1by #2;{\br{#1}_{#2\in\N}}
\def\seqInText #1by #2;{(#1)_{#2\in\N}}
\newcommand{\innerProduct}[2]{\left\langle#1\,,\, #2\right\rangle}
\newcommand{\ip}[2]{\innerProduct{#1}{#2}}
\newcommand{\lebesgue}{\mathcal{L}}
\newcommand{\lebesguePow}[1]{\lebesgue^{#1}}

\newcommand{\sgn}{\mathsf{sgn}} 
\newcommand{\dl}{\mathrm{d}}
\def\converges for #1to #2;{\xrightarrow{#1} #2}
\def\convergesAlmostSurely for #1to #2;{\xrightarrow{#1}_{\mathsf{fs}} #2}
\def\convergesInProbability for #1to #2;{\xrightarrow{#1}_{\mathsf{p}} #2}
\def\convergesInL #1for #2to #3;{\xrightarrow{#2}_{\lebesguePow{#1}} #3}

\newcommand{\ind}{\mathds{1}}
\newcommand{\indOf}[1]{\ind_{\!#1}}%
\newcommand{\normof}[1]{\Vert #1 \Vert}
\newcommand{\normOf}[1]{\left\Vert #1 \right\Vert}
\newcommand{\equationFullstop}{\, .}
\newcommand{\eqfs}{\equationFullstop}
\newcommand{\equationComma}{\, ,}
\newcommand{\eqcm}{\equationComma}

\newcommand{\euler}{\mathrm{e}}
\DeclareMathOperator*{\argmin}{arg\,min}

\def\postBoxSkip{1.0ex}
\def\postBoxSkipCmd{\vskip\postBoxSkip}
\def\preBoxSkip{1.0ex}
\def\preBoxSkipCmd{\vskip\preBoxSkip}
\declaretheoremstyle[
	bodyfont=\normalfont,
	postfoothook={\postBoxSkipCmd},
	preheadhook={\preBoxSkipCmd},
	mdframed={
		backgroundcolor = black!2,
		startcode={\def\environmentEnumerateLabel{(\roman*)}},
}]{ruledBoxStyle}
\declaretheoremstyle[
	bodyfont=\normalfont,
	postfoothook={\postBoxSkipCmd},
	preheadhook={\preBoxSkipCmd},
	mdframed={
		backgroundcolor=white,
}]{ruledBoxStyleWhite}
\declaretheoremstyle[
	bodyfont=\normalfont,
	postfoothook={\postBoxSkipCmd},
	preheadhook={\preBoxSkipCmd},
	mdframed={
		backgroundcolor=black!2,
		linecolor = black!2,
		tikzsetting = {
			draw = black,
			line width = 2pt,%
			dashed,%
			dash pattern = on 10pt off 3pt
		},
}]{dashedBoxStyle}
\declaretheoremstyle[
	bodyfont=\normalfont,
	postfoothook={\postBoxSkipCmd},
	preheadhook={\preBoxSkipCmd},
	mdframed={
		linecolor = white,
		startcode={\def\environmentEnumerateLabel{(\roman*)}},
		tikzsetting = {
			draw = black,
			line width = 1pt,%
			loosely dotted,
		},
	}
]{dashedStyle}
\declaretheoremstyle[
	bodyfont=\normalfont,
	headformat={\NAME \NOTE},
	postfoothook={\postBoxSkipCmd},
	preheadhook={\preBoxSkipCmd},
	mdframed={
		linecolor = white,
		startcode={\def\environmentEnumerateLabel{(\roman*)}},
		tikzsetting = {
			draw = black,
			line width = 1pt,%
			loosely dotted,
		},
	}
]{dashedStyle2}
\declaretheoremstyle[
	bodyfont=\normalfont,
	postfoothook={\postBoxSkipCmd},
	preheadhook={\preBoxSkipCmd},
	mdframed={
		linecolor = black,
		innerlinewidth=1pt,outerlinewidth=1pt,
		middlelinewidth=1pt,
		linecolor=black,middlelinecolor=white,
		startcode={\def\environmentEnumerateLabel{(\roman*)}},
	}
]{doubleStyle}
\declaretheoremstyle[
	bodyfont=\normalfont,
	postfoothook={\postBoxSkipCmd},
	preheadhook={\preBoxSkipCmd},
	mdframed={
		backgroundcolor = black!4,
		linecolor = black!4,
		startcode={\def\environmentEnumerateLabel{(\alph*)}},
}]{boxStyle}
\declaretheoremstyle[
	headfont=\normalfont\itshape, 
	notefont=\normalfont\itshape, 
	notebraces={}{},
	bodyfont=\normalfont,
	qed=\qedsymbol,
	numbered=no,
	headindent=0pt,
	postheadspace=1ex,
	name={Proof},
	postheadhook={\def\environmentEnumerateLabel{(\roman*)}},
	mdframed={
		hidealllines = true,
		innerrightmargin = 0pt,
		innerleftmargin = 0pt,
		innertopmargin = 0pt,
		innerbottommargin = 0pt,
		leftmargin = 0pt,
		rightmargin = 0pt,
	}
]{proofStyle}
\declaretheoremstyle[
	bodyfont=\normalfont,
	postfoothook={\postBoxSkipCmd},
	preheadhook={\preBoxSkipCmd},
	mdframed={
		backgroundcolor = white,
		linecolor = black,
		startcode={\def\environmentEnumerateLabel{(\alph*)}},
		leftline = false,
		rightline = false,
}]{tobBottomStyle}
\declaretheoremstyle[
bodyfont=\normalfont,
]{standardStyle}
\declaretheorem[style=ruledBoxStyle,name=Definition]{definition}
\declaretheorem[style=ruledBoxStyle,name=Lemma]{lemma}

\declaretheorem[style=ruledBoxStyle,name=Theorem]{theorem}
\declaretheorem[style=ruledBoxStyle,name=Corollary]{corollary}
\declaretheorem[style=ruledBoxStyle,name=Theorem,numbered=no]{theorem*}
\declaretheorem[style=boxStyle,name=Remark]{remark}

\declaretheorem[style=dashedStyle2,name=Assumptions]{assumptions}

\makeatletter
\let\proof\@undefined
\let\endproof\@undefined
\makeatother
\declaretheorem[style=proofStyle]{proof}
\def\theoremContentInNewLine{\text{}}
\def\environmentEnumerateLabel{(\roman*)}
\newcounter{subExample}%
\renewcommand{\thmcontinues}[1]{Teil \arabic{subExample}} 
\newcommand{\bigOp}{\mo O_{\Pr}}
\newcommand{\smallop}{\mo o_{\Pr}}
\newcommand{\bigO}{\mo O}
\newcommand{\Ball}[3]{\ball_{#1}(#2,#3)}

\DeclareMathOperator{\tsize}{\ms{entr}}
\DeclareMathOperator{\entrn}{\ms{entr}_\mathit{n}}
\newcommand\leftidx[2]{{\vphantom{#2}}#1#2}

\newcommand{\olb}[2]{\leftidx{^{\mf b}}{\overline{#1#2}}}

\newcommand{\olc}[2]{\leftidx{^{\mf c}}{\overline{#1#2}}}

\newcommand{\olt}[2]{\overline{#1#2}^{2}}
\newcommand{\ol}[2]{\overline{#1#2}}
\newcommand{\assu}[2]{\hyperlink{#1}{\texttt{#2}}}
\newcommand{\assitem}[2]{%
	\item[]\hspace*{-0.5cm}%
	\hypertarget{#1}{\texttt{#2}}%
	\index[iass]{\texttt{#2}}%
	:\\}%
\indexsetup{noclearpage}
\makeindex[name=inot,title=Index of Notation]
\makeindex[name=iass,title=Index of Assumptions]
\begin{document}
\title{Convergence Rates for the Generalized Fréchet Mean via the Quadruple Inequality}
\subtitle{\url{https://github.com/chroetz/PaperQuadRate19}}
\author{Christof Schötz\\\href{mailto:math@christof-schoetz.de}{math@christof-schoetz.de}}
\date{}
\maketitle
\section*{Abstract}
For sets $\mathcal Q$ and $\mathcal Y$, the \textit{generalized Fréchet mean} $m \in \mathcal Q$ of a random variable $Y$, which has values in $\mathcal Y$, is any minimizer of $q\mapsto \mathbb E[\mathfrak c(q,Y)]$, where $\mathfrak c \colon \mathcal Q \times \mathcal Y \to \mathbb R$ is a cost function. 
There are little restrictions to $\mathcal Q$ and $\mathcal Y$. In particular, $\mathcal Q$ can be a non-Euclidean metric space.
We provide convergence rates for the empirical generalized Fréchet mean. 
Conditions for rates in probability and rates in expectation are given.
In contrast to previous results on Fréchet means, we do not require a finite diameter of the $\mathcal Q$ or $\mathcal Y$. Instead, we assume an inequality, which we call \textit{quadruple inequality}. It generalizes an otherwise common Lipschitz condition on the cost function. This quadruple inequality is known to hold in Hadamard spaces. We show that it also holds in a suitable way for certain powers of a Hadamard-metric.
\tableofcontents
\section{Introduction}
Let $\mc Q, \mc Y$ be sets, $Y$ a $\mc Y$-valued random variable, and $\mf c \colon \mc Y\times \mc Q \to \R$ a \textit{cost function}. Every element $m$ of the set  $\argmin_{q\in\mc Q}\Ex{\mf c(Y,q)}$ is called \textit{generalized Fréchet mean} or \textit{$\mf c$-Fréchet mean}.
Given independent copies $Y_1, \dots, Y_n$ of $Y$, natural estimators of the generalized Fréchet mean are elements $m_n$ of  the set $\argmin_{q\in\mc Q}  \frac1n\sum_{i=1}^n \mf c(Y_i,q)$. 
Our goal is to find suitable conditions for establishing convergence rates for such plug-in estimators. 

The described setting generalizes the usual setting for \textit{Fréchet means}, where $\mc Q = \mc Y$ is a metric space with metric $d : \mc Q \times \mc Q \to [0, \infty)$ and $\mf c = d^2$, which has been introduced in \cite{frechet48}.

The Fréchet mean has been investigated in many specific settings, often under a different name, e.g., center of mass or barycenter.
In the context of Riemannian manifolds, it has been studied -- among others -- by \cite{bhattacharya03}.
An asymptotic normality result for generalized Fréchet means on finite dimensional manifolds is shown in \cite{eltzner19}.
For complete metric spaces of nonpositive curvature, called Hadamard spaces, \cite{sturm03} shows how some classical results of probability theory in Euclidean spaces (e.g., strong law of large numbers, Jensen's inequality) can be transferred to the Fréchet mean setting. An algorithm for calculating Fréchet means in Hadamard spaces is described in \cite{bacak14}.

One important application of statistics in Hadamard spaces is the space phylogenetic trees. A phylogenetic tree represents the genetic relatedness of biological species. The geometry of the space of phylogenetic trees $T_m$ with $m$ leaves is studied in \cite{billera01}. In particular, it is shown that $T_m$ is a Hadamard space. There has been a lot of recent interest in statistics on $T_m$. E.g., \cite{barden18} show a central limit theorem for the Fréchet mean in $T_m$ and \cite{nye11} apply principal component analysis in that space.

For general metric spaces \cite{ziezold77} shows consistency of the Fréchet mean estimator. This is extended to generalized Fréchet means in \cite{huckemann11}.

The Fréchet mean estimator is a \textit{M-estimator}. Thus, we can build upon many classical and deep results from the M-estimation literature, see, e.g., \cite{vaart96, geer00, talagrand14}.
Using such  M-estimation techniques, rates of convergence in probability for Fréchet means in general bounded metric spaces are obtained in \cite{petersen16}; in fact the authors  consider a more complex regression setting. In \cite{dubey17} results on the analysis of variance in metric spaces are shown. 

Results on convergence rates in expectation, i.e, bounds on $\Ex{d(m,m_n)^2}$, seem to be rare in the literature on the Fréchet mean. Common are convergence rates in probability or exponential concentration. The latter also implies rates in expectation, but under rather strong assumptions. One publication that establishes rates in expectation more directly, for general cost functions in Euclidean spaces is \cite{banholzer17}.

The recent article \cite{legouic18} provides nonasymptotic concentration rates in general bounded metric spaces. Its relation to our results will be discussed in the next subsection.

\subsection{Our Contribution}

Our contribution consists of three parts:
\begin{enumerate}[label=(\alph*)]
\item 
We introduce a condition, which we call \textit{quadruple inequality}, that is used to establish convergence rates in probability and expectation for spaces with infinite diameter, see \autoref{thm:abstr_rate_prob}, \autoref{thm:abstr_rate_exp}, and \autoref{thm:abstr_weak_strong}.
\item 
We formulate our results in the setting of the generalized Fréchet mean with a cost-function $\mf c$ that is not restricted to being the square of a metric.
\item 
We prove a quadruple inequality for exponentiated metrics of Hadamard spaces, \autoref{thm:power_inequ}. We apply it to obtain rates of convergence for estimators of the Fréchet mean of an exponentiated metric.
\end{enumerate}

\cite{petersen16} and \cite{legouic18} show rates of convergence for metric spaces which have a finite diameter (or at least the support of the distribution of observations must be bounded). The proofs in both papers rely on \textit{empirical process theory}. In particular, they make use of \textit{symmetrization} and the \textit{generic chaining} to bound the supremum of an empirical process. But where \cite{legouic18} use that bound to be able to apply \textit{Talagrand's inequality} \cite{bousquet02}, \cite{petersen16} employ a \textit{peeling device} (also called \textit{slicing}; see, e.g., \cite{geer00}) to obtain rates. 
As a consequence, \cite{legouic18} achieve stronger results (nonasymptotic exponential concentration instead of $\bigOp$-statements), but they rely more heavily on the boundedness of the metric.
As our goal is to obtain results for spaces with infinite diameter, our proof technique is closer to \cite{petersen16}, i.e., we also apply a peeling device.

A law of large numbers, such that the estimator of the Fréchet mean converges in probability to the true value, implies that the estimator eventually is in a subset with finite diameter. Thus, for asymptotic rates in probability as in \cite{petersen16}, it is not very restrictive to assume a finite diameter. Our motivation to directly deal with infinite diameter comes from our interest in nonasymptotic results and in rates in expectation (asymptotic or nonasymptotic).

As \cite{petersen16} and \cite{legouic18}, we use the \textit{generic chaining}. Therefore we have entropy bounds as conditions of our theorems. These entropy bounds can be stated by requiring a bound on the \textit{covering numbers} 
\begin{equation*}
	N(Q, d, r) := \min\cbOf{k \in \N \,\middle|\, \exists q_1,\dots,q_k\in\mc Q\colon Q\subset \bigcup_{j=1}^k \ball_r(q_j) }\eqcm
\end{equation*}
where $(\mc Q, d)$ is a metric space, $Q \subset \mc Q$, and $r > 0$.
To be more precise, in a metric space $(\mc Q, d)$, we require $\log N(\ball_\delta(m), d, r) \leq \br{\frac{C\delta}{r}}^D$ for some constants $C,D >0$ and all $0 < r < \delta$, which is the same assumption as in \cite{legouic18}.
We note, that this requirement could be weakened by using the optimal bound on Rademacher (or Bernoulli) processes \cite{bednorz14} at the cost of a more complicated and less comprehensible condition.

In the classical Fréchet mean case, where $(\mc Q, d)$ is a metric space and the cost function is $\mf c = d^2$, the empirical process that has to be bounded consists of functions of the form $y \mapsto d(y,q)^2$ for $q\in\mc Q$. To apply some classical empirical process results, one requires a Lipschitz condition on these functions. In \cite{petersen16} and \cite{legouic18} this \textit{Lipschitz condition} is fulfilled by
\begin{equation}\label{eq:lipbound}
	d(y,q)^2 - d(y,p)^2 \leq  2 \diam(\mc Q) d(q, p)
\end{equation}
for all $y,q, p \in \mc Q$. Thus, a finite diameter is required. We show, that one can instead require that 
\begin{equation}\label{eq:nicequad}
	d(y,q)^2 - d(y,p)^2 - d(z,q)^2 + d(z,p)^2 \leq  2 d(y, z) d(q, p)
\end{equation}
holds for all $y,z,q, p \in \mc Q$ and then bound the supremum of the empirical process even if $\diam(\mc Q) = \infty$. Equation \eqref{eq:nicequad} is a special instance of what we call \textit{quadruple inequality}.

Roughly speaking, the transition from Lipschitz to quadruple condition removes certain squared terms and the right hand side by adding and subtracting further squared terms on the left hand side. This is related to the idea of defining the Fréchet mean as minimizer of $q \mapsto \Ex{d(Y, q)^2 - d(Y, o)^2}$ for an arbitrary fixed point $o\in\mc Q$ instead of $q \mapsto \Ex{d(Y, q)^2}$. Then, for existence of the Fréchet mean, only a first moment condition on $Y$ is required instead of a second moment condition, see \cite[Acknowledgement to Lutz Mattner]{sturm03}. 

The inequality \eqref{eq:nicequad} does not hold in every metric space. But it characterizes Hadamard spaces among geodesic metric spaces, see \cite{berg08}. In Hadamard spaces, \eqref{eq:nicequad} is known as \textit{Reshetnyak’s quadruple inequality} \cite{sturm03} or \textit{quadrilateral inequality} \cite{berg08} and can be interpreted as generalization  of the Cauchy--Schwartz inequality to metric spaces \cite{berg08}. Note that our results are not restricted to geodesic metric spaces.

In (subsets of) Hadamard spaces $(\mc Q, d)$, we can not only utilize the quadruple inequality with the squared metric $d^2$ \eqref{eq:nicequad}. But we show that for $d^a$ with $a \in [1,2]$, we also obtain a version of the quadruple inequality, namely
\begin{equation}\label{eq:power}
	d(y,q)^a - d(y,p)^a - d(z,q)^a + d(z,p)^a \leq  4a 2^{-a} d(y, z)^{a-1} d(q, p)
	\eqcm
\end{equation}
for all $y,z,q, p \in \mc Q$, see \autoref{thm:power_inequ}.
We show that the constant $4a 2^{-a}$ is optimal.
Similar to equation \eqref{eq:lipbound}, one can easily show --- using the mean value theorem --- that 
\begin{equation*}
	d(y,q)^a - d(y,p)^a \leq  a \diam(\mc Q)^{a-1} d(q, p)
\end{equation*}
for $a > 0$, $q,p,y\in\mc Q$, where $(\mc Q, d)$ is an arbitrary metric space. The proof of equation \eqref{eq:power} is much more complicated, see appendix \autoref{app:power_inequality}.

We state our convergence rate results in a general way, where observations live in a space $\mc Y$ and a cost function $\mf c \colon \mc Y \times \mc Q \to \R$ is minimized over $\mc Q$. The quadruple inequality then reads
\begin{equation*}
	\mf c(y,q) - \mf c(y,p) - \mf c(z,q) + \mf c(z,p) \leq  \mf a(y, z) \mf b(q, p)
\end{equation*}
for all $y,z \in \mc Y$ and $q, p \in \mc Q$ and an arbitrary function $\mf a \colon \mc Y \times \mc Y \to [0,\infty)$ and a pseudo-metric $\mf b \colon \mc Q \times \mc Q \to [0,\infty)$. This general formulation includes, among others, arbitrary bounded metric spaces, Hadamard spaces (including Euclidean and non-Euclidean spaces) with an exponentiated metric $d^a$, $a\in[1,2]$, and regression settings with $\mc Q \neq \mc Y$, where observations $(x,y) \in \mc Y$ are described by regression functions $(x \mapsto q(x)) \in \mc Q$.

Furthermore, some trivial statements in appendix \autoref{app:quad_stab} show that the quadruple inequality is stable under many operations such as taking subsets, limits, or product spaces.

We prove -- via a peeling device -- nonasymptotic rates of convergence in probability, \autoref{thm:abstr_rate_prob}. We do not achieve exponential concentration as \cite{legouic18}, but our results can be applied in cases where the cost function is not bounded by a finite constant, i.e., in metric spaces with infinite diameter. Furthermore, we show two ways of obtaining rates in expectation: One -- nonasymptotic -- under the assumption of a stronger version quadruple inequality, \autoref{thm:abstr_rate_exp}; the other -- asymptotic -- with a stricter entropy condition, \autoref{thm:abstr_weak_strong}.

Aside from the application in Hadamard spaces (including the use of the power inequality, \autoref{thm:power_inequ}), we illustrate our results in different toy examples: Euclidean spaces and infinite dimensional Hilbert spaces. In (convex subsets of) Hilbert spaces the Fréchet mean is equal to the expectation. Thus, these examples are interesting as a benchmark, because we can compare results from our general Fréchet mean approach to exact results. In two additional examples, we apply our results to nonconvex subsets of Hilbert spaces and to Hadamard spaces.
\subsection{Outline}
We start by presenting the convergence rates results of \autoref{thm:abstr_rate_prob} (rates in probability) and \autoref{thm:abstr_rate_exp} (rates in expectation) in the abstract setting in \autoref{sec:abstract_results}. The different versions of the quadruple inequality are discussed in \autoref{sec:quadruple}, including the power inequality, \autoref{thm:power_inequ}. This discussion concludes with the statement of \autoref{thm:abstr_weak_strong} (alternative route to rates in expectation). In \autoref{sec:applications}, we apply the abstract results in different settings: Euclidean spaces, infinite dimensional Hilbert spaces, nonconvex sets, and Hadamard spaces.
\section{Abstract Results}\label{sec:abstract_results}
In this section, we prove rates of convergence for the Fréchet mean in a very general setting, see section \ref{ssec:ares:generalsetting}. For rates in probability \autoref{thm:abstr_rate_prob} is stated in section \ref{ssec:ares:inprob} and for rates in expectation \autoref{thm:abstr_rate_exp} is stated in section \ref{ssec:ares:inexpec}. The proofs can be found in appendix \autoref{app:proofs:1and2}. Some remarks on further extensions are given in section \ref{ssec:ares:extensions}.
\subsection{Setting}\label{ssec:ares:generalsetting}
Here we define an \textit{Abstract Setting} in which we will state our most general results. This setting of the generalized Fréchet mean is similar to what is used in \cite{huckemann11, eltzner19}.

Let $\mc Q$ be a set, which is called \textit{descriptor space}. Let $(\mc Y, \Sigma_{\mc Y})$ be a measurable space, which is called \textit{data space}. Let $Y$ be a $\mc Y$-valued random variable. Let $\mf c \colon \mc Y \times \mc Q \to \R$ be a function such that $y\to \mf c (y,q)$ is measurable for every $q\in\mc Q$. We call $\mf c$ \emph{cost function}. Define $F \colon \mc Q \to \R\eqcm q\mapsto \Ex*{\mf c(Y, q)}$, assuming that $\Ex*{\abs{\mf c(Y, q)}} < \infty$ for all $q\in\mc Q$. The function $F$ is called \textit{objective function}. 
Let $n\in\N$. Let $Y_1, \dots, Y_n$ be independent copies of $Y$. Define $F_n \colon \mc Q \to \R\eqcm q\mapsto \frac1n \sum_{i=1}^n \mf c(Y_i, q)$. We call $F_n$ \textit{empirical objective function}. Let $\mf l \colon \mc Q \times \mc Q \to [0,\infty)$ be a function such that $\mf l(m,q)$ measures the \textit{loss} of choosing $q$ given that the true value is $m$. 

We want to bound $\mf l(m, m_n)$ for $m\in\argmin_{q\in\mc Q} F(q)$ and $m_n\in\argmin_{q\in\mc Q} F_n(q)$.
\index[inot]{$\mc Q$}
\index[inot]{$\mc Y$}
\index[inot]{$\Sigma_{\mc Y}$}
\index[inot]{$\mf c$}
\index[inot]{$F$}
\index[inot]{$Y_i$}
\index[inot]{$F_n$}
\index[inot]{$\mf l$}
\index[inot]{$m$}
\index[inot]{$m_n$}
\subsection{Rate of Convergence in Probability}\label{ssec:ares:inprob}
For our result on convergence rates in probability, we make some assumptions, which are listed in the following.
We denote the "closed" ball with center $o\in\mc Q$ of radius $r>0$ in the set $\mc Q$ with respect to an arbitrary distance function $d \colon \mc Q \times \mc Q \to [0 ,\infty)$ as
\begin{equation*}
	\ball_r(o, d) := \cb{q\in\mc Q \colon d(o, q) \leq r}
	\index[inot]{$\ball_r(o, d)$}
	\eqfs
\end{equation*}
\begin{assumptions}
\theoremContentInNewLine
\begin{enumerate}[label=\environmentEnumerateLabel]
\assitem{ass:ex}{Existence}
	We have $\Ex*{\abs{\mf c(Y, q)}} < \infty$ for all $q\in\mc Q$.
	There are $m_n\in\argmin_{q\in\mc Q} F_n(q)$ measurable and $m\in\argmin_{q\in\mc Q} F(q)$.
\assitem{ass:gr}{Growth}
	There are constants $\gamma > 0$ and $c_{\ms g}>0$ such that 
	$F(q)-F(m) \geq c_{\ms g} \mf l(m,q)^\gamma$ for all $q\in\mc Q$.
	\index[inot]{$\gamma$}\index[inot]{$c_{\ms g}$}
\assitem{ass:wquad}{Weak Quadruple}
	There are a function $\mf a \colon \mc Y \times \mc Y \to [0,\infty)$ measurable and a pseudo-metric $\mf b \colon \mc Q \times \mc Q \to [0,\infty)$, such that,
	for all $p,q\in\mc Q$, $y,z\in\mc Y$, we have
	\begin{equation*}
		\olc yq- \olc zq -\olc yp+\olc zp \ \leq \ \mf a(y,z)\, \mf b(q,p)
		\eqcm
	\end{equation*}
	where we use the notation  $\olc yq :=\mf c(y,q)$.\index[inot]{$\olc yq$}
	We call $\mf a$ the \textit{data distance} and $\mf b$ the \textit{descriptor metric}.
	\index[inot]{$\mf a$}
	\index[inot]{$\mf b$}
\assitem{ass:mom}{Moment}
	Let $\zeta \geq 1$. Set
	\begin{equation*}
		\mf M(\zeta) :=
		\begin{cases}
			\Ex*{\mf a(Y\pr,Y)^\zeta}\eqcm &\text{ if } \zeta\geq 2\eqcm\\
			\Ex*{\mf a(Y\pr,Y)^2}^{\frac{\zeta}{2}}\eqcm &\text{ if } \zeta\leq 2\eqcm
		\end{cases}
	\end{equation*}
	where $Y\pr$ is an independent copy of $Y$. We have $\mf M(\zeta)  < \infty$.
	\index[inot]{$\mf M(\zeta)$}
	\index[inot]{$\zeta$}
	\index[inot]{$Y\pr$}
\assitem{ass:ent}{Entropy}
	There are $\alpha, \beta > 0$ with $\frac\alpha\beta < \gamma$ such that
	\begin{equation*}
		\sqrt{\log N(\ball_\delta(m, \mf l), \mf b, r)} \leq c_{\ms e} \frac{\delta^\alpha}{r^\beta}
	\end{equation*}
	for a constant $c_{\ms e}>0$ and all $\delta, r > 0$.
	\index[inot]{$\alpha$}\index[inot]{$\beta$}\index[inot]{$c_{\ms e}$}
\end{enumerate}
\end{assumptions}
Here
\begin{equation*}
	N(A, \mf b, r) = \min\cbOf{k \in \N \,\middle|\, \exists q_1,\dots,q_k\in\mc Q\colon A\subset \bigcup_{j=1}^k \ball_r(q_j, \mf b) }\eqcm
	\index[inot]{$N(Q, d, r)$}
\end{equation*}
is the \textit{covering number} of $A \subset \mc Q$ with respect to $\mf b$-balls $\ball_r(\cdot, \mf b)$ of radius $r$.
\assu{ass:ent}{Entropy} is essentially the same condition as in \cite{legouic18}, but written down for the setting of the generalized Fréchet mean instead of the classical Fréchet mean in metric spaces.

We shortly discuss other assumptions before stating the theorem for rates of convergence in probability.

The measurability assumptions can be weakened by using the \textit{outer expectation}, see \cite{vaart96}. 

In \cite{barrio07}, the \assu{ass:gr}{Growth} condition is called \textit{margin condition}. It is called \textit{low noise assumption} in \cite{legouic18}.
If \assu{ass:gr}{Growth} holds for every distribution of $Y$ and we are in the traditional setting of the (not generalized) Fréchet mean, it implies that the metric space $\mc Q$ has nonpositive curvature:
Assume that $(\mc Q, d)$ is a complete \textit{geodesic space} \cite[Definition 1.1]{sturm03}, i.e., every pair of points $y_1, y_2$ has a \textit{mid-point} $m$, i.e., $\ol {y_1}m=\ol {y_2}m = \frac12 \ol {y_1}{y_2}$, where we use the notation $\ol qp := d(q,p)$. \index[inot]{$\ol qp$}
Set $\mc Y = \mc Q$, $\mf c = d^2$, and $\mf l = d$.
If $\Prof{Y=y_1} = \Prof{Y=y_2} = \frac12$ with $y_1, y_2 \in \mc Q$, the Fréchet mean $m\in\mc Q$ of $Y$ is the mid-point between $y_1$ and $y_2$.
If we assume that the growth condition holds for every distribution of $Y$, in particular, for every uniform 2-point distribution, with $c_{\ms g} = 1$ and $\gamma=2$, then
\begin{equation*}
	\frac12 \ol{y_1}q^2 + \frac12 \ol{y_2}q^2 - \frac12 \ol{y_1}m^2- \frac12 \ol{y_2}m^2 \geq \ol mq^2
	\eqfs
\end{equation*}
As $m$ is the mid-point between $y_1$ and $y_2$, we obtain
\begin{equation*}
	\ol mq^2 \leq \frac12 \ol{y_1}q^2 + \frac12 \ol{y_2}q^2 - \frac14 \ol{y_1}{y_2}^2
	\eqfs
\end{equation*}
This inequality implies that the space $(\mc Q, d)$ has nonpositive curvature \cite[Definition 2.1]{sturm03}. Such spaces are called \textit{Hadamard spaces}. Aside from the \assu{ass:gr}{Growth} condition they also fulfill the quadruple inequality, which we discuss in section \ref{ssse:quad:npc}.

The \assu{ass:wquad}{Weak Quadruple}-condition will be discussed in detail in \autoref{sec:quadruple}. Among other things, we will show that we have in a nice way in all Hadamard spaces, which include the Euclidean spaces.

The following theorem states rates of convergence for the estimator $m_n$ to the true value $m$ measured with respect to the loss function $\mf l$.
\begin{theorem}[Convergence rate in probability]\label{thm:abstr_rate_prob}
	In the \textit{Abstract Setting} of section \ref{ssec:ares:generalsetting}, assume that following conditions hold:
	\assu{ass:ex}{Existence}, \assu{ass:gr}{Growth}, \assu{ass:wquad}{Weak Quadruple}, \assu{ass:mom}{Moment}, \assu{ass:ent}{Entropy}.
	Define
	\begin{equation*}
		\eta_{\beta,n} := 
		\begin{cases} 
			n^{-\frac12} & \text{ for }\beta < 1\eqcm\\
			n^{-\frac12} \log(n+1) & \text{ for } \beta = 1\eqcm\\
			n^{-\frac1{2\beta}} & \text{ for } \beta > 1\eqfs
		\end{cases} 
		\index[inot]{$\eta_{\beta,n}$}
	\end{equation*}
	Then, for all $t > 0$, we have
	\begin{equation*}
		\PrOf{\eta_{\beta,n}^{-\frac{1}{\gamma-\frac{\alpha}{\beta}}} \mf l (m, m_n) \geq t} \leq 
		c \, \mf M(\zeta)\, t^{-\zeta(\gamma-\frac\alpha\beta)}
	\end{equation*}
	where $c > 0$ depends on $\alpha, \beta, \gamma, c_{\ms e}, c_{\ms g}, \zeta$.
\end{theorem}
The proof can be found in appendix \autoref{app:proofs:1and2}.

Without loss of generality, one can choose $\gamma=1$ by using the loss $\mf l\pr = \mf l^\gamma$. This is consistent with the result: If \assu{ass:gr}{Growth} and \assu{ass:ent}{Entropy} are fulfilled with $\mf l, \alpha, \beta, \gamma$, then they are also fulfilled with $\mf l\pr = l^\gamma, \alpha\pr = \frac{\alpha}{\gamma}, \beta\pr = \beta, \gamma\pr = 1$, which gives the same result.
We keep this redundancy in the parameters of the theorem for convenience.

A more common way of stating rates of convergence in probability is the $\bigOp$-notation, as in the following corollary. Note that the $\bigOp$-result is asymptotic and, thus, weaker than the non-asymptotic \autoref{thm:abstr_rate_prob}.
\begin{corollary}\label{cor:Op}
	In the \textit{Abstract Setting} of section \ref{ssec:ares:generalsetting}, assume that following conditions hold:
	\assu{ass:ex}{Existence}, \assu{ass:wquad}{Weak Quadruple}, \assu{ass:gr}{Growth}, \assu{ass:mom}{Moment} with $\zeta=1$, \assu{ass:ent}{Entropy}.
	Then
	\begin{equation*}
		\mf l (m, m_n) = \bigOp\brOf{\eta_{\beta,n}^{\frac{1}{\gamma-\frac{\alpha}{\beta}}}}
	\end{equation*}
	with $\eta_{\beta,n}$ as in \autoref{thm:abstr_rate_prob}.
\end{corollary}
It is possible to weaken the assumptions in \autoref{cor:Op}. In particular, we can restrict the \assu{ass:gr}{Growth} and \assu{ass:ent}{Entropy} conditions to hold only in a neighborhood of $m$ if we already know that $\mf l(m_n,m) \in \smallop(1)$.

In \autoref{thm:abstr_rate_prob}, the probability of large losses decays polynomially. If the exponent $\zeta(\gamma-\frac\alpha\beta)$ is strictly  greater than 1, we can integrate the tail probabilities to obtain a bound on the expectation of the loss.
\begin{corollary}\label{cor:probtoexpec}
	Let $\kappa \geq 1$.
	In the \textit{Abstract Setting} of section \ref{ssec:ares:generalsetting}, assume that following conditions hold:
	\assu{ass:ex}{Existence}, \assu{ass:wquad}{Weak Quadruple}, \assu{ass:gr}{Growth}, \assu{ass:mom}{Moment} with $\zeta > \kappa(\gamma - \frac\alpha\beta)^{-1}$, \assu{ass:ent}{Entropy}. Set $\xi := \zeta(\gamma-\frac\alpha\beta)\kappa^{-1}$. 
	Then
	\begin{align*}
		\eta_{\beta,n}^{-\frac{\kappa}{\gamma-\frac{\alpha}{\beta}}} \Ex{\mf l (m, m_n)^\kappa} 
		&\leq 
		c\pr \frac{\xi}{\xi-1} \mf M(\zeta)^{\frac1\xi}
		\eqfs
	\end{align*}
\end{corollary}
The proof can be found in appendix \autoref{app:proofs:1and2}.

\autoref{cor:probtoexpec} may require unnecessarily high moments as $\xi$ needs to be \textit{strictly} larger than 1. In the next section, we present a more direct approach to rates in expectation, that requires weaker moment conditions, at least in some settings.
\subsection{Rate of Convergence in Expectation}\label{ssec:ares:inexpec}
For obtaining rates in expectation directly, we need slightly modified, stronger assumptions.
\begin{assumptions}
\theoremContentInNewLine
\begin{enumerate}[label=\environmentEnumerateLabel]
	\assitem{ass:squad}{Strong Quadruple}
		Define $\dot{\mc Q} := \mc Q \setminus \ball_0(m, \mf l) = \cb{q\in\mc Q\colon \mf l(m, q) > 0}$.\index[inot]{$\dot{\mc Q}$}\index[inot]{$\ball_0(m, \mf l)$}
		There are functions $\mf b_m \colon \dot{\mc Q} \times \dot{\mc Q} \to [0,\infty)$ (possibly depending on $m$) and $\mf a \colon \mc Y \times \mc Y \to [0,\infty)$ with $\mf a$ measurable and $\xi\in(0,\gamma)$, such that,
		for all $p,q\in\dot{\mc Q}$, $y,z\in\mc Y$, we have
		\begin{equation*}
			\frac{\olc yq- \olc ym- \olc zq + \olc zm}{\mf l(m,q)^\xi} - \frac{\olc yp- \olc ym- \olc zp + \olc zm}{\mf l(m,p)^\xi} \leq \mf a(y,z) \, \mf b_m(q,p)
			\eqcm
			\index[inot]{$\mf b_m$}\index[inot]{$\mf a$}
		\end{equation*}
		Assume that $\mf b_m$ is a pseudo-metric on $\dot{\mc Q}$.
		We call $\mf a$ the \textit{data distance} and $\mf b_m$ the \textit{strong quadruple metric} at $m$.
		\index[inot]{$\xi$}
	\assitem{ass:smom}{Strong Moment}
		For $\zeta>0$, set
		\begin{equation*}
			\mf M(\zeta) :=
			\begin{cases}
				\Ex*{\mf a(Y\pr,Y)^\zeta}\eqcm &\text{ if } \zeta\geq 2\eqcm\\
				\Ex*{\mf a(Y\pr,Y)^2}^{\frac{\zeta}{2}}\eqcm &\text{ if } \zeta\leq 2\eqcm
			\end{cases}
			\index[inot]{$\mf M(\zeta)$}
		\end{equation*}
		where $Y\pr$ is an independent copy of $Y$.
		Let $\kappa\geq \gamma-\xi$ and assume $\mf M\brOf{\frac{\kappa}{\gamma-\xi}} < \infty$.
		\index[inot]{$\kappa$}
	\assitem{ass:sent}{Strong Entropy}
		We have $D := \diam(\dot{\mc Q}, \mf b_m) < \infty$ and there is $\beta > 0$ such that
		\begin{equation*}
			\sqrt{\log N(\dot{\mc Q}, \mf b_m, r)} \leq c_{\ms e} \br{\frac{D}{r}}^\beta
		\end{equation*}
		for all $r \in (0, D)$.
		\index[inot]{$D$}
		\index[inot]{$\beta$}
	\end{enumerate}
\end{assumptions}
For later use in the application to Hilbert spaces, section \ref{ssec:app:hilbert}, and for \autoref{thm:abstr_rate_exp}, we state the entropy part of \autoref{thm:abstr_rate_exp} in a more general way than in \autoref{thm:abstr_rate_prob}. To this end, we need to introduce different measures of entropy. 
\begin{definition}[Measures of Entropy]\label{chaining:def_entropy}
\theoremContentInNewLine
\begin{enumerate}[label=\environmentEnumerateLabel]
\item 
	Given a set $\mc Q$ an \emph{admissible sequence} is an increasing
	sequence $(\mc A_k)_{k\in\N_0}$ of partitions of $\mc Q$ such that $\mc A_0 = \mc Q$ and $\ms{card}(\mc A_k) \leq 2^{2^k}$ for $k\geq 1$.
	
	By an increasing sequence of partitions we mean that every set of $\mc A_{k+1}$ is
	contained in a set of $\mc A_k$. We denote by $A_k(q)$ the unique
	element of $\mc A_k$ which contains $q\in\mc Q$.
\item
	Let $(\mc Q, \mf b)$ be a pseudo-metric space.
	Define 
	\begin{equation*}
		\gamma_2(\mc Q, \mf b) := \inf \sup_{q\in\mc Q}\sum_{k=0}^\infty 2^{\frac k2} \diam(A_k(q), \mf b)\eqcm
	\end{equation*}
	where the infimum is taken over all admissible sequences in $\mc Q$ and
	\begin{equation*}
		\diam(A, \mf b) := \sup_{q,p\in A} \mf b(q,p)
	\end{equation*}
	for $A\subset \mc Q$.
	\index[inot]{$\diam(A, \mf b)$}
	\index[inot]{$\gamma_2(\mc Q, \mf b)$}
\item 
	Let $(Q, \mf b)$ be a pseudo-metric space and $n \in \N$.
	Define 
	\begin{equation*}
		\entrn(Q, \mf b) :=  \inf_{\epsilon > 0} \br{\epsilon \sqrt{n} + \int_\epsilon^\infty \!\! \sqrt{\log N(Q, \mf b, r)}\dl r} 
		\eqfs
	\end{equation*}
	\index[inot]{$\entrn(Q, \mf b)$}
\end{enumerate}
\end{definition}
Items (i) and (ii) are basic definitions form \cite{talagrand14}. Item (iii) is just a convenient notation.
\begin{theorem}[Convergence rate in expectation]\label{thm:abstr_rate_exp}
	In the \textit{Abstract Setting} of section \ref{ssec:ares:generalsetting}, assume that following conditions hold:
	\assu{ass:ex}{Existence}, \assu{ass:gr}{Growth}, \assu{ass:squad}{Strong Quadruple}, \assu{ass:smom}{Strong Moment}.
	Then, we have
	\begin{equation*}
		\Ex*{\mf l(m ,m_n)^\kappa}
		\ \leq \ 
		c n^{-\frac{\kappa}{2(\gamma-\xi)}} \mf M\brOf{\frac{\kappa}{\gamma-\xi}} \min\brOf{\entrn(\mc Q, \mf b_m), \gamma_2(\mc Q, \mf b_m)}^{\frac{\kappa}{\gamma-\xi}}
		\eqcm
	\end{equation*}
	where $c>0$ depends only on $\kappa,\gamma,\xi,c_{\ms g}$.
	
	If additionally \assu{ass:sent}{Strong Entropy} holds, then
	\begin{equation*}
		\Ex*{\mf l(m ,m_n)^\kappa}
		\ \leq \
		C \, \mf M\brOf{\frac{\kappa}{\gamma-\xi}} \,  D^{\frac{\kappa}{\gamma-\xi}} \,  \eta_{\beta,n}^{\frac{\kappa}{\gamma-\xi}}
		\eqcm
	\end{equation*}
	where 
	\begin{equation*}
		\eta_{\beta,n} := 
		\begin{cases} 
			n^{-\frac12} & \text{ for }\beta < 1\eqcm\\
			n^{-\frac12} \log(n+1) & \text{ for } \beta = 1\eqcm\\
			n^{-\frac1{2\beta}} & \text{ for } \beta > 1\eqcm
		\end{cases} 
	\end{equation*}
	and $C>0$ depends only on $\kappa,\beta,\gamma,\xi,c_{\ms g}$.
\end{theorem}
The proof can be found in appendix \autoref{app:proofs:1and2}.

As in \autoref{thm:abstr_rate_prob} the statement contains some redundancy. E.g., by using the loss $\tilde {\mf l} = \mf l^\xi$ we set $\xi=1$ without loss of generality. Then the growth exponent and the resulting rate of convergence will scale accordingly.
\subsection{Further Extensions}\label{ssec:ares:extensions}
In general $M := \argmin_{q\in\mc Q} \Ex{\mf c(Y, q)}$ is some subset of $\mc Q$. One can also extend the main theorems of this paper to deal with a the whole set of Fréchet means and Fréchet mean estimators. To do that, the \assu{ass:gr}{Growth} condition has to be stated as growth of the minimal distance to $M$. Furthermore, some of the statements and assumptions made in the theorems and proofs have to be modified so that the hold uniformly over all $m\in M$. Additionally, one has to think about the right notion of convergence for sets. We found that those results hard to read without significantly increasing insight into the problem, which is why we chose to stick with unique Fréchet means and only remark that an extension to Fréchet mean sets is possible.

One can also consider $\varepsilon\text{-}\argmin$-sets, i.e., the sets of elements which minimize a function up to an $\varepsilon > 0$. If one chooses $m_n \in \varepsilon_n\text{-}\argmin_{q\in\mc Q} F_n(q)$ with $\varepsilon_n \to 0$ fast enough, the convergence rate is of the same as for the absolute minimizer.
\section{Quadruple Inequalities}\label{sec:quadruple}
Recall the definition of the weak and strong quadruple inequalities.
Let $(\mc Q, \mf b)$ be a pseudo-metric space (\textit{descriptor space} with \textit{descriptor metric}), 
$(\mc Y, \Sigma_{\mc Y})$ a measurable space (\textit{data space}), 
$\mf c \colon \mc Y \times \mc Q \to \R$ such that $y\mapsto \mf c(y,q)$ is measurable for every $q\in\mc Q$ (\textit{cost function}),
$\mf a \colon \mc Y \times \mc Y \to [0,\infty)$ measurable (\textit{data distance}),
$m \in \mc Q$ (reference point, usually the Fréchet mean),
$\mf l\colon \mc Q \times \mc Q \to [0,\infty)$ (\textit{loss}),
$\xi > 0$ (\textit{rate parameter}),
$\dot{\mc Q} = \cb{q\in\mc Q\colon \mf l(m, q) > 0}$
$\mf b_m \colon \dot{\mc Q} \times \dot{\mc Q} \to [0,\infty)$ a pseudo-metric on $\dot{\mc Q}$ (\textit{strong quadruple metric} at $m$).
We write $\olc yq :=\mf c(y,q)$.
\begin{enumerate}[label=(\alph*)]
\item 
The tuple $(\mc Q, \mc Y, \mf c, \mf a, \mf b)$ fulfills the \emph{(weak) quadruple inequality}
	if and only if,
	for all $p,q\in\mc Q$, $y,z\in\mc Y$, we have
	\begin{equation*}
		\olc yq- \olc zq -\olc yp+\olc zp \leq \mf a(y,z) \mf b(q,p)
		\eqfs
	\end{equation*}
\item	
	The tuple $(\mc Q, \mc Y, \mf c, \mf l^\xi, \mf a, \mf b_m)$ fulfills the \emph{strong quadruple inequality} at $m\in\mc Q$ 
	if and only if,
	for all $p,q\in\dot{\mc Q}$, $y,z\in\mc Y$, we have
	\begin{equation*}
		\frac{\olc yq- \olc ym- \olc zq + \olc zm}{\mf l(m,q)^\xi} - \frac{\olc yp- \olc ym- \olc zp + \olc zm}{\mf l(m,p)^\xi} \leq \mf a(y,z) \mf b_m(q,p)
		\eqfs
	\end{equation*}
\end{enumerate}
There are a couple of trivial stability results for quadruple inequalities, see appendix \autoref{app:quad_stab}.

In section \ref{ssec:bounded} we compare the quadruple inequality with a more common Lipschitz property. The simplest advantageous applications of the quadruple inequality are in inner product spaces and quasi-inner product spaces, as is discussed in section \ref{ssec:quad:ip}. In section \ref{ssec:pwer_inequality} we state the power inequality, \autoref{thm:power_inequ}. It allows to establish quadruple inequalities for exponentiated metrics. We conclude with \autoref{thm:abstr_weak_strong} in section \ref{ssec:weak_implies_strong}, which yields rates of convergence in expectation under the assumption of only a weak quadruple inequality instead of a strong one as in \autoref{thm:abstr_rate_exp}.
\subsection{Bounded Spaces and Smooth Cost Function}\label{ssec:bounded}
Let $(\mc Q, d)$ be a metric space and use the notation $\ol qp = d(q,p)$.
For obtaining convergence rates in probability for the Fréchet mean estimator, \cite{petersen16} use 
\begin{equation*}
	 \ol yq^2 - \ol yp^2 = \br{\ol yq - \ol yp}\br{\ol yq + \ol yp} \leq 2 \ol qp \diam(\mc Q)
\end{equation*}
for all $q,p,y\in\mc Q$.
In the proof of \autoref{thm:abstr_rate_prob}, we have replaced this bound by the weak quadruple inequality, i.e.,
\begin{equation*}
	\olc yq - \olc yp - \olc zq + \olc zp \leq \mf a(y,z) \mf b(q,p)
	\eqfs
\end{equation*}
This generalizes the results by \cite{petersen16} as for bounded metric spaces $(\mc Q, d)$ and cost function $\mf c = d^2$, the weak quadruple inequality holds with $\mf a(y,z) = 4 \diam(\mc Q)$ and $\mf b = d$:
\begin{equation*}
	\ol yq^2 - \ol yp^2 - \ol zq^2 + \ol zp^2 \leq \abs{\ol yq^2 - \ol yp^2} + \abs{\ol zq^2 - \ol zp^2} \leq 4 \ol qp \diam(\mc Q)
	\eqfs
\end{equation*}
More generally, if we can show Lipschitz continuity in the second argument of the cost function, i.e., $\olc yq - \olc yp \leq \mf a(y) \mf b(q,p)$, then the quadruple inequality holds with data distance $\mf a(y)+\mf a(z)$ and descriptor metric $\mf b$.
But this might lead to an unnecessarily large bound. We will see in section \ref{ssse:quad:npc} that at least for certain metric spaces, we can find a bound via the quadruple inequality that does not involve the diameter of the space and, thus, allows for meaningful results in unbounded spaces.
\subsection{Relation to Inner Product and Cauchy--Schwartz Inequality}\label{ssec:quad:ip}
\subsubsection{Inner Product Space}\label{sssec:inner_product_space}
Let $(\mc Q, d)$ be a metric space such that $d$ comes from an inner product $\ip{\cdot}{\cdot}$ on $\mc Q$, i.e., $\mc Q$ is a subset of am inner product space and $d(y,q)^2 = \ip{y-q}{y-q}$. Use $\mc Y = \mc Q$ and the squared metric as cost function, $\mf c = d^2$. Then
\begin{align*}
	\olc yq- \olc zq -\olc yp+\olc zp &= -2\ip{y-z}{q-p} \\&\leq 2\normof{q-p}\normof{y-z}
	\eqfs
\end{align*} 
Here the Cauchy--Schwartz inequality gives rise to an instance of the weak quadruple inequality.
The very general framework that we impose also allows for a more flexible bound:
If $\mc Q \subset \mb H$ is the subset of an infinite dimensional, separable Hilbert space $\mb H$, we can use a weighted Cauchy--Schwartz inequality:
Let $s = (s_k)_{k\in\N} \subset (0,\infty)$. Then
\begin{align*}
	\olc yq- \olc zq -\olc yp+\olc zp \leq 2 \normof{y-z}_{s^{-1}} \normof{q-p}_s 
	\eqcm
\end{align*} 
where $\normof{x}^2_s = \sum_{k=1}^\infty s_k^2 x_k^2$ with generalized Fourier coefficients $(x_k)_{k\in\N}$ with respect to a fixed orthonormal basis of $\mb H$.

For the strong quadruple inequality, we set $\xi=1$, $\mf l(q,p) = \normof{q-p}$ and obtain
\begin{align*}
	&\frac{\olc yq- \olc ym- \olc zq + \olc zm}{\mf l(m,q)} - \frac{\olc yp- \olc ym- \olc zp + \olc zm}{\mf l(m,p)}
	\\&=
	-2\left\langle y-z, \frac{q-m}{\normof{q-m}}-\frac{p-m}{\normof{p-m}} \right\rangle
	\\&\leq
	2\normof{y-z} \normOf{\frac{q-m}{\normof{q-m}}-\frac{p-m}{\normof{p-m}}}
	\eqfs
\end{align*}
Thus, the strong quadruple inequality hold with $\mf a(y,z) = 2\normof{y-z}$ and $\mf b_m(q,p) = \normOf{\frac{q-m}{\normof{q-m}}-\frac{p-m}{\normof{p-m}}}$. The pseudo-metric $\mf b_m$ first projects the points $q$ and $p$ onto the surface of unit ball around $m$ and then measures their Euclidean distance.

The analogous result for the weighted Cauchy--Schwartz inequality is
\begin{align*}
	&\frac{\olc yq- \olc ym- \olc zq + \olc zm}{\mf l(m,q)} - \frac{\olc yp- \olc ym- \olc zp + \olc zm}{\mf l(m,p)}
	\\&\leq
	2\normof{y-z}_{s^{-1}} \normOf{\frac{q-m}{\normof{q-m}}-\frac{p-m}{\normof{p-m}}}_{s}
	\eqfs
\end{align*}
\subsubsection{Bregman Divergence}
Let $\mc Q\subset \R^r$ be a closed convex set. Let $\psi\colon\mc Q\to \R$ be a continuously differentiable and strictly convex function. The Bregman divergence $D_\psi \colon \mc Q \times \mc Q \to [0,\infty)$ associated with $\psi$ for points $y, q \in \mc Q$ is defined as $D_\psi(y,q) = \psi(y)-\psi(q)-\ip{\nabla \psi(q)}{y-q}$. It is the difference between the value of $\psi$ at point $y$ and the value of the first-order Taylor expansion of $\psi$ around point $q$ evaluated at point $y$. It is well-known, that the minimizer $m$ of $q \mapsto\Ex{D_\psi(Y,q)}$ for a random variable $Y$ with $\Ex{D_\psi(Y,q)}<\infty$ for all $q\in\mc Q$ is the expectation $m = \Ex{Y}$, see \cite[Theorem 1]{banerjee05}. The Bregman divergence $\mf c = D_\psi$ fulfills the weak quadruple inequality:
\begin{align*}
	D_\psi(y,q) - D_\psi(z,q) - D_\psi(y,p) +D _\psi(z,p) &= \ip{\nabla \psi(q)-\nabla \psi(p)}{y-z} \\&\leq \normof{\nabla \psi(q)-\nabla \psi(p)}\normof{y-z}
	\eqfs
\end{align*} 
Similarly, we obtain a version of the strong quadruple inequality with  $\xi=1$, $\mf l(q,p) = \normof{q-p}$, 
\begin{align*}
	&\frac{\olc yq- \olc ym- \olc zq + \olc zm}{\mf l(m,q)} - \frac{\olc yp- \olc ym- \olc zp + \olc zm}{\mf l(m,p)}
	\\&=
	\left\langle y-z, \frac{\nabla \psi(q)-\nabla \psi(m)}{\normof{q-m}}-\frac{\nabla \psi(p)-\nabla \psi(m)}{\normof{p-m}} \right\rangle
	\\&\leq
	\normof{y-z} \normOf{\frac{\nabla \psi(q)-\nabla \psi(m)}{\normof{q-m}}-\frac{\nabla \psi(p)-\nabla \psi(m)}{\normof{p-m}}}
	\eqfs
\end{align*}
\subsubsection{Hadamard Spaces and Quasi-Inner Product}\label{ssse:quad:npc}
Let $(\mc Q, d)$ be a metric space. Use the notation $\ol qp := d(q,p)$. We use the squared metric as the cost function $\mf c(y,q) = d(y,q)^2 = \ol yq^2$.
One particularly nice version of the weak quadruple inequality with this cost function is
\begin{equation*}
	\ol yq^2 - \ol yp^2 - \ol zq^2 + \ol zp^2 \leq 2 \,\ol yz\, \ol qp
	\eqfs
\end{equation*}
Let us call this inequality the \textit{nice quadruple inequality}.
As seen before, this holds for subsets of inner product spaces. It also plays an important role for geodesic metric spaces.
In this section, we paraphrase some results of \cite{berg08}. In particular, we state that the nice quadruple inequality characterizes CAT(0)-spaces.

Let $(\mc Q, d)$ be a metric space. A \textit{curve} is a continuous mapping $\gamma \colon [a,b] \to \mc Q$, where $[a,b]$ is a closed interval. The \textit{length} of the curve $\gamma \colon [a,b] \to \mc Q$ is
\begin{equation*}
	L(\gamma) := \sup\cbOf{\sum_{i=1}^I d(\gamma(t_{i-1}), \gamma(t_i)) \,\Bigg\vert\, a = t_0 < t_1 < \dots < t_I = b, I\in\N}
	\eqfs
\end{equation*}
A curve $\gamma \colon [a,b] \to \mc Q$ is called a \textit{geodesic} if $L(\gamma) = d(\gamma(a),\gamma(b))$.
A metric space is called \textit{geodesic}, if any two points $q,p\in\mc Q$ can be joined by a geodesic $\gamma \colon [a,b] \to \mc Q$ with $\gamma(a) = q, \gamma(b) = p$.
A \textit{midpoint} of two points $q,p\in\mc Q$ is a point $m\in \mc Q$ such that $\ol qm = \ol pm = \frac12 \ol qp$.
A complete metric space is a geodesic space if and only if all pairs of points have a midpoint, see \cite[Proposition 1.2]{sturm03}.
Now, let $(\mc Q, d)$ be a geodesic metric space.
For any triple of points $a,b,c\in\mc Q$ one can construct a \textit{comparison triangle} in the Euclidean plane with corners $a\pr, b\pr, c\pr \in \R^2$, such that $\ol ab = \normof{b\pr - a\pr}$, $\ol ac = \normof{c\pr - a\pr}$, and $\ol bc = \normof{c\pr - b\pr}$. A geodesic metric space  $(\mc Q, d)$ is called \textit{CAT(0)} if and only if for every triple of points $a,b,c\in\mc Q$ with comparison triangle $(a\pr, b\pr, c\pr)$ following condition holds: For every point $d$ on a geodesic connecting $a$ and $b$, we have $\ol dc \leq \normof{c\pr - d\pr}$, where $d\pr \in \R^2$ is the point on the edge of the comparison triangle between $a\pr$ and $b\pr$ such that $\normof{d\pr- a\pr} = \ol ad$.
A complete CAT(0)-space is called \textit{Hadamard space} or \textit{global NPC space} (\textbf{n}on\textbf{p}ositive \textbf{c}urvature).

A metric space $(\mc Q, d)$ is said to fulfill the \textit{NPC-inequality} if and only if for all $y_1, y_2 \in \mc Q$ there exists a point $m\in\mc Q$ such that for all $q \in \mc Q$, we have
$\ol mq^2 \leq \frac12 \ol{y_1}q^2 + \frac12 \ol{y_2}q^2 - \frac14 \ol{y_1}{y_2}^2$. Then $m$ is the midpoint of $y_1$ and $y_2$.

A characterization of CAT(0)-spaces can be found in \cite[Section 2]{sturm03}:
	A metric space is CAT(0) if and only if it fulfills the NPC-inequality.

Another characterization of CAT(0)-spaces by the nice quadruple inequality is given in \cite[Corollary 3]{berg08}:
A geodesic space is CAT(0) if and only if it fulfills the nice quadruple inequality.
	
In \cite{berg08}, the authors define the \textit{quadrilateral cosine} for $q,p,y,z\in\mc Q$ as
\begin{equation*}
	\mr{cosq}\brOf{\vec{yz}, \vec{qp}} := \frac{\ol yq^2 - \ol yp^2 - \ol zq^2 + \ol zp^2}{-2 \,\ol yz\, \ol qp}\eqfs
\end{equation*}
Obviously, the statement $\mr{cosq}\brOf{\vec{yz}, \vec{qp}} \leq 1$ for all $q,p,y,z\in\mc Q$ is equivalent to the nice quadruple inequality. To further motivate this notation and compare it with inner product spaces, they introduce a \textit{quasilinearization} of the metric space and a \textit{quasi-inner product}:
Define $\ip{\vec{yz}}{\vec{qp}}_d =	\mr{cosq}\brOf{\vec{yz}, \vec{qp}} \normof{\vec{yz}}_d \normof{\vec{qp}}_d$, where $\normof{\vec{yz}}_d := \ol yz$.
Thus, the nice quadruple inequality can be viewed as the Cauchy--Schwartz inequality of the quasi-inner product. 
\subsection{Power Inequality}\label{ssec:pwer_inequality}
If the metric space $(\mc Q,d)$ fulfills the nice quadruple inequality, i.e, $\ol yq^2 - \ol yp^2 - \ol zq^2 + \ol zp^2 \leq 2 \,\ol yz\, \ol qp$, where $\ol yq = d(y,q)$, then $(\mc Q,d^a)$, $a\in[\frac12,1]$, also fulfills a weak quadruple inequality with a suitably adapted bound. The implications of this result for the estimators of the corresponding Fréchet means are discussed in section \ref{sssec:app:hadamard:pfm}.

According to \cite{deza09}, the metric $d^a$ is called \textit{power transform metric} or \textit{snowflake transform metric}.
\begin{restatable}[Power Inequality]{theorem}{TheoremPowerIneqaulity}\label{thm:power_inequ}
	Let $(\mc Q, d)$ be a metric space. Use the short notation $\ol qp := d(q,p)$.
	Let $q,p,y,z\in\mc Q$, $a\in[\frac12,1]$. 
	Assume
	\begin{equation}\label{eq:nice_power_inequ}
		\olt yq - \olt yp - \olt zq + \olt zp \leq 2 \, \ol yz\, \ol qp
		\eqfs
	\end{equation}
	Then 
	\begin{equation}\label{eq:nice_power_inequ_res}
		\ol yq^{2a} - \ol yp^{2a} - \ol zq^{2a} + \ol zp^{2a} \leq 8 a 2^{-2a} \, \ol yz^{2a-1}\, \ol qp
		\eqfs
	\end{equation}
\end{restatable}
In particular, if the metric space $(\mc Q,d)$ fulfills the nice quadruple inequality and $a\in[\frac12,1]$, then the weak quadruple inequality for $\mf c = d^{2a}$ is fulfilled with $\mf a = 8 a 2^{-2a} d^{2a-1}$ and $\mf b = d$.

Following the intermediate step \autoref{lmm:power_implies} (appendix \autoref{app:power_inequality}) in the proof of \autoref{thm:power_inequ}, one can easily show a similar result if the constant on the right hand side of equation \eqref{eq:nice_power_inequ} is larger than 2. Only the constant $8 a 2^{-2a}$ on the right hand side of equation \eqref{eq:nice_power_inequ_res} changes.

The theorem applies to subsets of Hadamard spaces. But note that it is not required that $\mc Q$ is geodesic, but can consist of only the points $q,p,y,z$.  As a statement purely about metric spaces, it might be of interest outside the realm of statistics.  

In \autoref{coro:probrates_power} (section \ref{ssse:quad:npc}) it is used to show rates of  convergence for the Fréchet mean estimator of the power transform metric $d^a$. There the asymmetry of the exponents of the factors on the right hand side of  \eqref{eq:nice_power_inequ_res} is essential for proving the result under weak assumption. 

Unfortunately, the only proof of this statement that the author was able to derive (see appendix \ref{app:power_inequality}) is very long and does not give much insight into the problem as it mostly consists of distinguishing many cases and then using simple calculus. The author is convinced that a more appealing proof is possible.

The concave function $[\frac12, 1] \to (0,\infty), a\mapsto8a2^{-2a}$ is maximal at $a_0 = (2\ln(2))^{-1} \approx 0.721$ with $8a_02^{-2a_0} = \frac{4}{\euler \ln(2)} \leq 2.123$. Thus, the constant factor in the bound is very close to 2, but 2 is not sufficient.

In appendix \autoref{app:power_inequ_opti}, we show that $8a2^{-2a}$ is the optimal constant, and that we cannot extend \autoref{thm:power_inequ} to $a > 1$ or $a < \frac12$.

It is not known to the author whether the nice quadruple inequality in $(\mc Q, d)$ does or does not imply the nice quadruple inequality in $(\mc Q, d^a)$ for $a \in (\frac12, 1)$, i.e.,
\begin{equation*}
	\ol yq^{2a} - \ol yp^{2a} - \ol zq^{2a} + \ol zp^{2a} \leq 2 \, \ol yz^{a}\, \ol qp^{a}
	\eqfs
\end{equation*}
\subsection{Weak Implies Strong}\label{ssec:weak_implies_strong}
The weak quadruple inequality is well justified as a condition: Aside from allowing to establish rates in probability (\autoref{thm:abstr_rate_prob}), it can be interpreted as a form of Cauchy--Schwartz inequality (section \ref{ssse:quad:npc}), it is fulfilled in a large class of metric spaces (bounded metric spaces, Hadamard spaces, appendix \autoref{app:quad_stab}), and the power inequality (\autoref{thm:power_inequ}) implies even more applications with a nice interpretation in statistics (section \ref{sssec:app:hadamard:pfm}).

The case for the strong quadruple inequality, which we use in \autoref{thm:abstr_rate_exp} to establish rates in expectation, seems much weaker. 
Although it can be established in Hilbert spaces, see section \ref{sssec:inner_product_space}, it is not directly clear whether we can have a suitable version for Hadamard spaces or a power inequality. 

The next section examines the strong quadruple inequality in Hadamard spaces and concludes with a negative result. Thereafter, we discuss an approach to infer convergence rates in expectation when only assuming the weak quadruple inequality by showing that a weak quadruple inequality imply certain strong quadruple inequalities. This approach is executed to obtain \autoref{thm:abstr_weak_strong} for convergence rates in expectation, where the result holds only asymptotically, in contrast to \autoref{thm:abstr_rate_exp}.
\subsubsection{Projection Metric}
In Euclidean spaces, we can take $\mf b_m(q,p) = \normOf{\frac{q-m}{\normof{q-m}}-\frac{p-m}{\normof{p-m}}}$ as the strong quadruple metric. This pseudo-metric can be written down only depending on the metric (not the norm or vector space operations) as
\begin{equation*}
	d_{m}^{\ms{proj}}(q,p) := \sqrt{\frac{\ol qp^2 - \br{\ol qm - \ol pm}^2}{\ol qm\, \ol pm}}\eqcm\quad  d_{m}^{\ms{proj}}(q,p) = \mf b_m(q,p)
	\eqfs
\end{equation*} 
The metric $d_{m}^{\ms{proj}}(q,p)$ can be defined in any metric space. Unfortunately, it does not yield a strong quadruple inequality in non-Euclidean Hadamard spaces in the same way as in Euclidean spaces. See appendix \autoref{app:dproj} for details.
\subsubsection{Power Metric}
To establish rates of convergence in expectation for the $\mf c$-Fréchet mean, given that a weak quadruple inequality holds, we first show that some version of the strong quadruple inequality is implied by the weak one, \autoref{lmm:weak_implies_strong_power}. Unfortunately, we obtain a strong quadruple distance $\mf b_m$ such that the measure of entropy $\tsize(\mc Q, \mf b_m)$ might be infinite. To solve this problem, we define an increasing sequence of sets $\mc Q_n$ such that $\mc Q_n \subset \mc Q_{n+1}$ and $\bigcup_{n\in\N} \mc Q_n = \mc Q$
with distances $\mf b_{m,n}$ such that the strong quadruple inequality is fulfilled on $\mc Q_n$ with strong quadruple distance $\mf b_{m,n}$, and $\tsize(\mc Q_n, \mf b_{m,n})$ is finite and can be suitably controlled in $n$. This allows us to prove an asymptotic result for the rate of convergence in expectation, \autoref{thm:abstr_weak_strong}.
\begin{lemma}\label{lmm:weak_implies_strong_power}
Assume $(\mc Q, \mc Y, \mf a, \mf b, \mf c)$ fulfills the weak quadruple inequality.
Let $\xi\in[0,1]$.
Then
\begin{equation}\label{eq:strquadpowerbound}
	\frac{\olc yq - \olc ym - \olc zq + \olc zm}{\mf b(q,m)^{\xi}}
	-
	\frac{\olc yp - \olc ym - \olc zp + \olc zm}{\mf b(p,m)^{\xi}}
	\leq
	2^\xi\, \mf a(y,z)\, {\mf b}(q,p)^{1-\xi}
\end{equation}
for all $y,z,q,p,m \in \mc Q$ with $\mf b(q,m), \mf b(p,m) > 0$.
\end{lemma}
See appendix \autoref{app:power_metric_bound} for a proof.
We would like to have $\xi$ large, i.e., close to 1, to obtain the same rate of convergence in expectation as in probability. We achieve that by defining sequences $\xi_n \nearrow 1$ and $\mc Q_n \nearrow \mc Q$, and control the entropy of $\mc Q_n$ with respect to $\mf b^{1-\xi_n}$.

To state the result, we have to modify the \assu{ass:ent}{Entropy} and the \assu{ass:ex}{Existence} condition.
Recall the definition of the objective function $F(q) = \Ex{\olc Yq}$ and the empirical objective function $F_n(q) = \frac1n \sum_{i=1}^n \olc{Y_i}q$.
\begin{assumptions}
\theoremContentInNewLine
	\begin{enumerate}[label=\environmentEnumerateLabel]
	\assitem{ass:ex2}{Existence'}
		We have $\Ex*{\abs{\mf c(Y, q)}} < \infty$ for all $q\in\mc Q$.
		Let $o\in\mc Q$. Define $R_n := n$ and $\mc Q_n := \Ball{R_n}{o}{\mf b}$.
		There are $m^{Q_n}_n \in \argmin_{q\in \mc Q_n} F_n(q)$ measurable and $m\in\argmin_{q\in\mc Q} F(q)$.
		\index[inot]{$o$}
		\index[inot]{$R_n$}
		\index[inot]{$\mc Q_n$}
		\index[inot]{$m^{Q_n}_n$}
	\assitem{ass:se}{Small Entropy}
		There are $\beta, c_{\ms e} > 0$ such that for $\delta > 0$ large enough
		\begin{equation*}
			\sqrt{\log N(\ball_\delta(o, \mf b), \mf b, r)} \leq c_{\ms e} \log\brOf{\frac{\delta}{r}}^\beta
		\end{equation*}
		for all $r > 0$.
		\index[inot]{$\beta$}
	\end{enumerate}
\end{assumptions}
Note that the \assu{ass:se}{Small Entropy} condition is much stronger than \assu{ass:ent}{Entropy}, which we assumed in \autoref{thm:abstr_rate_prob}. In Euclidean subspaces $\mc Q \subset \R^b$, we have \begin{equation*}
N(r,\Ball{\delta}{0}{d},d) \leq \br{\frac{3 \delta}{r}}^b
\end{equation*} for all $R>r>0$ \cite[section 4]{pollard90}. Thus, \assu{ass:se}{Small Entropy} is fulfilled in Euclidean spaces.
\begin{theorem}[Convergence rate in expectation]\label{thm:abstr_weak_strong}
	In the \textit{Abstract Setting} of section \ref{ssec:ares:generalsetting} with loss $\mf l = \mf b$, where $\mf b$ is a pseudo-metric, and rate parameter $\xi=1$, assume that following conditions hold:
	\assu{ass:ex2}{Existence'}, \assu{ass:gr}{Growth} with $\gamma>1$, \assu{ass:wquad}{Weak Quadruple}, \assu{ass:smom}{Strong Moment} with $\kappa > \gamma-1$, \assu{ass:se}{Small Entropy}.
	Then
	\begin{equation*}
		\Ex*{\mf b(m, m^{Q_n}_n)^\kappa} = \bigO\brOf{\br{n^{-\frac12} \log(n)^\beta}^{\frac{\kappa}{\gamma-1}}}
		\eqfs
	\end{equation*}
\end{theorem}
See appendix \autoref{app:proofs:1and2} for the proof.
\section{Application of the Abstract Results}\label{sec:applications}
We apply the abstract results of Theorems 1 to 4 in this section. We first consider two toy examples -- Euclidean spaces, section \ref{ssec:app:euclid} and infinite dimensional Hilbert spaces, section \ref{ssec:app:hilbert} --  to better understand the result and compare them to optimal bounds. Then we discuss two more involved settings: The Fréchet mean for non-convex subsets of Euclidean spaces, section \ref{ssec:app:nonconvex}, and for Hadamard spaces, section \ref{ssec:app:hadamard}.
\subsection{Euclidean Spaces}\label{ssec:app:euclid}
	Let $\mc Q \subset \R^b$ be convex with the Euclidean metric $d(p,q) = \normof{p-q}$.
	Choose $\mc Y = \mc Q$, $\mf c =  d^2$, $\mf l = d$, $\xi=1$.
	Let $Y$ be a $\mc Q$-valued random variable with $\Ex{\normof{Y}^2}<\infty$.
	Then the Fréchet mean equals the expectation $m =\Ex{Y} \in \mc Q$. We can easily calculate
	\begin{equation*}
		\Ex*{\normof{Y-q}^2-\normof{Y-m}^2} = \normof{q-m}^2\eqfs
	\end{equation*}
	Thus, the \assu{ass:gr}{Growth} condition is fulfilled with $\gamma=2$.
	The space has the strong quadruple inequality at every point with data distance $\mf a(y,z) = 2 \normof{y-z}$ and strong quadruple distance $\mf b_m(p,q) = \normOf{\frac{q-m}{\normof{q-m}}-\frac{p-m}{\normof{p-m}}}$, see section \ref{sssec:inner_product_space}. Thus, \autoref{thm:abstr_rate_exp} implies
	\begin{align*}
		\Ex*{\normOf{\Ex{Y}-\frac1n\sum_{i=1}^nY_i}^2} 
		&=
		\Ex*{\mf l(m,m_n)^2} 
		\\&\leq 
		C
		n^{-1} 
		\entrn(\mc Q, \mf b_m)^2
		\Ex*{\mf a(Y\pr, Y)^2}
		\\&\leq
		C\pr b \frac1n \Ex*{\normOf{Y-\Ex{Y}}^2}
		\eqcm
	\end{align*}
	where we used $N(r,\Ball{R}{0}{d},d) \leq \br{\frac{3 R}{r}}^b$ for all $R>r>0$ \cite[section 4]{pollard90} to fulfill \assu{ass:sent}{Strong Entropy}. The constants $C, C\pr > 0$ are universal.
	Compare this with the result that one obtains by direct calculations, i.e.,
	\begin{equation*}
		\Ex*{\normOf{\Ex{Y}-\frac1n\sum_{i=1}^nY_i}^2}  = \frac1n\Ex*{\normOf{Y-\Ex{Y}}^2}
		\eqfs
	\end{equation*}
	We pay an extra dimension factor $b$ when using the Fréchet mean approach instead of direct calculations. This comes from the use of the Cauchy--Schwartz inequality, which powers the strong quadruple inequality in Euclidean spaces.
\subsection{Hilbert Spaces}\label{ssec:app:hilbert}
	Let $\mb H$ be an infinite dimensional Hilbert space and $\mc Q \subset \mb H$ convex.
	Let $d(p,q)^2 = \normof{p-q}^2 = \ip{p-q}{p-q}$.
	Choose $\mc Y = \mc Q$, $\mf c =  d^2$, $\mf l = d$, $\xi=1$.
	Let $Y$ be a $\mc Q$-valued random variable with $\Ex{\normof{Y}^2}<\infty$.
	As in the Euclidean case, the Fréchet mean $m$ equals the expectation $\Ex{Y}$, the \assu{ass:gr}{Growth} condition holds with $\gamma=2$, and the strong quadruple inequality is fulfilled with $\mf a(y,z) = 2 \normof{y-z}$ and pseudometric $\mf b_m(p,q) = \normOf{\frac{q-m}{\normof{q-m}}-\frac{p-m}{\normof{p-m}}}$.
	
	Unfortunately, \assu{ass:sent}{Strong Entropy} is not fulfilled on $\mb H$ if $\dim(\mb H) = \infty$. 
	By introducing a weight sequence, we can make $\mf b_m$ smaller by making $\mf a$ larger: 
	Assume that the Hilbert space $\mb H$ is separable and thus admits a countable basis. Let $s=(s_k)_{k\in\N} \subset(0,\infty)$.
	In section \ref{sssec:inner_product_space}, we derived that the strong quadruple condition holds with
	$\mf a(y,z) = 2\normof{y-z}_{s^{-1}}$ and $\mf b_m^s (p,q)=  \normOf{\frac{q-m}{\normof{q-m}}-\frac{p-m}{\normof{p-m}}}_{s}$.
	Then $\entrn(\mb H, \mf b_m^s) \leq \gamma_2(\mb H, \mf b_m^s) \leq \gamma_2(\mc E_s, d)$, where 
	\begin{equation*}
		\mc E_s = \cb{h\in\mb H \colon \sum_{k=1}^\infty \frac{h_k^2}{s_k^2}\leq 1}
		\eqfs
	\end{equation*}
	There is a universal constant $c>0$ such that $\gamma_2(\mc E_s, d)^2 \leq c \sum_{k=1}^\infty s_k^2$, see \cite[Proposition 2.5.1]{talagrand14}.
	As a condition on the variance term, we need
	\begin{equation*}
		\Ex*{\normof{Y-\Ex{Y}}_{s^{-1}}^2} = \normof{\sigma}_{s^{-1}}^2 = \sum_{k=1}^\infty \sigma^2_k s_k^{-2}< \infty
		\eqcm
	\end{equation*}
	where $\sigma^2_k := \VOf{Y_k}$ and $\sigma = (\sigma_k)_{k\in\N}$.
	Similar to the Euclidean case, \autoref{thm:abstr_rate_exp} implies
	\begin{equation*}
		\Ex*{\normOf{\Ex{Y}-\frac1n\sum_{i=1}^nY_i}^2} 
		=
		\Ex*{\mf l(m,m_n)^2} 
		\leq
		C \frac1n \normof{s}_{\ell_2}^2 \normof{\sigma}_{s^{-1}}^2
		\eqcm
	\end{equation*}
	where $\normof{s}_{\ell_2}^2 = \sum_{k=1}^\infty s_k^2$.
	
	Direct calculations yield a better result:
	\begin{equation*}
		\Ex*{\normOf{\Ex{Y}-\frac1n\sum_{i=1}^nY_i}^2} 
		=
		\frac1n \normof{\sigma}_{\ell_2}^2
		\eqfs
	\end{equation*}
	As in the Euclidean case, we pay a factor related to the dimension for using the more generally applicable Fréchet mean approach instead of using the inner product for direct calculations.
\subsection{Non-Convex Subsets}\label{ssec:app:nonconvex}
	Assume we are in the setting of section \ref{ssec:app:hilbert} and the mentioned conditions for convergence are fulfilled.
	But now we want to take $\mc Q \subset \mb H$ not necessarily convex and $\mc Y = \mb H$. Assume that \assu{ass:ex}{Existence} of the Fréchet mean $m\in\mc Q$ is fulfilled. The expectation $\mu := \Ex{Y} \in \mb H$ might not be an element of $Q$. Then the Fréchet mean $m$ is the closest projection of $\mu$ to $\mc Q$, in the sense that
	\begin{align*}
		\argmin_{q\in\mc Q} \Ex{\normof{Y-q}^2} 
		= 
		\argmin_{q\in\mc Q} \normof{\mu-q}
		\eqfs
	\end{align*}
	To get the same rate as in section \ref{ssec:app:hilbert}, we mainly need to be concerned with the \assu{ass:gr}{Growth} condition, as the quadruple condition holds in all subsets.
	For $q\in\mb H$, simple calculations show
	\begin{align*}
		\Ex*{\normof{Y-q}^2-\normof{Y-m}^2} 
		=
		\normof{\mu-q}^2-\normof{\mu-m}^2
		\eqfs
	\end{align*}
	We want to find a lower bound of this term in the form of $c_{\ms g} \normof{q-m}^\gamma$ for constants $\gamma, c_{\ms g} > 0$.
	For $a>1$, we have
	\begin{align*}
		&a\normOf{\mu-q}^2-a\normOf{\mu-m}^2 - \normOf{q-m}^2 
		\\&= (a-1)\normOf{q-\br{\mu+\frac{\mu-m}{a-1}}}^2 - \frac{a^2}{a-1} \normOf{m-\mu}^2
		\eqfs
	\end{align*}
	Thus, $\normOf{\mu-q}^2-\normOf{\mu-m}^2 \geq \frac1a\normOf{q-m}^2$ if and only if 
	\begin{equation*}
		\normOf{q-\br{\mu+\frac{\mu-m}{a-1}}} \geq \frac{a}{a-1} \normOf{\mu-m}
		\eqfs
	\end{equation*}
	Equivalently, the \assu{ass:gr}{Growth} condition holds with $\gamma=2$ and $c_{\ms g} \in (0,1)$ if and only if
	\begin{equation*}
		\normOf{q-\br{\mu+\frac{c_{\ms g}}{1-c_{\ms g}} \br{\mu-m}}} \geq \frac{1}{1-c_{\ms g}} \normOf{\mu-m}
	\end{equation*}
	for all $q\in\mc Q$, i.e., if and only if $\mc Q \cap \ball_{r}(p) = \emptyset$, where $r= \frac{1}{1-c_{\ms g}} \normOf{\mu-m}$ and $p=\mu+\frac{1-c_{\ms g}}{c_{\ms g}} \br{\mu-m}$.  Note that $\normof{p-m} = r$.
	\begin{figure}
	\begin{center}
		\begin{tikzpicture}
			\fill[green!20!white] (-7,-4) rectangle (4,4);
			\draw[fill=red!5!white] (-2,0) circle (3);
			\draw[fill=red!12!white] (-1,0) circle (2);
			\draw[fill=red!25!white] (0,0) circle (1);
			\filldraw (0,0) circle (0.05) node[below] {$\mu = p_0$};
			\filldraw (1,0) circle (0.05) node[right] {$m$};
			\filldraw (-1,0) circle (0.05) node[left] {$p_1$};
			\filldraw (-2,0) circle (0.05) node[left] {$p_2$};
			\draw[thick,dashed] (0,0) -- node[left]{$r_0$} (0,1);
			\draw[thick,dashed] (-1,0) -- node[left]{$r_1$} (-1,2);
			\draw[thick,dashed] (-2,0) -- node[near end, left]{$r_2$} (-2,3);
			\node at (1.5,1.5) {$\mc Q$};
		\end{tikzpicture}
	\end{center}
	\caption{If $\mc Q \cap \ball_{r_0}(p_0) \neq \emptyset$, the point $m$ cannot be the Fréchet mean of a distribution with expectation $\mu$. To fulfill the 
	\texttt{Growth}
	condition, we need $\mc Q \cap \ball_{r_1}(p_1) = \emptyset$ for a ball with larger radius $r_1 > r_0$ and adjusted center $p_1$. Increasing the radius further, $r_2>r_1$, only improves the constant $c_{\ms g}$ of the 
	\texttt{Growth}
	condition, but not the exponent $\gamma$.
	}\label{fig:nonconvsubset}
	\end{figure}
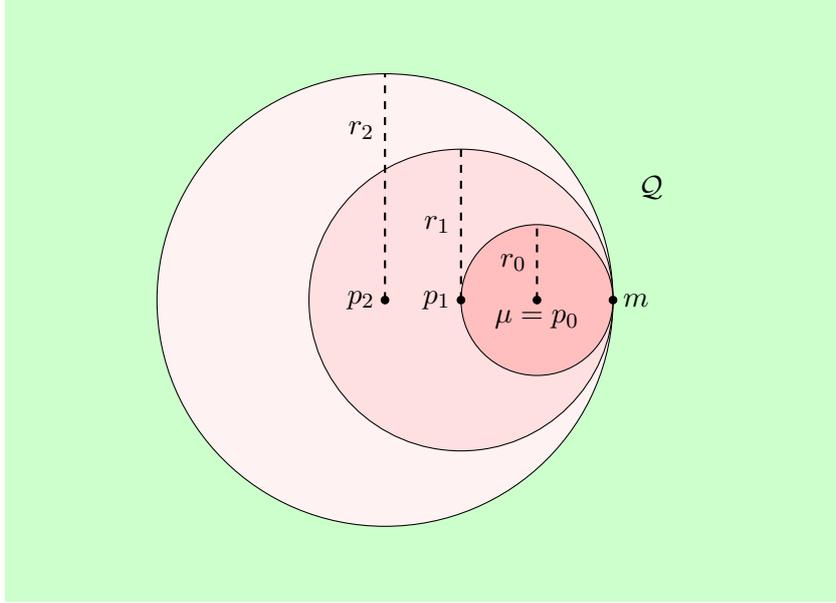
	This is illustrated in \autoref{fig:nonconvsubset}.
	We have answered the question of how $\mc Q$ may look like, given the location of $\mu$ and $m$. Possibly more interesting is the question of, given $\mc Q$, where may $\mu$ be located so that $m$ can be estimated with the same rate as for convex sets. We will answer this question only informally via a description similar to a \textit{medial axis transform} \cite{choi97}:
	
	For simplicity assume $\mc Q = \R^2 \setminus A$, where $A$ is a nonempty, open, and simply connected set with border $\partial A$ that is parameterized by the continuous function $\gamma\colon[0,1] \to \partial A$. Roll a ball along the border on the inside of $A$. Make the ball as large as possible at any point so that it is fully contained in $A$ and touches the border at point $\gamma(t)$. Denote the center of the ball as $c \colon [0,1]\to A$ and the radius as $r \colon [0,1] \to [0,\infty)$. Take $\epsilon\in(0,1)$ and trace the point $p_\epsilon\colon[0,1]\to A$ on the radius connecting the center of the ball $c(t)$ and the border $\gamma(t)$ such that it divides the radius into two pieces of length $\ol {p_\epsilon(t)}{c(t)} = \epsilon r(t)$ and $\ol {p_\epsilon(t)}{\gamma(t)} =(1-\epsilon) r(t)$. If $\mu$ lies on the outside of the set prescribed by $p\colon[0,1]\to A$, it can be estimated with the same rate as for convex sets.
	This is illustrated in \autoref{fig:nonconvxheart}.
	\begin{figure}
		\begin{center}
			\includegraphics[width=0.8\textwidth]{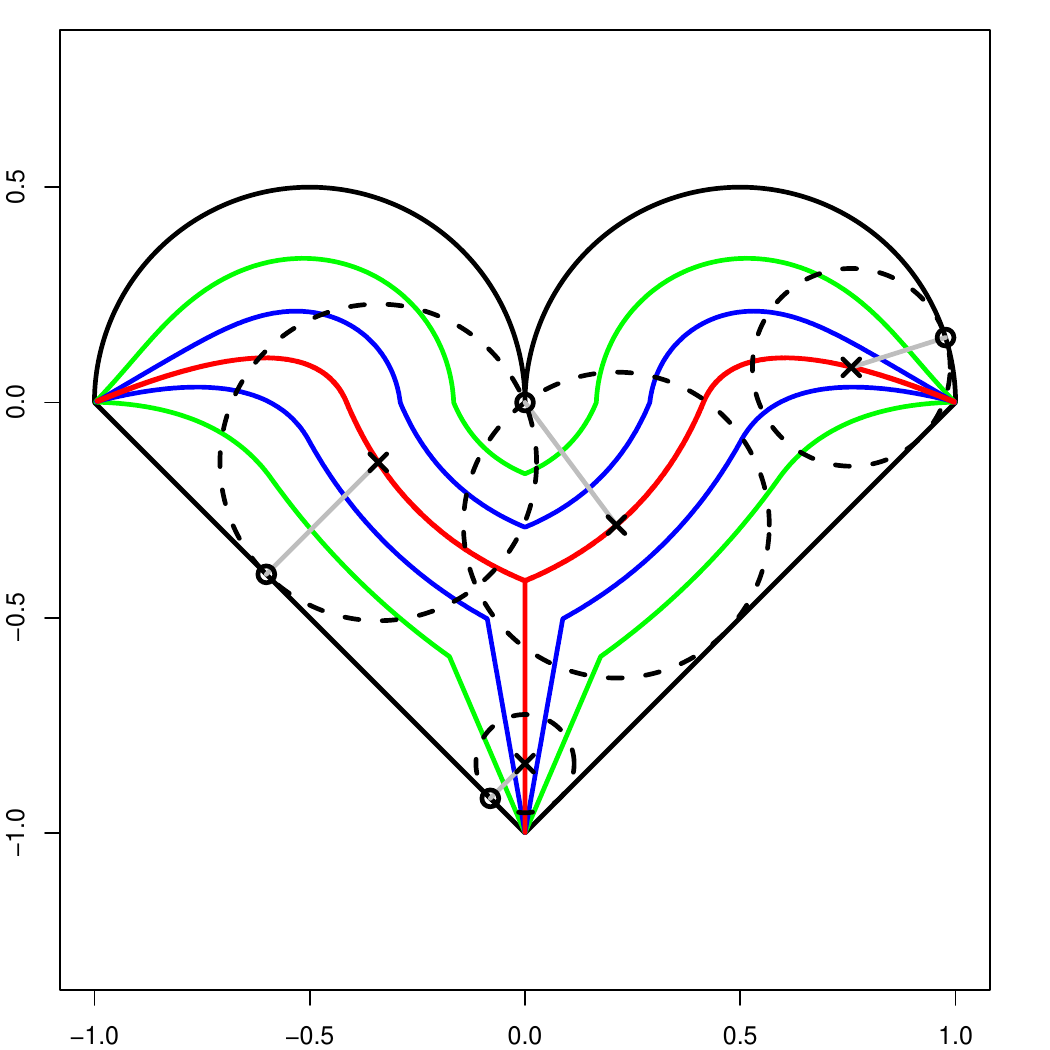}
		\end{center}
		\caption{Let $A\subset\R^2$ be the set enclosed by the heart (solid black lines). Let $\mc Y = \R^2$ and $\mc Q = \R^2 \setminus A$. We consider a distribution on $\R^2$ with mean $\mu\in\mc Y$ and Fréchet mean $m\in\mc Q$ with respect to the Euclidean metric and the descriptor space $\mc Q$. The green, blue, and red lines show $p_\epsilon(t)$ for $\epsilon = 0.6, 0.3, 0$.}\label{fig:nonconvxheart}
	\end{figure}
	The set of all centers $\mc C := \cb{c(t) \,|\, t \in [0,1]}$, also called the \textit{medial axis} ot \textit{cut locus}, is critical: The closer $\mu$ is to $\mc C$, the larger the guaranteed error bound for the estimator. In particular, we cannot guarantee consistency of  the estimator if $\mu \in \mc C$. A very similar phenomenon is described in \cite[section 3]{bhattacharya03}. The authors consider a Riemannian manifold $\mc Q$ that is embedded in an Euclidean space $\mc Y$. The \textit{extrinsic mean} of a distribution on $\mc Q$ is the projection of the mean $\mu$ in $\mc Y$ to $\mc Q$. The points $\mc C$ are called \textit{focal points}. It is shown \cite[Theorem 3.3]{bhattacharya03} that in many cases the \textit{intrinsic mean}, i.e, the Fréchet mean in $\mc Q$ with respect to the Riemannian metric on $\mc Q$, is equal to the extrinsic mean, i.e, the Fréchet mean in $\mc Q$ with respect to the Euclidean metric on $\mc Y$.
	
	The conditions described above are connected to the term reach of a set \cite{federer59}. The reach of $\mc Q \subset \R^b$ is the largest $\epsilon>0$ (possibly $\infty$) such that $\inf_{q\in\mc Q} d(x, q) < \epsilon$ implies that $x \in \R^b$ has a unique projection to $\mc Q$, i.e., a unique point $x_\mc Q$ with $d(x, x_{\mc Q}) = \inf_{q\in\mc Q} d(x, q)$. If the distance of the mean $\mu$ to $\mc Q$ is less than the reach of $\mc Q$, then the \assu{ass:gr}{Growth} condition holds with $\gamma = 2$. Thus, the rate of convergence is upper bounded by $c n^{-\frac12}$ for some $c>0$. Note that convex sets have infinite reach and exhibit this upper bound for any distribution with finite second moment.

	By considering the growth condition $\normOf{\mu-q}^2-\normOf{\mu-m}^2 \geq c_{\ms g}\normOf{q-m}^\gamma$, one can also find examples of subspaces where the growth exponent for specific distributions is different from 2.
\subsection{Hadamard Spaces}\label{ssec:app:hadamard}
Let $(\mc Q, d)$ be a Hadamard space. A definition of Hadamard spaces is given in section \ref{ssse:quad:npc}. Use the notation $\ol yq = d(y,q)$. For our purposes the most notable property of Hadamard spaces is that they fulfill the nice quadruple property, i.e., $\ol yq^2 - \ol yp^2 - \ol zq^2 + \ol zp^2 \leq 2 \,\ol yz\, \ol qp$. In the following subsections, we will see how this translates to convergence rates for the Fréchet mean estimator and use the power inequality to obtain results for a generalized Fréchet mean with cost function $d^{2a}$ for $a\in[\frac12, 1]$.

For an introduction to Hadamard spaces see \cite{bacak14}. A survey of recent developments can be found in \cite{bacak18}. 
In \cite{berg08} the authors characterize Hadamard spaces by the nice quadruple inequality and discuss a quasilinearzation of these spaces by observing that the left hand side of the nice quadruple inequality behaves like an inner product to some extend. \cite{sturm03} shows how some important theorems of probability theory in Euclidean spaces, like the law of large numbers and Jensen's inequality, translate to non-Euclidean Hadamard spaces. In \cite{sturm02} martingale theory on Hadamard spaces is discussed.

Turning to more applied topics, \cite{bacak14b} shows algorithms for calculating the Fréchet mean in Hadamard spaces with cost function $d^{2a}$ for $a=\frac12$ and $a=1$.
An important application of Hadamard spaces in Bioinformatics are phylogenetic trees \cite{billera01}. See also \cite[section 6.3]{bacak18} for a quick overview. Another application of Hadamard spaces is taking means in the manifold of positive definite matrices, e.g., in diffusion tensor imaging. But note that, as the underlying space is a differentiable manifold, one an use gradient-based approaches, see \cite{pennec06}.

Further examples of Hadamard spaces include Hilbert spaces, the Poincaré disc, complete metric trees, complete simply-connected Riemannian manifolds of nonpositive sectional curvature. See also \cite[section 3]{sturm03}.
\subsubsection{Fréchet Mean}\label{sssec:app:hadamard:fm}
Let $(\mc Q, d)$ be a Hadamard space. We use $\mc Q$ as data space as well as descriptor space, i.e., $\mc Q = \mc Y$. The cost function is $\mf c=d^2$, the loss $\mf l = d$. As described in section  \ref{ssse:quad:npc} the weak quadruple inequality holds with $\mf a = 2d$ and $\mf b = d$, i.e., $(\mc Q,d)$ fulfills the nice quadruple inequality.
Let $Y$ be a random variable with values in $\mc Q$. Let $Y_1, \dots, Y_n$ be iid copies of $Y$.

If $\Ex{d(Y,q)^2} < \infty$ for one $q\in\mc Q$, then it is also finite for every $q\in\mc Q$ and the Fréchet mean $m\in\argmin_{q\in\mc Q} \Ex{d(Y,q)^2}$ exists and is unique, see \cite[Proposition 4.3]{sturm03}. The same holds true for the estimator $m_n\in\argmin_{q\in\mc Q} \sum_{i=1}^n{d(Y_i,q)^2}$. Thus, \assu{ass:ex}{Existence} is fulfilled.

Here, we chose a second moment condition, because we will need it for estimation anyway. But note that choosing the cost function as $\mf c(y,q) = d(y,q)^2 - d(y,o)^2$ for a fixed, arbitrary point $o\in\mc Q$ allows us to require only a finite first moment for \assu{ass:ex}{Existence} and the resulting Fréchet mean coincides with the $d^2$-Fréchet mean if the second moment is finite. This is described in more detail and utilized in \cite{sturm03}.

Furthermore, the \assu{ass:gr}{Growth}-condition holds in Hadamard spaces with $\gamma=2$ and $c_{\ms g} = 1$, see \cite[Proposition 4.4]{sturm03}.
Thus, we obtain following corollary of \autoref{thm:abstr_rate_prob}.
\begin{corollary}[Convergence rate in probability]
	Assume \assu{ass:mom}{Moment} with $\zeta = 2$ and $\mf a = 2d$ and \assu{ass:ent}{Entropy}  with $\mf b = d$ and $\alpha=\beta$.
	Define
	\begin{equation*}
		\eta_{\beta,n} := 
		\begin{cases} 
			n^{-\frac12} & \text{ for }\beta < 1\eqcm\\
			n^{-\frac12} \log(n+1) & \text{ for } \beta = 1\eqcm\\
			n^{-\frac1{2\beta}} & \text{ for } \beta > 1\eqfs
		\end{cases} 
	\end{equation*}
	Then, for all $s > 0$, we have
	\begin{equation*}
		\PrOf{\eta_{\beta,n}^{-1} d(m,m_n) \geq s} \leq c \,\Ex{d(Y,Y\pr)^2} \,s^{-2}
		\eqcm
	\end{equation*}
	with a constant $c > 0$ depending only on $\beta$ and $c_{\ms e}$.
	In particular,
	\begin{equation*}
		d(m, m_n) = \bigOp\brOf{\eta_{\beta,n}}
		\eqfs
	\end{equation*}
\end{corollary}
As described in section \ref{ssse:quad:npc}, it may be difficult to find a version of the strong quadruple inequality such that the same rate can be derived for convergence in expectation. Thus, instead of trying to apply \autoref{thm:abstr_rate_exp}, we utilize (i) \autoref{cor:probtoexpec} and (ii) \autoref{thm:abstr_weak_strong}, respectively.
\begin{corollary}
\theoremContentInNewLine
	\begin{enumerate}[label=\environmentEnumerateLabel]
	\item 
		Let $\epsilon > 0$.
		Assume $\Ex*{d(Y,Y\pr)^{2+\epsilon}} < \infty$.
		Assume \assu{ass:ent}{Entropy}  with $\mf b = d$ and $\alpha=\beta < 1$.
		Then we have
		\begin{equation*}
			\Ex*{d(m, m_n)^2} \leq c\, \Ex*{d(Y,Y\pr)^{2+\epsilon}}^{\frac{2}{2+\epsilon}} \,\frac1{\epsilon n }
		\end{equation*}
		for a constant $c>0$ depending only on $\beta$.
	\item 
		Assume $\Ex*{d(Y,Y\pr)^2} < \infty$.
		Let $o\in\mc Q$. Assume \assu{ass:se}{Small Entropy} with $\mf b = d$.
		Let $\tilde m_n \in \argmin_{q\in\ball_n(o)} \sum_{i=1}^n d(Y_i, q)^2$.
		Then we have
		\begin{equation*}
			\Ex*{d(m, \tilde m_n)^2} = \bigO\brOf{\frac1n \log(n)^{2\beta}}
			\eqfs
		\end{equation*}
	\end{enumerate}
\end{corollary}
\subsubsection{Power Fréchet Mean}\label{sssec:app:hadamard:pfm}
We go beyond Hadamard spaces by utilizing the power inequality, \autoref{thm:power_inequ}. 
Let $(\mc Q, d)$ is a Hadamard space and $a \in[\frac12,1)$. Then $(\mc Q, d^a)$ is not Hadamard, but fulfills a weak quadruple inequality: Fix an arbitrary point $o\in\mc Q$. We use the cost function $\mf c(y,q) = d^{2a}(y,q)-d^{2a}(y,o)$ and the loss $\mf l = d$. Then the weak quadruple inequality holds with $\mf a (y,z) = 8a2^{-2a} d(y,z)^{2a-1}$ and $\mf b = d$.

We need to choose the cost function $d^{2a}(y,q)-d^{2a}(y,o)$ instead of $d^{2a}(y,q)$ to obtain a result with minimal moment requirement. To fulfill \assu{ass:mom}{Moment} we need $\Ex{d(Y,Y\pr)^{2(2a-1)}} < \infty$ and for \assu{ass:ex}{Existence}, we need $\Ex{\abs{\mf c(Y,q)}} < \infty$. We fulfill both by assuming that
$\Ex{d(Y,o)^{2(2a-1)}} < \infty$. Then the both conditions are satisfied: On one hand
$\Ex{d(Y,Y\pr)^{2(2a-1)}} \leq 2 \Ex{d(Y,o)^{2(2a-1)}}$. On the other hand, using 
the tight power bound of \autoref{lmm:tightpowerbound} (appendix \autoref{app:power_inequality}),
\begin{align*}
	\ol yq^{2a}-\ol yo^{2a}
	\leq 
	2 a \,\ol qo \br{\frac{\ol yq + \ol yo}{2}}^{2a-1}
\end{align*} 
and thus 
\begin{equation*}
	|\mf c(Y,q)|  \leq 2a \,\ol qo \br{\frac{\ol qo}{2} +  \ol Yo}^{2a-1}\eqcm
\end{equation*}
which implies 
$\Ex{\abs{\mf c(Y, q)}} < \infty$. But $\Ex{d(Y,q)^{2a}}$ might be infinite as $2a > 2(2a-1)$.

\autoref{thm:abstr_rate_prob} with $\zeta=2$ implies following corollary.
\begin{corollary}[Rates in probability for power mean]\label{coro:probrates_power}
	Assume:
	\begin{enumerate}[label=\environmentEnumerateLabel]
	\item
		\assu{ass:ex}{Existence}:	Let $a\in[\frac12,1]$. Let $o\in\mc Q$ be an arbitrary fixed point. Assume there are $m_n\in\argmin_{q\in\mc Q} \frac1n\sum_{i=1}^n \br{\ol{Y_i}{q}^{2a}-\ol {Y_i}o^{2a}}$ measurable and $m\in\argmin_{q\in\mc Q} \Ex{\ol Yq^{2a}-\ol Yo^{2a}}$.
	\item 
		\assu{ass:gr}{Growth}: There are constants $c_{\ms g} > 0, \gamma\in(1,\infty)$ such that 
				$\Ex{\ol Yq^{2a}}-\Ex{\ol Ym^{2a}}  \geq c_{\ms g} d(m,q)^\gamma$ for all $q\in\mc Q$.
	\item \assu{ass:mom}{Moment}:   $\Ex{\ol{Y}{q}^{2(2a-1)}} < \infty$ for one (and thus for all) $q\in\mc Q$.
	\item \assu{ass:ent}{Entropy}: 
		There is $\beta > 0$ such that
		\begin{equation*}
			\sqrt{\log N(\ball_\delta(m, d), d, r)} \leq c_{\ms e} \br{\frac{\delta}{r}}^\beta
		\end{equation*}
		for all $\delta, r > 0$.
	\end{enumerate}
	Then, for all $s > 0$, we have
	\begin{equation*}
		\PrOf{\eta_{\beta,n}^{-\frac1{\gamma-1}}\, \ol{m}{m_n} \geq s} \leq c \,\Ex{\ol{Y}{o}^{2(2a-1)}} \, s^{-2(\gamma-1)}
		\eqcm
	\end{equation*}
	where $c > 0$ depends only on $\beta, \gamma, c_{\ms e}$.
	In particular,
	\begin{equation*}
		d(m, m_n) = \bigOp\brOf{\eta_{\beta,n}^{-\frac1{\gamma-1}}}
		\eqfs
	\end{equation*}
\end{corollary}
Note that the moment condition becomes weaker as $a$ gets smaller and vanishes for $a=\frac12$, where, in the Euclidean case, the Fréchet mean is the median. 

\assu{ass:ex}{Existence} of $m_n$ and $m$ is a purely technical condition, as one will usually only be able to minimize the objective functions up to an $\epsilon>0$ and the set of $\epsilon$-minimizers is always nonempty.

The \assu{ass:gr}{Growth} condition is more interesting. It seems possible to choose $\gamma=2$ for all $a\in[\frac12,1]$ in many circumstances, at least under some conditions on the distribution of $Y$. But precises statements of this sort are unknown to the author. If $\gamma$ really can be chosen independently of $a$, then the rate is the same for all $a\in[\frac12,1]$. In the Euclidean case, this is manifested in the fact that we can estimate median ($a = \frac12$) and mean ($a = 1$) and all statistics ``in between" ($a \in (\frac12,1)$) with the same rate (under some conditions), but with less restrictive moment assumptions for smaller powers $a$.

Similarly to the corollary above, we can apply \autoref{cor:probtoexpec} or \autoref{thm:abstr_weak_strong} to obtain rates in expectation.
\section{Further Research}
The growth condition, especially for power Fréchet means, see section \ref{sssec:app:hadamard:pfm}, needs to be studied further to get a better understanding of what properties a distribution must have, so that all power means can be estimated with the same rate.

In \cite{bacak14} the author describes algorithms for calculating means and medians in Hadamard spaces, i.e., power Fréchet means as in section \ref{sssec:app:hadamard:pfm} with $a\in\{\frac12, 1\}$. As we have shown results also for $a \in (\frac12,1)$, it would be interesting to see, whether one can generalize the algorithms for $a = \frac12$ and $a=1$ to $a\in[\frac12, 1]$.

We plan to use the results of this paper in a regression setting similar to \cite{petersen16}. We will show convergence rates for an orthogonal series-type regression estimator for the conditional Fréchet mean $m(x) \in \argmin_{q\in\mc Q}\Ex*{\mf c(Y,q) | X=x}$.

\begin{appendices}
\section{Proofs of Theorem 1, 2, and 4}\label{app:proofs:1and2}
\subsection{Proof of \autoref{thm:abstr_rate_prob} and \autoref{cor:probtoexpec}}
Define
\begin{equation*}
	\Delta_n(\delta) := \sup_{q\in\mc Q\colon \mf l(m,q)\leq \delta} F(q)-F(m)-F_n(q)+F_n(m)
	\eqfs
\end{equation*}
\index[inot]{$\Delta_n$}
Results similar to following Lemma are well known in the M-estimation literature. The proof relies on the \textit{peeling device}, see \cite{geer00}.
\begin{lemma}[Weak argmin transform]\label{lmm:weak:argmin}	
	Assume \assu{ass:gr}{Growth}. Let $\zeta \geq 1$. Assume that there are constants $\xi\in(0,\gamma)$, $h_n\geq 0$ such that
	$\Ex{\Delta_n(\delta)^\zeta} \leq \br{h_n \delta^\xi}^\zeta$	for all $\delta>0$. Then
	\begin{equation*}
		\PrOf{\mf l(m,m_n) \geq s}
		\leq
		c \br{h_n s^{-(\gamma-\xi)}}^\zeta
		\eqcm
	\end{equation*}
	where $c > 0$ depends only on $c_{\ms g}, \gamma, \xi, \zeta$.
\end{lemma}
\begin{proof}
	Let $0 < a < b$.
	If $\mf l(m, m_n) \in [a,b]$, we have
	\begin{equation*}
		c_{\ms g} a^\gamma 
		\leq 
		c_{\ms g} \mf l(m,m_n)^\gamma
		\leq 
		F(m_n)-F(m)
		\leq 
		F(m_n)-F(m)-F_n(m_n)+F_n(m)
		\leq 
		\Delta_n(b)
		\eqfs
	\end{equation*}
	Let $s>0$. 
	For $k\in \N_0$, set $a_k := s 2^{k}$. We have
	\begin{align*}
		&\PrOf{\mf l(m,m_n) \geq s}
		\\&\leq 
		\sum_{k=0}^{\infty}\PrOf{\mf l(m, m_n) \in [a_k, a_{k+1}]}
		\\&\leq 
		\sum_{k=0}^{\infty} \PrOf{c_{\ms g} a_k^\gamma \leq \Delta_n(a_{k+1})}
		\eqfs
	\end{align*}
	We use Markov's inequality and the bound on $\Ex{\Delta_n(\delta)^\zeta}$ to obtain
	\begin{align*}
		\PrOf{c_{\ms g} a_k^\gamma \leq \Delta_n(a_{k+1})}
		&\leq
		\frac{\EOf{ \Delta_n(a_{k+1})^\zeta}}{\br{c_{\ms g} a_k^\gamma}^\zeta} 
		\\&\leq
		\br{\frac{h_n a_{k+1}^\xi}{c_{\ms g} a_k^\gamma}}^\zeta
		\\&=
		\br{2 c_{\ms g}^{-1} h_n s^{-(\gamma-\xi)} 2^{-k(\gamma-\xi)}}^\zeta
		\eqfs
	\end{align*}
	As $\gamma-\xi>0$, we get 
	\begin{align*}
		\PrOf{\mf l(m,m_n) \geq s}
		&\leq
		\br{2 c_{\ms g}^{-1} h_n s^{-(\gamma-\xi)}}^\zeta \sum_{k=0}^\infty 2^{-k\zeta(\gamma-\xi)}
		\\&=
		\br{2 c_{\ms g}^{-1} h_n s^{-(\gamma-\xi)}}^\zeta \frac{1}{1-2^{-\zeta(\gamma-\xi)}}
		\eqfs
	\end{align*}
\end{proof}
\begin{lemma}\label{lmm:weak:close}	
	Let $\zeta \geq 1$. Assume \assu{ass:mom}{Moment}, \assu{ass:wquad}{Weak Quadruple}, and \assu{ass:ent}{Entropy}.
	Then
	\begin{equation*}
		\Ex{\Delta_n(\delta)^\zeta} 
		\leq c\, \mf M(\zeta)\, \br{\delta^{\frac\alpha\beta}\, \eta_{\beta,n}}^\zeta
	\end{equation*}
	where $Y\pr$ is an independent copy of $Y$, $c > 0$ is a constant depending only on $\beta, c_{\ms e}, \zeta$, and
	\begin{equation*}
		\eta_{\beta,n} := 
		\begin{cases} 
			n^{-\frac12} & \text{ for }\beta < 1\eqcm\\
			n^{-\frac12} \log(n+1) & \text{ for } \beta = 1\eqcm\\
			n^{-\frac1{2\beta}} & \text{ for } \beta > 1\eqfs
		\end{cases} 
	\end{equation*}
\end{lemma}
\begin{proof}
	Recall the notation $\olc yq := \mf c(y,q)$, $F(q) = \Ex{\olc Yq}$,  $F_n(q) = \frac1n \sum_{i=1}^n \olc {Y_i}q$.
	Define
	\begin{align*}
		Z_i(q) 
		&:= 
		\frac1n \br{\Ex*{\olc Yq-\olc Ym}- \olc{Y_i}q+\olc{Y_i}m}
		\eqfs
	\end{align*}
	Thus, $\Delta_n(\delta) = \sup_{q\in\ball_\delta(m, \mf l)}\sum_{i=1}^n Z_i(q)$.
	The \assu{ass:mom}{Moment} condition together with the \assu{ass:wquad}{Weak Quadruple} condition imply that $Z_i$ are integrable.
	Let $(Z_1\pr, \dots, Z_n\pr)$ be an independent copy of $(Z_1, \dots, Z_n)$, where $(Y_1\pr, \dots, Y_n\pr)$ is an independent copy of $(Y_1, \dots, Y_n)$.
	By \assu{ass:wquad}{Weak Quadruple}, we have
	\begin{align*}
		&
		n^2\br{Z_i(q)-Z_i(p)-Z_i\pr(q) +Z_i\pr(p)}^2 
		\\&= 
		\br{\olc{Y_i}q-\olc{Y_i}p-\olc{Y_i\pr}q+\olc{Y_i\pr}p}^2
		\\&\leq 
		\mf b(q,p)^2 \mf a(Y_i,Y_i\pr)^2
		\eqfs
	\end{align*}
	Furthermore,
	\begin{equation*}
		\Ex*{\br{\frac1n\sum_{i=1}^n \mf a(Y_i,Y_i\pr)^2}^{\frac\zeta2}} \leq \mf M(\zeta)
		\eqfs
	\end{equation*}
	Thus, \autoref{chaining:empproc} (appendix \autoref{app:chaining}) implies
	\begin{equation*}
		\Ex{\Delta_n(\delta)^\zeta} \leq c_1 \mf M(\zeta) \br{\frac1{\sqrt{n}} \entrn(\ball_\delta(m, \mf l), \mf b)}^\zeta\eqcm
	\end{equation*}
	where $\entrn$ is defined in \autoref{chaining:def_entropy}.

	To bound $\entrn(\ball_\delta(m, \mf l), \mf b)$ by applying \autoref{lmm:chaining:rate} (appendix \autoref{app:chaining}), we need to find an upper bound on $\diam(\ball_\delta(m, \mf l), \mf b)$.
	Set $r_0 := (2 c_{\ms e} \delta^\alpha)^{\frac1\beta}$.
	It fulfills $c_{\ms e} \frac{\delta^\alpha}{r_0^\beta} < \sqrt{\log(2)}$. Thus, \assu{ass:ent}{Entropy} implies $N(\ball_\delta(m, \mf l), \mf b, r_0) < 2$.
	As the covering number is an integer, $N(\ball_\delta(m, \mf l), \mf b, r_0) = 1$, which implies, $\diam(\ball_\delta(m, \mf l), \mf b) \leq 2r_0 =: D_\delta$.
	Rewriting the  \texttt{Entropy}-condition in terms of $D_\delta$ yields
	\begin{equation*}
	\sqrt{\log N(\ball_\delta(m, \mf l), \mf b, r)} \leq c_\beta \br{\frac{D_\delta}{r}}^\beta
	\end{equation*}
	for a constant $c_\beta > 0$ depending only on $\beta$ and $c_{\ms e}$.

	Together with \autoref{lmm:chaining:rate} (appendix \autoref{app:chaining}) we get
	\begin{equation*}
		\Ex{\Delta_n(\delta)^\zeta} \leq c \mf M(\zeta) \br{\delta^{\frac\alpha\beta} \eta_{\beta,n}}^\zeta
	\end{equation*}
	for a constant $c > 0$ depending only on $\beta, c_{\ms e}, \zeta$.
\end{proof}
\begin{proof}[of \autoref{thm:abstr_rate_prob}]
	Combine \autoref{lmm:weak:argmin} and \autoref{lmm:weak:close}.
\end{proof}
\begin{proof}[of \autoref{cor:probtoexpec}]
\autoref{thm:abstr_rate_prob} yields
\begin{align*}
	\eta_{\beta,n}^{-\frac{\kappa}{\gamma-\frac{\alpha}{\beta}}} \Ex{\mf l (m, m_n)^\kappa} 
	&= 
	\int_0^\infty 
	\PrOf{\eta_{\beta,n}^{-\frac{1}{\gamma-\frac{\alpha}{\beta}}} \mf l (m, m_n) \geq t^{\frac1\kappa}} \dl t
	\\&\leq 
	\int_0^\infty \max\brOf{1, c\, \mf M(\zeta)\, t^{-\xi}} \dl t
	\eqfs
\end{align*}
In general for $a > 1$, $b>0$, we have
\begin{equation*}
	\int \max(1, bt^{-a}) \dl t = \frac{a}{a-1} b^{\frac1a}
	\eqfs
\end{equation*}
The proof is concluded by applying this statement and noting that $\xi > 1$.
\end{proof}
\subsection{Proof of \autoref{thm:abstr_rate_exp}}
To state the next Lemma, which will be used to prove \autoref{thm:abstr_rate_exp}, we introduce an intermediate condition, which we call \assu{ass:clo}{Closeness}.
\begin{assumptions}
\theoremContentInNewLine
\begin{enumerate}[label=\environmentEnumerateLabel]
\assitem{ass:clo}{Closeness}
	There is $\xi \in (0, \gamma)$ and a random variable $H_n \geq 0$, such that
	\begin{equation}
		F(q) - F(m) - F_n(q) + F_n(m) \leq H_n \mf l(m, q)^\xi
	\end{equation}
	for all $q \in \mc Q$ almost surely.
\end{enumerate}
\end{assumptions}
\begin{lemma}\label{lmm:strong:growth}
	Assume \assu{ass:clo}{Closeness} and \assu{ass:gr}{Growth}, and let $\kappa > 0$. Then, 
	\begin{equation*}
		\Ex*{\mf l(m,m_n)^{\kappa}} \leq c \Ex*{H_n^{\frac{\kappa}{\gamma-\xi}}}\eqcm
	\end{equation*}
	where $c >0$ depends only on $c_{\ms g}, \gamma, \xi, \kappa$.
\end{lemma}
\begin{proof}
	We use \assu{ass:gr}{Growth} and the fact that $m_n$ minimizes $F_n$ to obtain
	\begin{align*}
		c_{\ms g}\mf l(m,m_n)^\gamma 
		&\leq
		F(m_n)-F(m)
		\\&\leq
		F(m_n)-F(m)-F_n(m_n)+F_n(m)
		\\&\leq
		H_n \mf l(m, m_n)^\xi
		\eqcm
	\end{align*}
	where we applied the \assu{ass:clo}{Closeness} condition in the last step. Thus,
	\begin{equation*}
		c_{\ms g} \mf l (m,m_n)^{\gamma-\xi} \leq H_n\eqcm
	\end{equation*}
	which implies the claimed inequality.
\end{proof}
Define
\begin{equation*}
	X(q) :=  \frac{F_n(q)-F_n(m)-F(q)+F(m)}{\mf l(m, q)^\xi}\eqfs
\end{equation*}
\begin{lemma}\label{lmm:strong:close}	
	Let $\zeta \geq 1$.
	Assume \assu{ass:smom}{Strong Moment} and \assu{ass:squad}{Strong Quadruple}.
	Then
	\begin{equation*}
		\Ex*{\sup_{q\in\mc Q} \abs{X(q)}^\zeta} \leq c n^{-\frac\zeta2} \mf M(\zeta) \min\brOf{\entrn(\mc Q, \mf b_m), \gamma_2(\mc Q, \mf b_m)}^\zeta
		\eqcm
	\end{equation*}
	where $c>0$ is a constant depending only on $\zeta$.
	Additionally, assume \assu{ass:sent}{Strong Entropy}. Then
	\begin{equation*}
		\Ex*{\sup_{q\in\mc Q} \abs{X(q)}^\zeta} \leq  C \mf M(\zeta) D^\zeta \eta_{n,\beta}^\zeta
		\eqcm
	\end{equation*}
	where $C>0$ is a constant depending only on $\zeta, \beta, c_{\ms e}$, and
	\begin{equation*}
		\eta_{\beta,n} := 
		\begin{cases} 
			n^{-\frac12} & \text{ for }\beta < 1\eqcm\\
			n^{-\frac12} \log(n+1) & \text{ for } \beta = 1\eqcm\\
			n^{-\frac1{2\beta}} & \text{ for } \beta > 1\eqfs
		\end{cases} 
	\end{equation*}
\end{lemma}
\begin{proof}
	Define
	\begin{align*}
		Z_i(q) 
		&:= 
		\frac1n \frac{\olc{Y_i}q-\olc{Y_i}m-\Ex*{\olc Yq-\olc Ym}}{\mf l(m, q)^\xi}
		\eqfs
	\end{align*}
	Thus, $X(q) = \sum_{i=1}^n Z_i(q)$.
	The \assu{ass:smom}{Strong Moment} condition together with the \assu{ass:squad}{Strong Quadruple} condition imply that $Z_i$ integrable.
	Let $(Z_1\pr, \dots, Z_n\pr)$ be an independent copy of $(Z_1, \dots, Z_n)$, where $(Y_1\pr, \dots, Y_n\pr)$ is an independent copy of $(Y_1, \dots, Y_n)$.
	By \assu{ass:squad}{Strong Quadruple}, we have
	\begin{align*}
			&
		n^2 \br{Z_i(q)-Z_i(p)-Z_i\pr(q)+Z_i\pr(p)}^2 
		\\&= 
		\br{\frac{\olc{Y_i}q-\olc{Y_i}m-\olc{Y_i\pr}q+\olc{Y_i\pr}m}{\mf l(m,q)^\xi}-\frac{\olc{Y_i}p-\olc{Y_i}m-\olc{Y_i\pr}p+\olc{Y_i\pr}m}{\mf l(m,p)^\xi}}^2
		\\&\leq 
		\mf b_m(q,p)^2 \mf a(Y_i,Y_i\pr)^2
		\eqfs
	\end{align*}
	Furthermore,
	\begin{equation*}
		\Ex*{\br{\frac1n\sum_{i=1}^n \mf a(Y_i, Y_i\pr)^2}^{\frac\zeta2}} \leq \mf M(\zeta)
	\end{equation*}
	with $\mf M(\zeta) < \infty$ due to the assumption \assu{ass:smom}{Strong Moment}.
	Thus, \autoref{chaining:empproc} (appendix \autoref{app:chaining}) implies
	\begin{equation*}
		\Ex*{\sup_{q\in\mc Q} \abs{X(q)}^\zeta} \leq c n^{-\frac\zeta2} \mf M(\zeta) \min\brOf{\entrn(\mc Q, \mf b_m), \gamma_2(\mc Q, \mf b_m)}^\zeta \eqfs
	\end{equation*}
	\assu{ass:sent}{Strong Entropy} together with \autoref{lmm:chaining:rate} (appendix \autoref{app:chaining}) yield
	\begin{equation*}
		\Ex*{\sup_{q\in\mc Q} \abs{X(q)}^\zeta} \leq C \mf M(\zeta) \br{D \eta_{n,\beta}}^\zeta
\end{equation*}
	for a constant $C > 0$ depending only on $\beta, \zeta, c_{\ms e}$.
\end{proof}
\begin{proof}[of \autoref{thm:abstr_rate_exp}]
	Using $H_n := \sup_{q\in\mc Q} \abs{X(q)}$ in \autoref{lmm:strong:growth} fulfills the \assu{ass:clo}{Closeness} condition by definition of $X$. Next, apply  \autoref{lmm:strong:close} with $\zeta := \frac{\kappa}{\gamma-\xi}$ to conclude the proof.
\end{proof}
\subsection{Proof of \autoref{thm:abstr_weak_strong}}
\begin{lemma}\label{lmm:entropyuse}
	The condition \assu{ass:se}{Small Entropy} implies
	\begin{equation*}
		\entrn(\ball_{R}(o, \mf b), \mf b^{1-\xi})
		\leq 
		c R^{1-\xi} (1-\xi)^{-\beta}
	\end{equation*}
	for $\xi \in (0,1)$, where $c > 0$ depends only on $\beta, c_{\ms e}$.
\end{lemma}
\begin{proof}
	Obviously, we have
	\begin{equation*}
		\entrn(Q, \mf b^{1-\xi}) \leq \int_0^\infty \sqrt{\log N(Q, \mf b^{1-\xi}, r)} \dl r
	\end{equation*}
	for any set $Q \subset \mc Q$. Furthermore,
	\begin{equation*}
		N(Q, \mf b^{1-\xi}, r) = N(Q, \mf b, r^{\frac1{1-\xi}})
		\eqcm
	\end{equation*}
	which yields
	\begin{equation*}
		\int_0^\infty \sqrt{\log N(Q, \mf b^{1-\xi}, r)} \dl r
		=
		(1-\xi) \int_0^\infty s^{-\xi} \sqrt{\log N(Q, \mf b, s)} \dl s
	\end{equation*}
	Thus, for $Q := \ball_{R}(o, \mf b)$, we obtain, using the \assu{ass:se}{Small Entropy} condition,
	\begin{equation*}
		\entrn(Q, \mf b^{1-\xi})
		\leq 
		c_{\ms e} (1-\xi) \int_0^R r^{-\xi} \log\brOf{\frac{R}{r}}^\beta \dl r
		\eqfs
	\end{equation*}
	To calculate the integral, we substitute $s := \frac rR$ and get
	\begin{equation*}
		\int_0^R r^{-\xi} \log\brOf{\frac{R}{r}}^\beta \dl r
		=
		R^{1-\xi} \int_0^1 s^{-\xi} \log\brOf{\frac{1}{s}}^\beta \dl s
		\eqfs
	\end{equation*}
	For general $a\in(0,1),b>0$, we have
	\begin{equation*}
		\int_0^1 x^{-a} \log\brOf{\frac1x}^b \dl x = (1-a)^{-b-1} \Gamma(b+1)
		\eqcm
	\end{equation*}
	where $\Gamma(\cdot)$ is the Gamma function.
	Thus,
	\begin{equation*}
		\int_0^1 s^{-\xi} \log\brOf{\frac{1}{s}}^\beta \dl s \leq c_\beta (1-\xi)^{-\beta-1}
	\end{equation*}
	for a constant $c_\beta > 0$ depending only on $\beta$.
	Putting everything together, we obtain
	\begin{equation*}
		\entrn(Q, \mf b^{1-\xi})
		\leq 
		c R^{1-\xi} (1-\xi)^{-\beta}
		\eqfs
	\end{equation*}
\end{proof}
\begin{lemma}\label{lmm:seqbound}
	Set $\xi_n := 1 - \log(n)^{-1}$. Then
	\begin{equation*}
		\br{n^{-\frac12} \br{1-\xi_n}^{-\beta}}^{\frac{1}{\gamma-\xi_n}} 
		\leq
		c_\gamma n^{-\frac1{2(\gamma-1)}} \log(n)^{\frac{\beta}{\gamma-1}}
	\end{equation*}
	where $c_\gamma > 0$ is a constant depending only on $\gamma$.
\end{lemma}
\begin{proof}
	We have
	\begin{align*}
		\br{n^{-\frac12} \br{1-\xi_n}^{-\beta}}^{\frac{1}{\gamma-\xi_n}} 
		&= 
		\br{n^{-\frac{1}{2(\gamma-\xi_n)}}} \log(n)^{\frac{\beta}{\gamma-\xi_n}} 
		\\&= 
		\br{n^{-\frac{1}{2\br{z+\frac1{\log(n)}}}}} \log(n)^{\frac{\beta}{z+\frac1{\log(n)}}} 
		\eqcm
	\end{align*}
	where $z = \gamma-1$. We use
	\begin{equation*}
		\log(n)^{\frac{\beta}{z+\frac1{\log(n)}}} 
		\leq
		\log(n)^{\frac{\beta}{z}}
		\eqcm
	\end{equation*}
	\begin{equation*}
		n^{-\frac{1}{2(z+\frac1{\log(n)})}}
		=
		n^{-\frac{\log(n)}{2z\log(n)+2}}
		=
		\exp\brOf{-\frac{\log(n)^2}{2z\log(n)+2}+\frac{\log(n)}{2z}} n^{-\frac1{2z}}
		\eqcm
	\end{equation*}
	and
	\begin{align*}
		-\frac{\log(n)^2}{2z\log(n)+2}+\frac{\log(n)}{2z}
		&=
		\frac{\log(n)}{2z (z\log(n)+1)}
		\\&\leq
		\frac1{2z^2}
		\eqcm
	\end{align*}
	to obtain
	\begin{equation*}
		\br{n^{-\frac12} \br{1-\xi_n}^{-\beta}}^{\frac{1}{\gamma-\xi_n}} 
		\leq
		\exp\brOf{\frac1{2z^2}} n^{-\frac1{2z}} \log(n)^{\frac{\beta}{z}}
		\eqfs
	\end{equation*}
\end{proof}
\begin{proof}[of \autoref{thm:abstr_weak_strong}]
	For $n\in\N, n\geq 3$, set $\xi_n := 1 - \log(n)^{-1}$, $Q_n := \ball_{R_n}(o, \mf b)$, and $R_n := n$.
	For $n$ large enough, the \assu{ass:ex2}{Existence'} condition implies the existence of $m^{Q_n}_n \in \argmin_{q\in Q_n} F_n(q)$ and $m^{Q_n}\in\argmin_{q\in Q_n} F(q)$.
	
	\autoref{thm:abstr_rate_exp} implies
	\begin{equation*}
		\Ex*{\mf b(m^{Q_n}, m^{Q_n}_n)^{\kappa}} \leq C n^{-\frac{\kappa}{2(\gamma-\xi_n)}} \entrn(Q_n, \mf b^{1-\xi_n})^{\frac{\kappa}{\gamma-\xi_n}}\, \mf M\brOf{\frac{\kappa}{\gamma-\xi_n}}
		\eqcm
	\end{equation*}
	for $n$ large enough.
	Note, that $C>0$ can be chosen independently of $n$ (even for $\xi_n$ depending on $n$).
	
	In \assu{ass:smom}{Strong Moment} we require $\kappa \geq \gamma-1$, because then $x\mapsto x^{\frac{\kappa}{\gamma-1}}$ is convex, which is needed for the symmetrization argument in the proof of \autoref{thm:abstr_rate_exp}. But, if $\kappa = \gamma-1$, then $\frac{\kappa}{\gamma-\xi_n} < 1$, and \autoref{thm:abstr_rate_exp} cannot be applied directly. For this technical reason, we assumed $\kappa > \gamma-1$, so that $\kappa \geq \gamma-\xi_n$ for $n$ large enough.	
	
	By \assu{ass:se}{Small Entropy} and \autoref{lmm:entropyuse} there is $c_\beta>0$ such that for $n\in\N$ large enough, we have 
	\begin{equation*}
		\entrn(\Ball{R_n}{o}{\mf b}, \mf b^{1-\xi_n}) \leq c_\beta R_n^{1-\xi_n} \br{{1-\xi_n}}^{-\beta}
		\eqfs
	\end{equation*}
	Using $R_n^{1-\xi_n} = n^{\frac1{\log(n)}} = \exp(1)$ together with \autoref{lmm:seqbound}, we obtain
	\begin{equation*}
		\Ex*{\mf b(m^{Q_n}, m^{Q_n}_n)^\kappa} \leq C\pr \br{n^{-\frac12} \log(n)^\beta}^{-\frac{\kappa}{\gamma-1}} \, \mf M\brOf{\frac{\kappa}{\gamma-\xi_n}}
		\eqfs
	\end{equation*}
	As $\lim_{n\to\infty}\mf M\brOf{\frac{\kappa}{\gamma-\xi_n}} = \mf M\brOf{\frac{\kappa}{\gamma-1}}$, we have 
	\begin{equation*}
		\Ex*{\mf b(m^{Q_n}, m^{Q_n}_n)^{\kappa}} \leq C\prr n^{-\frac{\kappa}{2(\gamma-\xi_n)}} \br{R_n^{1-\xi_n} \br{{1-\xi_n}}^{-\beta}}^{\frac{\kappa}{\gamma-\xi_n}}\, \mf M\brOf{\frac{\kappa}{\gamma-1}}
		\eqfs
	\end{equation*}
	Finally, there is a $n_0\in\N$ such that for all $n\geq n_0$, we have $m \in Q_n$, which implies $m = m^{Q_n}$. Thus,
	\begin{equation*}
		\Ex*{\mf b(m, m^{Q_n}_n)^\kappa} = \bigO\brOf{\br{n^{-\frac12} \log(n)^\beta}^{-\frac{\kappa}{\gamma-1}}}
		\eqfs
	\end{equation*}
\end{proof}
\section{Stability of Quadruple Inequalities} \label{app:quad_stab}
We present some trivial stability results for quadruple inequalities. The notation we use here is introduced in the beginning of \autoref{sec:quadruple}.
\begin{itemize}
\item[]
	\hspace*{-0.5cm}\textbf{Subsets:}\\
	If $(\mc Q, \mc Y, \mf c, \mf a, \mf b)$ fulfills the weak quadruple inequality, then so does \linebreak$(\mc Q\pr, \mc Y\pr, \mf c, \mf a, \mf b)$ with $\mc Q\pr \subset \mc Q$, $\mc Y\pr \subset \mc Y$.
\item[]
	\hspace*{-0.5cm}\textbf{Images:}\\
	Assume $(\mc Q, \mc Y, \mf c, \mf a, \mf b)$ fulfills the weak quadruple inequality and $f\colon\mc Y\pr\to\mc Y$, $g\colon\mc Q\pr\to\mc Q$. Then $(\mc Q\pr, \mc Y\pr, \mf c\pr, \mf a\pr, \mf b\pr)$ fulfills the weak quadruple inequality with $\mf c\pr(y,q) = \mf c(f(y), g(q))$, $\mf a\pr(y,z) = \mf a(f(y), f(z))$, $\mf b\pr(q,p) = \mf b(g(q), g(p))$.
\item[]
	\hspace*{-0.5cm}\textbf{Limits:}\\
	Let $(\mc Q, \mc Y, \mf c_i, \mf a_i, \mf b_i)$ fulfill the weak quadruple inequality for $i \in \N$ and assume for all $q,p\in\mc Q$ and $y,z\in\mc Y$ the point-wise limits 
	\begin{align*}
		\mf a(y,z) &:= \lim_{i\to\infty} \mf a_i(y,z)\\
		\mf b(q,p) &:= \lim_{i\to\infty} \mf b_i(q,p)\\
		\mf c(y,q) &:= \lim_{i\to\infty} \mf c_i(y,q)
	\end{align*}
	exist. Then $(\mc Q, \mc Y, \mf c, \mf a, \mf b)$ also fulfills the weak quadruple inequality.
\end{itemize}
Similar results hold for the strong quadruple inequality. For the following results it may not be so easy to obtain an analog for the strong quadruple inequality.
\begin{itemize}
\item[] \hspace*{-0.5cm}\textbf{Product Spaces:}\\
	If $(\mc Q_i, \mc Y_i, \mf c_i, \mf a_i, \mf b_i)$ fulfill the weak quadruple inequality for all $i \in \N$, then so does
	$(\mc Q, \mc Y, \mf c, \mf a, \mf b)$ where $\mc Q = \bigtimes_{i\in \N} \mc Q_i$, $\mc Y = \bigtimes_{i\in \N} \mc Y_i$, $\mf c = \sum_{i=1}^\infty \mf c_i$, $\mf a = \normof{(\mf a_i)_{i\in\N}}_{\ell^2}$, $\mf b = \normof{(\mf b_i)_{i\in\N}}_{\ell^2}$.	
	\begin{proof}
		We have
		\begin{align*}
			\olc yq- \olc zq -\olc yp+\olc zp 
			&= 
			\sum_{i=1}^\infty \br{
			\leftidx{^{\mf c_i}}{\ol {y_i}{q_i}}
			- \leftidx{^{\mf c_i}}{\ol {z_i}{q_i}} 
			- \leftidx{^{\mf c_i}}{\ol {y_i}{p_i}}
			+\leftidx{^{\mf c_i}}{\ol {z_i}{p_i}}}
			\\&\leq 
			\sum_{i=1}^\infty \mf a_i(y_i, z_i) \mf b_i(q_i, p_i)
			\\&\leq 
			 \mf a(y,z) \mf b(q,p)
			\eqcm
		\end{align*}
		using the Cauchy--Schwartz inequality.
	\end{proof}
\item[] \hspace*{-0.5cm}\textbf{Measure Spaces:}\\
	Let $(\Omega, \mc A, \mu)$ be a measure space.
	Assume $(\mc Q, \mc Y, \mf c(\omega), \mf a(\omega), \mf b(\omega))$ fulfills the weak quadruple inequality for every $\omega \in \Omega$. Let $s,t>0$ with $\frac1s+\frac1t=1$. 
	Let $L(\Omega, \mc Q)$ be the set of measurable functions from $\Omega$ to $\mc Q$,  define $L(\Omega, \mc Y)$ analogously.
	For $q,p \in L(\Omega, \mc Q)$, $y,z \in L(\Omega, \mc Y)$, let
	\begin{align*}
		\mf C(y, q) &:= \int\!\! \mf c(\omega, y(\omega), q(\omega)) \, \dl \mu(\omega) \eqcm\\
		\mf A(y, z) &:= \br{\int\!\! \mf a(\omega; y(\omega),z(\omega))^t \, \dl \mu(\omega)}^{\frac1t} \eqcm \\
		\mf B(q, p) &:= \br{\int\!\! \mf b(\omega; q(\omega),p(\omega))^s \, \dl \mu(\omega)}^{\frac1s}
		\eqcm
	\end{align*}
	where we implicitly assume that the necessary measurablity and integrability conditions are fulfilled. Then
	\begin{equation*}
		\br{L(\Omega, \mc Q), L(\Omega, \mc Y), \mf C, \mf A, \mf B}
	\end{equation*}
	also fulfills the quadruple inequality.
	\begin{proof}
		We have
		\begin{align*}
			&\mf C(y,q) - \mf C( z,q) -\mf C(y,p)+\mf C(z,p)
			\\&= 
			\int \mf c(\omega; y(\omega),q(\omega)) -  \mf c(\omega; y(\omega),p(\omega))\\
			 &\qquad- \mf c(\omega; z(\omega),q(\omega)) + \mf c(\omega; z(\omega),p(\omega))\dl \mu(\omega)
			\\&\leq 
			\int \mf a(\omega; y(\omega),z(\omega)) \mf b(\omega; q(\omega),p(\omega))\dl \mu(\omega)
			\\&\leq 
			\mf A(y,z) \mf B(q, p)
			\eqcm
		\end{align*}
		by Hölder's inequality.
	\end{proof}
\item[]\hspace*{-0.5cm}\textbf{Minima:}\\
	Let $(\mc Q, \mc Y, \mf c, \mf a, \mf b)$ fulfill the weak quadruple inequality. Let $\tilde{\mc Y} \subset 2^{\mc Y}$. Define the cost function $\mf C \colon \tilde{\mc Y} \times \mc Q \to \R$ by $\mf C(\mo y, q) = \inf_{y\in\mo y} \mf c(y,q)$ and
	$\mf A(\mo y,\mo z) = \sup_{y\in\mo y,z\in\mo z} \mf a(y,z)$ assuming the infinma and suprema are finite. Then $(\mc Q, \tilde{\mc Y}, \mf C, \mf A, \mf b)$ fulfills the weak quadruple inequality.
	\begin{proof}
	Let $\mo y,\mo z\in \tilde{\mc Y}$ and  $q ,p\in \mc Q$.
	Assume there are $y_q, y_p \in \mo y, z_q, z_p \in \mo z$ such that $\mf C(\mo y,q) = \olc {y_{q}}{q}$, $\mf C(\mo y,p) = \olc {y_{p}}{p}$, $\mf C(\mo z,q) = \olc {z_{q}}{q}$, and $\mf C(\mo z,p) = \olc {z_{p}}{p}$. Then
	\begin{align*}
		\mf C(\mo y,q) - \mf C(\mo y,p) - \mf C(\mo z,q) + \mf C(\mo z,p)
		&=
		\olc {y_{q}}{q} - \olc {y_{p}}{p} - \olc {z_{q}}{q} + \olc {z_{p}}{p}
		\\&\leq
		\olc {y_{p}}{q} - \olc {y_{p}}{p} - \olc {z_{q}}{q} + \olc {z_{q}}{p}
		\\&\leq
		\mf a(y_{p},z_{q}) \mf b(q, p) 
		\\&\leq
		\mf A(\mo y,\mo z) \mf b(q,p)
		\eqfs
	\end{align*}
	If the infima are not attained, one can follow the same proof with minimizing sequences.
	\end{proof} 
	In many interesting problems the setting is opposite to what was described before, i.e., $\mf C \colon \mc Y \times \tilde{\mc Q} \to \R ,\, (y, \mo q) \mapsto \inf_{q\in\mo q} \mf c(y,q)$, where $\tilde{\mc Q} \subset 2^{\mc Q}$: the elements of the descriptor space are subsets and the elements of data space are points. Examples are $k$-means, where $\tilde{\mc Q}$ consists of $k$-tuples of points in $\mc Q$, or fitting hyperplanes. A quadruple inequality with $\sup_{q\in\mo q,p\in\mo p} \mf b(q,p)$ as the descriptor distance can be established. Unfortunately, this is usually not useful, as the entropy condition cannot be fulfilled with distances of this type. The framework described in this article can still be applied using inequalities as for bounded spaces, see section \ref{ssec:bounded}. But we cannot directly use the advantage of quadruple inequalities over Lipschitz-continuity.
\end{itemize}
\section{Proof of \autoref{lmm:weak_implies_strong_power}} \label{app:power_metric_bound}
We first state and prove two simple lemmas for some simple arithmetic expressions and then use those for the proof of \autoref{lmm:weak_implies_strong_power}.
\begin{lemma}\label{lmm:fraction_bound}
	Let $A,B \in \R$, $a,b,c,r\geq 0$, $s,t>0$.
	Assume $t \geq s \ \Leftrightarrow\  b\geq a$.
	Assume $|A| \leq ra$, $|B| \leq rb$, $|A-B| \leq rc$ .
	Then
	\begin{equation*}
		\abs{\frac As -\frac Bt} \leq  r \frac{\min(s,t)c+|s-t|\min(a,b)}{st}
		\eqfs
	\end{equation*}
\end{lemma}
\begin{proof}
	For $t \geq s$, using the bound on $A$ and on $A-B$ implies $\frac As -\frac Bt \leq  r \frac{(t-s)a+sc}{st}$. Similarly,
	for $s \geq t$, we get $\frac As -\frac Bt \leq  r \frac{tc+(s-t)b}{st}$ by using the bound on $B$ and $A-B$.
	Together, we obtain
	\begin{equation*}
		\frac As -\frac Bt \leq  r \frac{\min(s,t)c+|s-t|\min(a,b)}{st}
		\eqfs
	\end{equation*}
	We finish the proof by pointing out the symmetry between $(A,a,s)$ and $(B,b,t)$.
\end{proof}
\begin{lemma}\label{lmm:abc_beta_bound}
	Let $a,b,c > 0$, $\beta \in [0,1]$. Assume $a\leq b$, $b \leq a+c$, $c \leq a+b$.
	Then 
	\begin{equation*}
		\frac{c a^\beta+\br{b^\beta-a^\beta}a}{a^\beta b^\beta} \leq 2^{\beta} c^{1-\beta}
		\eqfs
	\end{equation*}
\end{lemma}
\begin{proof}
The statement is trivial for $\beta\in \{0,1\}$. So let $\beta\in(0,1)$.\\
\underline{Case I, $c \leq a$:}
Define $f(x) = 1-x-(1+x)^\beta(1-x^{1-\beta})$. Then $f(0)=f(1)=0$ and $f\prr(x) = -(1-\beta) \beta x^{-\beta - 1} (x + 1)^{\beta - 2} (1-x^{\beta + 1})\leq 0$ for $x\in(0,1)$. Thus, $f(x) \geq 0$ for $x\in[0,1]$. In particular
$f\brOf{\frac{c}{a}} \geq 0$, which implies $a-c \geq (a+c)^\beta(a^{1-\beta} - c^{1-\beta}) \geq b^\beta(a^{1-\beta} - c^{1-\beta})$. Thus,
\begin{equation*}
\frac{c a^\beta+\br{b^\beta-a^\beta}a}{a^\beta b^\beta} \leq  c^{1-\beta}\eqfs
\end{equation*}
\underline{Case II, $c \geq a$:}
As $1-\beta \leq 1$ and $c-a \geq 0$, we have
$(c-a)^{1-\beta} + a^{1-\beta} \leq 2^\beta c^{1-\beta}$. Multiplying by $(c-a)^\beta$ and using $c-a \leq b$, we get
$c-a \leq b^{\beta} \br{2^\beta c^{1-\beta} - a^{1-\beta}}$. Thus,
\begin{align*}
\frac{c a^\beta+\br{b^\beta-a^\beta}a}{a^\beta b^\beta} \leq 2^{\beta} c^{1-\beta}
\eqfs
\end{align*}
\end{proof}
\begin{proof}[of \autoref{lmm:weak_implies_strong_power}]
Applying \autoref{lmm:fraction_bound} to the left hand side of equation \eqref{eq:strquadpowerbound}, yields
\begin{align*}
	\frac{\olc yq - \olc ym - \olc zq + \olc zm}{\mf b(q,m)^{\xi}}
	-
	\frac{\olc yp - \olc ym - \olc zp + \olc zm}{\mf b(p,m)^{\xi}}
	\leq
	\mf a(y,z)\, \tilde {\mf b}_{m,\xi}(q,p)
\end{align*}
where
\begin{equation*}
	\tilde {\mf b}_{m,\xi}(q,p) = \frac{\min(\olb qm, \olb pm)^\xi\, \olb qp + \abs{\olb qm^\xi - \olb pm^\xi}\min(\olb qm, \olb pm)}{\olb qm^\xi\, \olb pm^\xi}
\end{equation*}
with the short notation $\olb qp := \mf b(q,p)$,
for all $y,z,q,p,m \in \mc Q$. Applying \autoref{lmm:abc_beta_bound} yields $\tilde {\mf b}_{m,\xi}(q,p) \leq 2^\xi \mf b(q,p)^{1-\xi}$.
\end{proof}
\section{Projection Metric Counter Example} \label{app:dproj}
We take a tripod $(\mc Q, d)$ as a simple example of a non-Euclidean Hadamard space, see \cite[Example 3.2]{sturm03}, and show that it does not fulfill the strong quadruple inequality with the projection metric 
\begin{equation*}
	d_{m}^{\ms{proj}}(q,p) := \sqrt{\frac{\ol qp^2 - \br{\ol qm - \ol pm}^2}{\ol qm\, \ol pm}}
	\eqfs
\end{equation*}  
Let $r > \varepsilon > 0$ and define $y,z,q,p,o$ on a tripod as in \autoref{fig:tripod}.
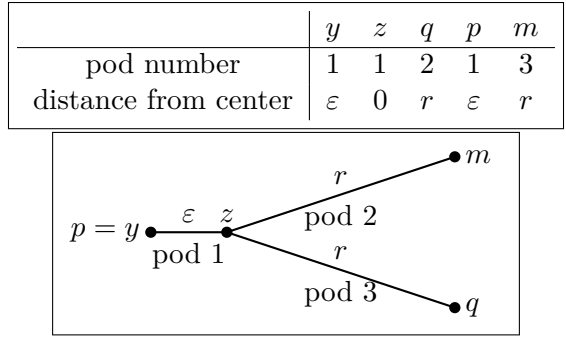
\begin{figure}
	\begin{center}
	\fbox{\begin{tabular}{c|ccccc}
	 & $y$ & $z$ & $q$ & $p$ & $m$ \\ 
	\hline 
	pod number& 1 & 1 & 2 & 1 & 3 \\ 
	distance from center& $\varepsilon$ & 0 & $r$ & $\varepsilon$ & $r$
	\end{tabular}}\\
	\fbox{\begin{tikzpicture}[every path/.style={thick}, baseline=(current bounding box.center)]
		\coordinate[label=left:{$p=y$}] (x) at (-1,0); 
		\coordinate[label=above:{$z$}] (m) at (0,0); 
		\coordinate[label=right:{$q$}] (y) at (3,-1); 
		\coordinate[label=right:{$m$}] (z) at (3,1); 
		\draw (x)--(m) node[midway, above] {$\varepsilon$} node[midway, below] {pod 1};
		\draw (y)--(m) node[midway, above] {$r$} node[midway, below] {pod 3};
		\draw (z)--(m) node[midway, above] {$r$} node[midway, below] {pod 2};
		\draw[fill] (x) circle [radius=0.06]; 
		\draw[fill] (m) circle [radius=0.06]; 
		\draw[fill] (y) circle [radius=0.06]; 
		\draw[fill] (z) circle [radius=0.06]; 
	\end{tikzpicture} }
	\caption{Tripod counter-example for the strong quadruple inequality}
	\label{fig:tripod}
	\end{center}
\end{figure}
We take $\mf c = d^2$, $\xi=1$, $\mf l = d$, $\mf a = K d$, and $\mf b_m = d_{m}^{\ms{proj}}$.
Then 
\begin{align*}
	\frac{\olc yq - \olc ym - \olc zq + \olc zm}{\mf l(q,m)}
	-
	\frac{\olc yp - \olc ym - \olc zp + \olc zm}{\mf l(p,m)}
	&=
	2 \varepsilon\eqcm
	\\
	\mf a(y,z) \mf b_m(q,p) = K \varepsilon \sqrt{2\frac{\varepsilon}{r+\varepsilon}}
	\eqfs
\end{align*}
If the strong quadruple inequality holds, then
\begin{equation*}
	K  \geq \sqrt{2\frac{r+\varepsilon}{\varepsilon}} \xrightarrow{\varepsilon\searrow0} \infty
	\eqfs
\end{equation*}
Thus, $d_{m}^{\ms{proj}}$ is not a suitable candidate for the strong quadruple distance in general Hadamard spaces.
\section{Optimality of Power Inequality}\label{app:power_inequ_opti}
We show that $8\alpha2^{-2\alpha}$ is the optimal constant, and that we cannot extend \autoref{thm:power_inequ} to $\alpha > 1$ or $\alpha < \frac12$.
Let $\epsilon\in(0,1)$ and $(\mc Q, d)$ be a metric space with $q,p,y,z\in\mc Q$ such that for each case below the distances have the values written down in \autoref{tbl:dists}. One can easily show that in all three cases the necessary triangle inequalities and the nice quadruple inequality hold.
\begin{table}
\begin{center}
\begin{tabular}{c||c|c|c|c|c|c}
	 Case&$\ol yq$ & $\ol yp$ & $\ol zq$ & $\ol zp$ & $\ol yz$ & $\ol qp$  \\ 
	\hline 
	\hline 
	 (a) &$1-\epsilon$ & $1-3\epsilon$ & $1-2\epsilon$ & $1$ & $2-3\epsilon$ & $2\epsilon$\\
	 \hline 
	 (b) & $1$ & $\epsilon$ & $1-\epsilon$ & $2\epsilon$ & $2\epsilon$ & $1$\\
	 \hline 
	 (c) & $2\epsilon$ & $\epsilon$ & $1$ & $1$ & $1$ & $\epsilon$ \\
\end{tabular}
\caption{Distances of four points $y,z,q,p\in\mc Q$ for showing lower bounds of the constant in \autoref{thm:power_inequ}.}
\label{tbl:dists}
\end{center}
\end{table}
%
\begin{enumerate}[label=(\alph*)]
\item 
	For $\alpha \in [\frac12,1]$, we have, using l'Hôpital's rule,
	\begin{align*}
	&\lim_{\epsilon\searrow0}\frac{\ol yq^{2\alpha} - \ol yp^{2\alpha} - \ol zq^{2\alpha} + \ol zp^{2\alpha}}{ \ol yz^{2\alpha-1}\, \ol qp} 
	\\&= 
	\lim_{\epsilon\searrow0}\frac{(1-\epsilon)^{2\alpha} - (1-3\epsilon)^{2\alpha} - (1-2\epsilon)^{2\alpha} + 1}{ (2\epsilon)(2-3\epsilon)^{2\alpha-1}}
	\\&=
	\lim_{\epsilon\searrow0} 2\alpha\frac{-(1-\epsilon)^{2\alpha-1} + 3(1-3\epsilon)^{2\alpha-1} + 2(1-2\epsilon)^{2\alpha-1}}{ 2(2-3\epsilon)^{2\alpha-1}-6\epsilon(2\alpha-1)(2-3\epsilon)^{2\alpha-2}}
	\\&=8\alpha2^{-2\alpha}
	\eqfs
	\end{align*}
	Thus, the constant $8\alpha2^{-2\alpha}$ in \autoref{thm:power_inequ} is optimal.
\item	
	For $\alpha >1$, we have, using l'Hôpital's rule,
	\begin{align*}
	&\lim_{\epsilon\searrow0}\frac{\ol yq^{2\alpha} - \ol yp^{2\alpha} - \ol zq^{2\alpha} + \ol zp^{2\alpha}}{ \ol yz^{2\alpha-1}\, \ol qp} 
	\\&= 
	\lim_{\epsilon\searrow0} \frac{1 -\epsilon^{2\alpha} - (1-\epsilon)^{2\alpha} + (2\epsilon)^{2\alpha}}{ (2\epsilon)^{2\alpha-1}} 
	\\&=
	\frac{2\alpha}{2\alpha-1} \lim_{\epsilon\searrow0} \frac{ - \epsilon^{2\alpha-1} + (1-\epsilon)^{2\alpha-1} + 2	(2\epsilon)^{2\alpha-1}}{ 2^{2\alpha-1}\epsilon^{2\alpha-2}}
	\\&=
	\infty
	\eqfs
	\end{align*}
	Thus, there is no power inequality in the form of \autoref{thm:power_inequ} for $\alpha > 1$.
\item
	For $\alpha \in (0,\frac12)$, we have
	\begin{align*}
	&\lim_{\epsilon\searrow0}\frac{\ol yq^{2\alpha} - \ol yp^{2\alpha} - \ol zq^{2\alpha} + \ol zp^{2\alpha}}{ \ol yz^{2 \alpha-1}\, \ol qp} 
	\\&= 
	\lim_{\epsilon\searrow0}\frac{(2\epsilon)^{2\alpha} - \epsilon^{2\alpha} - 1 + 1}{\epsilon} 
	\\&=
	\lim_{\epsilon\searrow0} (2^{2\alpha}-1)\epsilon^{2\alpha-1}
	\\&=\infty
	\eqfs
	\end{align*}
	Thus, there is no power inequality in the form of \autoref{thm:power_inequ} for $\alpha < \frac12$.
\end{enumerate}
\section{Chaining}\label{app:chaining}
Recall the measures of entropy $\gamma_2$ and $\entrn$ defined in \autoref{chaining:def_entropy}. We add another useful entry to this list.
\begin{definition}[Bernoulli Bound]\label{chaining:entropy}
	For $T \subset \R^n$ define
	\begin{equation*}
		b(T) := \inf\cbOf{\sup_{t\in T_1} \normof{t}_1 + \gamma_2(T_2)\,\colon\, T_1, T_2 \subset \R^n, T \subset T_1+T_2}
		\eqcm
	\end{equation*}
	where $\gamma_2(T_2) = \gamma_2(T_2, d_2)$ for the Euclidean metric $d_2$ on $\R^n$, $\normof{t}_1 = \sum_{i=1}^n \abs{t_i}$, and $T_1 + T_2 = \cb{t_1 + t_2 \colon t_1 \in T_1, t_2 \in T_2}$.
\end{definition}
We write down the Bernoulli bound for powers of the Bernoulli process. \cite{bednorz14} show that the bound can be reversed (up to an universal constant). Thus, this step can be regarded as optimal.
\begin{theorem}[Bernoulli bound]\label{chaining:bernoulli}
	Let $\sigma_1, \dots, \sigma_n$ be independent random signs, i.e., $\Prof{\sigma_i = \pm1} =\frac12$.
	For $t\in R^n$ set $\tilde X_t :=\sum_{i=1}^n \sigma_i t_i$. Let $T \subset \R^n$.
	Let $\kappa \geq 1$.
	Then 
	\begin{equation*}
		\Ex*{\sup_{t\in T} \abs{\tilde X_t}^\kappa} \leq c_\kappa b(T)^\kappa
		\eqcm
	\end{equation*}
	where $c_\kappa$ depends only on $\kappa$.
\end{theorem}
\begin{proof}
Let $T_1, T_2 \subset \R^n$ such that $T \subset T_1 + T_2$.
As $(a+b)^\kappa \leq 2^{\kappa-1}\br{a^\kappa + b^\kappa}$ for all $a,b\geq 0$, we can split the supremum into two parts,
\begin{align*}
\Ex*{\sup_{t\in T} \abs{\tilde X_t}^\kappa} 
\leq
2^{\kappa-1} \br{\Ex*{\sup_{t\in T_1} \abs{\tilde X_t}^\kappa}  + \Ex*{\sup_{t\in T_2} \abs{\tilde X_t}^\kappa}}
\eqfs
\end{align*}
The first term is bounded using the 1-norm,
$\Ex*{\sup_{t\in T_1} \abs{\tilde X_t}^\kappa} \leq \sup_{t\in T_1} \normof{t}_1^\kappa$. For the second we use Talagrand's generic chaining bound for the supremum of the subgaussian process $\Ex*{\sup_{t\in T_2} \abs{\tilde X_t}^\kappa} \leq c_\kappa\pr \gamma_2(T_2)^\kappa$, see \cite{talagrand14}. We obtain
\begin{align*}
\Ex*{\sup_{t\in T} \abs{\tilde X_t}^\kappa} 
&\leq 
c_\kappa \br{\sup_{t\in T_1} \normof{t}_1^\kappa + \gamma_2(T_2)^\kappa}
\\&\leq 
c_\kappa \br{\sup_{t\in T_1} \normof{t}_1 + \gamma_2(T_2)}^\kappa
\eqfs
\end{align*}
\end{proof}
\begin{lemma}[Lipschitz connection]\label{chaining:lipschitz}
Let $(\mc Q, b)$ be a pseudo-metric space.
Assume there are function $f_i \colon \mc Q \to \R$ such that $\abs{f_i(q) - f_i(p)} \leq a_i b(q,p)$.
Let $T := \cb{(f_i(q))_{i=1,\dots,n} \colon q\in\mc Q}$. Set $a = (a_1, \dots, a_n)$.
Then
\begin{equation*}
	b(T) \leq C \normof{a}_2 \min\brOf{\entrn(\mc Q, b), \gamma_2(\mc Q, b)}
	\eqcm
\end{equation*}
where $C>0$ is an universal constant.
\end{lemma}
\begin{proof}
	For $\epsilon>0$, choose $\mc Q_2$ to be an $\epsilon$-covering of $\mc Q$ with respect to $b$, i.e., for all $q \in \mc Q$ there is a $p_q \in \mc Q_2$ such that $b(q,p_q) \leq \epsilon$. For $q\in\mc Q$ denote $t_q := (f_i(q))_{i=1,\dots,n} \in \R^n$. Define $T_2 := \cb{t_p \colon p\in\mc Q_2}$ and $T_1 := \cb{t_q - t_{p_q} \colon q\in Q}$. Then $T \subset T_1 + T_2$.
		The Lipschitz-condition implies $\normof{t_q-t_p}_2 \leq \normof{a}_2 b(q,p)$ for all $q,p\in\mc Q$.
		Thus, 
		\begin{equation*}
			\sup_{t\in T_1} \normof{t}_1 \leq \sup_{q \in \mc Q} \sqrt{n}\normof{t_q-t_{p_q}}_2 \leq \epsilon \sqrt{n} \normof{a}_2 
			\eqfs
		\end{equation*}
		By the properties of $\gamma_2$, see \cite{talagrand14}, we obtain
		\begin{equation*}
			\gamma_2(T_2) \leq c \normof{a}_2 \gamma_2(\mc Q_2, b) \leq c\pr \normof{a}_2  \int_{\epsilon}^\infty \!\! \sqrt{\log N(Q, b, r)} \dl r
		\end{equation*}
		for universal constants $c, c\pr>0$. Applying the two inequalities to the definition of $b(T)$ concludes the proof.
\end{proof}
\begin{lemma}[Symmetrization]\label{lmm:symm}
	Let $\mc Q$ be set. Let $Z_1, \dots, Z_n$ be centered, independent, and integrable stochastic processes indexed by $\mc Q$. Let $\Phi\colon \R\to\R$ be a convex, non-decreasing function. Let $(Z_1\pr, \dots, Z_n\pr)$ be an independent copy of $(Z_1, \dots, Z_n)$. Let $\varepsilon_1, \dots, \varepsilon_n$ be iid with $\Prof{\varepsilon_1 = \pm1} = \frac12$.
	Then
	\begin{equation*}
		\Ex*{\sup_{q\in\mc Q} \Phi\brOf{\sum_{i=1}^nZ_i(q)}} \leq
		\Ex*{\sup_{q\in\mc Q} \Phi\brOf{\sum_{i=1}^n\varepsilon_i\br{Z_i(q) - Z_i\pr(q)}}}
		\eqcm
	\end{equation*}
	assuming measurability of the involved terms.
\end{lemma}
The symmetrization lemma is well-known. The statement here is an intermediate step of from the proof of \cite[2.3.6 Lemma]{vaart96}.
\begin{theorem}[Empirical process bound]\label{chaining:empproc}
	Let $(\mc Q, b)$ be a separable pseudo-metric space. Let $Z_1, \dots, Z_n$ be centered, independent, and integrable stochastic processes indexed by $\mc Q$ with a $q_0 \in \mc Q$ such that $Z_i(q_0) = 0$ for $i=1,\dots, n$.
	Let $(Z_1\pr, \dots, Z_n\pr)$ be an independent copy of $(Z_1, \dots, Z_n)$.
	Assume the following Lipschitz-property: There is a random vector $A$ with values in $\R^n$ such that
	\begin{equation*}
		\abs{Z_i(q)-Z_i(p)-Z_i\pr(q)+ Z_i\pr(p)} \leq A_i b(q,p)
	\end{equation*} 
	for $i = 1,\dots, n$ and all $q,p\in\mc Q$.
	Let $\kappa \geq 1$.
	Then
	\begin{equation*}
		\Ex*{\sup_{q\in\mc Q} \abs{\sum_{i=1}^nZ_i(q)}^\kappa} \leq C\, \Ex*{\normof{A}_2^\kappa} \, 
		\min\brOf{\entrn(\mc Q, b), \gamma_2(\mc Q, b)}^\kappa
		\eqcm
	\end{equation*}
	where $C > 0$ is an universal constant.
\end{theorem}
\begin{proof}
	Use \autoref{lmm:symm}. Then apply \autoref{chaining:bernoulli} and \autoref{chaining:lipschitz} conditionally on $Z_1, \dots, Z_n$.
\end{proof}
\begin{lemma}\label{lmm:chaining:rate}
Let $(\mc Q, b)$ be a pseudo-metric space.
Let $D >0 $ such that $\diam(\mc Q, b) \leq D < \infty$.
Let $\beta > 0$. Assume 
\begin{equation*}
	\sqrt{\log(N(\mc Q, b , r))} \leq c_{\ms e} \br{\frac{D}{r}}^\beta
\end{equation*}
for all $0 < r < D$.
\begin{enumerate}[label=\environmentEnumerateLabel]
	\item If $\beta < 1$ then $\entrn(\mc Q, b) \leq \frac{c_{\ms e}D}{1-\beta}$.
	\item If $\beta = 1$ then $\entrn(\mc Q, b) \leq c_{\ms e}\pr D \log(n+1)$, where $c\pr>0$ depends only on $c_{\ms e}$.
	\item If $\beta > 1$ then $\entrn(\mc Q, b) \leq c_{\ms e}^{\frac1\beta} \frac\beta{\beta-1} n^{-\frac1{2\beta}+\frac12}$.
\end{enumerate}
In particular
\begin{equation*}
	n^{-\frac12} \,\entrn(\mc Q, b) \ \leq\  c \, D\, \eta_{\beta,n} \eqcm
\end{equation*}
where $c$ depends only on $c_{\ms e}$ and $\beta$ and
\begin{equation*}
	\eta_{\beta,n} := 
	\begin{cases} 
		n^{-\frac12} & \text{ for }\beta < 1\eqcm\\
		n^{-\frac12} \log(n+1) & \text{ for } \beta = 1\eqcm\\
		n^{-\frac1{2\beta}} & \text{ for } \beta > 1\eqfs
	\end{cases} 
\end{equation*}
\end{lemma}
The proof consists of calculating the entropy integral with the given bound on the covering numbers and, for $\beta \geq 1$, choosing the minimizing starting point of the integral $\epsilon>0$.
\section{Proof of the Power Inequality, \autoref{thm:power_inequ}}\label{app:power_inequality}
Let $(\mc Q, d)$ be a metric space. Use the short notation $\ol qp := d(q,p)$.
Let $q,p,y,z\in\mc Q$, $\alpha\in[\frac12,1]$. 
Assume
\begin{equation*}
	\olt yq - \olt yp - \olt zq + \olt zp \leq 2 \, \ol yz\, \ol qp
	\eqfs
\end{equation*}
The goal of this section is to prove 
\begin{equation*}
	\ol yq^{2\alpha} - \ol yp^{2\alpha} - \ol zq^{2\alpha} + \ol zp^{2\alpha} \leq 8 \alpha 2^{-2\alpha} \, \ol yz^{2\alpha-1}\, \ol qp
	\eqfs
\end{equation*}
\subsection{Arithmetic Form}
\autoref{thm:power_inequ} will be proven in the form of \autoref{con:ana}. 
\begin{lemma}\label{con:ana}
	Let $a,b,c\geq0$, $r,s\in[-1,1]$, and $\alpha\in[\frac12, 1]$.
	Then
	\begin{align*}
		&a^{2\alpha}-c^{2\alpha}-\br{a^2-2rab+b^2}^\alpha+\br{c^2-2scb+b^2}^\alpha 
		\\&\leq 
		8 \alpha 2^{-2\alpha} b \max(ra - sc, |a-c|)^{2\alpha-1}
		\eqfs
	\end{align*}
\end{lemma}
The advantage of using \autoref{con:ana} to prove \autoref{thm:power_inequ} is, that we do not need to consider a system of additional conditions for describing that the real values in the inequality are distances, which have to fulfill the triangle inequality. The disadvantage is, that we loose the possibility for a geometric interpretation of the proof.
\begin{lemma}\label{lmm:power_implies}
	\autoref{con:ana} implies \autoref{thm:power_inequ}.
\end{lemma}
\begin{proof}
	Three points from an arbitrary metric space can be embedded in the Euclidean plane so that the distances are preserved.
	Thus, the cosine formula of Euclidean geometry can be applied to the three points $y,p,q\in\mc Q$: We have
	\begin{equation*}
		\ol yq^2 = \ol yp^2 + \ol qp^2 - 2 s \,\ol yp\, \ol qp\eqcm
	\end{equation*}
	where $s := \cos(\measuredangle ypq)$ with the angle $\measuredangle ypq$ in the Euclidean plane.
	Similarly
	\begin{equation*}
		\ol zq^2 = \ol zp^2 + \ol qp^2 - 2 r \,\ol zp\, \ol qp
		\eqcm
	\end{equation*}
	where $r := \cos(\measuredangle zpq)$. Thus,
	\begin{align*}
		&
		\ol yq^{2\alpha} - \ol yp^{2\alpha} - \ol zq^{2\alpha} + \ol zp^{2\alpha}
		\\&=
		\br{\ol yp^2 + \ol qp^2 - 2 s \,\ol yp\, \ol qp}^\alpha
		-
		\br{\ol zp^2 + \ol qp^2 - 2 r \,\ol zp\, \ol qp}^\alpha
		- \ol yp^{2\alpha}
		+ \ol zp^{2\alpha}
		\\&=
		\br{c^2 + b^2 - 2 s cb}^\alpha
		-
		\br{a^2 + b^2 - 2 r ab}^\alpha
		- c^{2\alpha}
		+ a^{2\alpha}
		\eqcm
	\end{align*}
	where $a:=\ol zp$, $c:=\ol yp$, $b:=\ol qp$.
	Hence, \autoref{con:ana} yields
	\begin{align}\label{eq:arith:intermediate}
		\ol yq^{2\alpha} - \ol yp^{2\alpha} - \ol zq^{2\alpha} + \ol zp^{2\alpha}
		\leq
		8 \alpha 2^{-2\alpha} b \max(ra - sc, |a-c|)^{2\alpha-1}
		\eqfs
	\end{align}
	The assumption of \autoref{thm:power_inequ} states $\olt yq - \olt yp - \olt zq + \olt zp \leq 2 \, \ol yz\, \ol qp$.
	This implies
	\begin{equation*}
		2b \br{ra-sc}
		=
		\br{c^2 + b^2 - 2 scb}
		-
		\br{a^2 + b^2 - 2 rab}
		- c^{2}
		+ a^{2}
		\leq 
		2 b\,
		\ol yz
		\eqfs
	\end{equation*}
	Therefore, $ra-sc \leq \ol yz$ (or $b=0$, but then $q=p$ and \autoref{thm:power_inequ} becomes trivial).
	Furthermore, the triangle inequality implies $|a-c| = |\ol zp-\ol yp| \leq \ol yz$.
	Thus, we obtain
	\begin{equation}\label{eq:arith:maxbound}
		\max(ra - sc, |a-c|) \leq\ol yz
		\eqfs
	\end{equation}
	Finally, \eqref{eq:arith:intermediate} and \eqref{eq:arith:maxbound} together yield
	\begin{align*}
		\ol yq^{2\alpha} - \ol yp^{2\alpha} - \ol zq^{2\alpha} + \ol zp^{2\alpha}
		\leq
		8 \alpha 2^{-2\alpha}  \,\ol qp\, \ol yz^{2\alpha-1}
		\eqfs
	\end{align*}
\end{proof}
The remaining part of this section is dedicated to proving \autoref{con:ana}.

The proof of \autoref{con:ana} can be described as \textit{brute force}. We will distinguish many different cases, i.e., certain bounds on $a,b,c,r,s$, e.g., $a\leq c$ and $a> c$. In each case, we try to simplify the inequality step by step until we can solve it easily. Mostly, the simplification consists of taking some derivative and showing that the derivative is always negative (or always positive). Then we only need to show the inequality at one extremal point. This process may have to be iterated. It is often not clear immediately which derivative to take in order to simplify the inequality. Even after finishing the proof there seems to be no deeper reason for distinguishing the cases that are considered. Thus, unfortunately, the proof does not create a deeper understanding of the result.
\subsection{First Proof Steps and Outline of the Remaining Proof}
We want to show \autoref{con:ana} to prove \autoref{thm:power_inequ}.
We refer to the left hand side of the inequality, $a^{2\alpha}-c^{2\alpha}-\br{a^2-2rab+b^2}^\alpha+\br{c^2-2scb+b^2}^\alpha $, as LHS. By RHS we, of course, mean the right hand side, $8 \alpha 2^{-2\alpha} b \max(ra - sc, |a-c|)^{2\alpha-1}$.

For $\max(ra - sc, |a-c|)=0$ we have $a=c$ and $r\leq s$. Thus, LHS $\leq 0$. If $\max(ra - sc, |a-c|)>0$, LHS and RHS are continuous in all parameters. Thus, it is enough to show the inequality on a dense set. In particular, we can and will ignore certain special cases in the following which might introduce technical problems, e.g., "$0^0$".

We have to distinguish the cases $|a-c|=\max(ra - sc, |a-c|)$ and $ra - sc=\max(ra - sc, |a-c|)$. We further distinguish $a\geq c$ and $c \geq a$.
\begin{lemma}[$ra - sc$ vs $|a-c|$]\label{lmm:srbound}
	Let $a,b,c\geq0$, $r,s\in[-1,1]$, and $\alpha\in[\frac12, 1]$.
	Then
	\begin{align*}
		&ra - sc \geq a-c
		\quad\Leftrightarrow\quad 
		s 
		\leq
		(r-1) \frac ac+1
		\quad\Leftrightarrow\quad 
		r \geq 
		(s-1)\frac ca +1
		\eqcm
\\
		&ra - sc \geq c-a
		\quad\Leftrightarrow\quad 
		s 
		\leq
		(r+1) \frac ac -1
		\quad\Leftrightarrow\quad 
		r \geq 
		(s+1)\frac ca -1
		\eqfs
	\end{align*}
\end{lemma}
\begin{proof}
We have
	\begin{align*}
	&
		ra - sc \geq a-c
		\quad\Leftrightarrow\quad
		ra - a+c \geq sc
		\quad\Leftrightarrow\quad
		\frac{a(r-1)+c}{c} \geq s
\eqcm\\&
		ra - sc \geq c-a
		\quad\Leftrightarrow\quad
		ra - c+a \geq sc
		\quad\Leftrightarrow\quad
		\frac{a(r+1)-c}{c} \geq s
\eqcm\\&
		ra - sc \geq a-c
	\quad	\Leftrightarrow\quad
		ra \geq a-c+sc
		\quad\Leftrightarrow\quad
		r \geq (s-1)\frac ca +1
\eqcm\\&
		ra - sc \geq c-a
		\quad\Leftrightarrow\quad
		ra  \geq c-a+sc
		\quad\Leftrightarrow\quad
		r \geq (s+1)\frac ca -1
\eqfs
	\end{align*}
\end{proof}
\subsubsection{The Case $|a-c| \leq ra-sc$}
Consider the case $ra-sc \geq |a-c|$. The next lemma shows convexity in $r$ of the function "LHS minus RHS". This means, we only have to check values of $r$ on the border of its domain.
\begin{lemma}[Convexity in $r$]\label{lmm:ddr}
	Let $a,b,c\geq0$, $s,r\in[-1,1]$, $\alpha\in[\frac12,1]$.
	Assume $ra - sc \geq 0$.
	Define 
	\begin{align*}
		F(r,s) 
		:= 
		&\,a^{2\alpha}-c^{2\alpha}-\br{a^2-2rab+b^2}^\alpha+\br{c^2-2scb+b^2}^\alpha 
		\\&-
		8 \alpha 2^{-2\alpha} b (ra - sc)^{2\alpha-1}
		\eqfs
	\end{align*}
	Then
	\begin{equation*}
		\partial_r^2 F(r,s) \geq 0
		\eqfs
	\end{equation*}
\end{lemma}
Note, neither $\partial_s^2 F(r,s) \geq 0$ nor $\partial_s^2 F(r,s) \leq 0$ for all $a,b,c,s,r$.
\begin{proof}
	Define
	\begin{equation*}
		\ell(r,s) := a^{2\alpha}-c^{2\alpha}-\br{a^2-2rab+b^2}^\alpha+\br{c^2-2scb+b^2}^\alpha 
		\eqfs
	\end{equation*}
	We have
	\begin{align*}
		\partial_r \ell(r,s) &= 2ab\alpha\br{a^2-2rab+b^2}^{\alpha-1}\eqfs
	\end{align*}
	Define $h(r,s) := 8 \alpha 2^{-2\alpha} b \br{ra - sc}^{2\alpha-1}$. We have
	\begin{align*}
		\partial_r h(r,s) &= 8 \alpha (2\alpha-1) 2^{-2\alpha} b a \br{ra - sc}^{2\alpha-2}\eqfs
	\end{align*}
	We have $F(r,s) = \ell(r,s)-h(r,s)$.
	We have
	\begin{align*}
		f(r) :=	 \frac{	\partial_r \ell(r,s)-\partial_r h(r,s)}{2ab\alpha} &= \br{a^2-2rab+b^2}^{\alpha-1} - (2\alpha-1) \br{\frac{ra - sc}{2}}^{2\alpha-2}
	\end{align*}
	and 
	\begin{align*}
		\partial_r f(r) 
		&= 
		-2ab(\alpha-1) \br{a^2-2rab+b^2}^{\alpha-2} 
		- \frac12 a(2\alpha-1)(2\alpha-2) \br{\frac{ra - sc}{2}}^{2\alpha-3} 
		\eqfs
	\end{align*}
	We have $2ab \br{a^2-2rab+b^2}^{\alpha-2} \geq 0$ and $(\alpha-1) \leq 0$. Thus,
	\begin{equation*}
		-2ab(\alpha-1) \br{a^2-2rab+b^2}^{\alpha-2} \geq 0\eqfs
	\end{equation*}
	We have $\frac12 a(2\alpha-1) \br{\frac{ra - sc}{2}}^{2\alpha-3} \geq 0$ and  $(2\alpha-2)\leq 0$. Thus, $-\frac12 a(2\alpha-1)(2\alpha-2) \br{\frac{ra - sc}{2}}^{2\alpha-3} \geq 0$. Hence,	$\partial_r f(r) \geq 0$. Hence, $\partial_r^2 F(r,s) \geq 0$.
\end{proof}
\subsubsection{The Case $|a-c| \geq ra-sc$}\label{ssec:acgreasc}
In the case $|a-c| \geq ra-sc$, the RHS does not depend on $s$ or $r$. Thus, we maximize the LHS with respect to $r$ and $s$ and only need to show the inequality for this maximized term.

Define
\begin{equation*}
	\ell(r,s) := a^{2\alpha}-c^{2\alpha}-\br{a^2-2rab+b^2}^\alpha+\br{c^2-2scb+b^2}^\alpha 
	\eqfs
\end{equation*}
We have
\begin{equation*}
	\max_{s\geq s_0, r\leq r_0} \ell(r,s) = \ell(r_0,s_0)
	\eqfs
\end{equation*}
Distinguish the two cases $a\geq c$ and $a\leq c$.\\
\underline{Case 1: $a\geq c$.} For fixed $r\in[-1,1]$, set $s = s_{\ms{min}}(r) = (r-1) \frac ac+1$, cf \autoref{lmm:srbound}. Define
\begin{align*}
	f(r) &:= \ell(r,s_{\ms{min}}(r)) 
	\\&= a^{2\alpha}-c^{2\alpha}-\br{a^2-2rab+b^2}^\alpha+\br{c^2-2rab+2ab-2cb+b^2}^\alpha
	\eqfs
\end{align*}
Then
\begin{equation*}
	\frac{f\pr(r)}{2ab\alpha} = \br{a^2-2rab+b^2}^{\alpha-1} - \br{c^2-2rab+2ab-2cb+b^2}^{\alpha-1}
	\eqfs
\end{equation*}
\underline{Case 1.1: $a^2 \leq c^2+2ab-2cb$.} Then
\begin{align*}
	a^2-2rab+b^2 &\leq c^2-2rab+2ab-2cb+b^2\eqcm\\
	\br{a^2-2rab+b^2}^{\alpha-1} &\geq \br{c^2-2rab+2ab-2cb+b^2}^{\alpha-1}
	\eqfs
\end{align*}
Thus, $f\pr(r)\geq0$. In this case, we need to show
\begin{equation*}
	a^{2\alpha}-c^{2\alpha}-|a-b|^{2\alpha}+|c-b|^{2\alpha} = f(1) \leq 8\alpha 2^{-2\alpha} b (a-c)^{2\alpha-1}
	\eqfs
\end{equation*}
\underline{Case 1.2: $a^2 \geq c^2+2ab-2cb$.} Then
\begin{align*}
	a^2-2rab+b^2 &\geq c^2-2rab+2ab-2cb+b^2\eqcm\\
	\br{a^2-2rab+b^2}^{\alpha-1} &\leq \br{c^2-2rab+2ab-2cb+b^2}^{\alpha-1}\eqfs
\end{align*}
Thus, $f\pr(r)\leq0$. The relevant values are $r = r_{\ms{min}} = 1-2\frac ca$, with $s = s_{\ms{min}}(r) = -1$. In this case, we need to show
\begin{equation*}
	a^{2\alpha}-c^{2\alpha}-\br{(a-b)^2+4cb}^{\alpha}+(c+b)^{2\alpha} = f(r_{\ms{min}}) \leq 8\alpha 2^{-2\alpha} b (a-c)^{2\alpha-1}
	\eqfs
\end{equation*}
\underline{Case 2: $a\leq c$.}  For fixed $r\in[-1,1]$, set $s = s_{\ms{min}}(r) = (r+1) \frac ac-1$. Define
\begin{align*}
	f(r) &:= \ell(r,s_{\ms{min}}(r)) 
	\\&= a^{2\alpha}-c^{2\alpha}-\br{a^2-2rab+b^2}^\alpha+\br{c^2-2rab-2ab+2cb+b^2}^\alpha
	\eqfs
\end{align*}
Then 
\begin{equation*}
	\frac{f\pr(r)}{2ab\alpha} = 
	\br{a^2-2rab+b^2}^{\alpha-1} - \br{c^2-2rab-2ab+2cb+b^2}^{\alpha-1}
	\eqfs
\end{equation*}
\underline{Case 2.1: $a^2 \leq c^2-2ab+2cb$.} Then
\begin{align*}
	a^2-2rab+b^2 &\leq c^2-2rab-2ab+2cb+b^2\eqcm\\
	\br{a^2-2rab+b^2}^{\alpha-1} &\geq \br{c^2-2rab-2ab+2cb+b^2}^{\alpha-1}\eqfs
\end{align*}
Thus, $f\pr(r)\geq0$. The critical value is $r = r_{\ms{max}} = 1$, with $s = s_{\ms{min}}(r) = 2\frac ac-1$. In this case, we need to show
\begin{equation*}
	a^{2\alpha}-c^{2\alpha}-|a-b|^{2\alpha}+\br{(c+b)^2-4ab}^{\alpha} = f(1) \leq 8\alpha 2^{-2\alpha} b (c-a)^{2\alpha-1}
	\eqfs
\end{equation*}
\underline{Case 2.2: $a^2 \geq c^2-2ab+2cb$.} This cannot happen for $a\leq c$.
\begin{remark}
Assume $a\geq c$. Then
\begin{align*}
	a^2 \geq c^2+2ab-2cb
	\quad\Leftrightarrow\quad
	a^2-c^2 \geq 2b(a-c)
	\quad\Leftrightarrow\quad
	a+c \geq 2b
	\eqfs
\end{align*}
\end{remark}
\subsubsection{Outline}
\begin{remark}[What we need to show]\label{rmk:outline}
	Define
	\begin{align*}
		F(r,s) 
		:= 
		&a^{2\alpha}-c^{2\alpha}-\br{a^2-2rab+b^2}^\alpha+\br{c^2-2scb+b^2}^\alpha 
		\\&-
		8 \alpha 2^{-2\alpha} b (ra - sc)^{2\alpha-1}
		\eqfs
	\end{align*}
	\begin{enumerate}[label=(\roman*)]
	\item 
		For the case $a\geq c$, we need $F(1,s) \leq 0$ for all $s\in[-1,1]$ (cf \autoref{lmm:ddr}, includes  case 1.1 of section \ref{ssec:acgreasc}) and $F(1-2\frac ca, -1) \leq 0$ (case 1.2 of section \ref{ssec:acgreasc}).
		Note $r= 1-2\frac ca, s=-1$ means $ra-sc = a-c$.
	\item	
		For the case $a \leq c$, we need $F(1,s) \leq 0$ for all $s\in[-1,2\frac ac-1]$  (cf \autoref{lmm:ddr}, includes case 2.1 of section \ref{ssec:acgreasc}).
		Note $r= 1, s=2\frac ac-1$ means $ra-sc = c-a$.
	\end{enumerate}
	The part $F(1-2\frac ca, -1) \leq 0$ for $a\geq c$ is discussed in section \ref{ssec:rascleqac}.
	The different case for $F(1,s) \leq 0$ ($a\leq c$ and $a\geq c$) are covered in the following way:
	\begin{enumerate}[label=(\alph*)]
	\item 
	$b \geq 2sc$: \autoref{lmm:rascMergingLemma} (Merging Lemma) + \autoref{lmm:tightpowerbound} (Tight Power Bound)
	\item
	$b \leq 2sc$ and $sc \leq a-b$: \autoref{lmm:merging:asc} (Merging Lemma) + \autoref{lmm:tightpowerbound} (Tight Power Bound)
	\item
	$b \leq 2sc$ and $sc \geq a-b$ and $a \geq c$: \autoref{lmm:rascgeqacandageqc} 
	\item 
	$b \leq 2sc$ and $sc \geq a-b$ and $a \leq c$, $sc \leq 2a-c$ and $b \leq 2a-2sc$: \autoref{lmm:bleqasc}
	\item
	$b \leq 2sc$ and $sc \geq a-b$ and $a \leq c$, $sc \leq 2a-c$ and $b \geq 2a-2sc$ and $a \leq b$ : \autoref{lmm:bgeqascandaleqb}
	\item
	$b \leq 2sc$ and $sc \geq a-b$ and $a \leq c$, $sc \leq 2a-c$ and $b \geq 2a-2sc$ and $a \geq b$ : \autoref{lmm:bgeqascandageqb}
	\end{enumerate}
\end{remark}
The proofs consist of distinguishing many different cases and applying simple analysis methods in each case.
Nonetheless, finding the poofs is often quite hard, as the inequalities are usually very tight and the right steps necessary for the proof are hard to guess. 

As intermediate steps we can, in some cases, use two lemmas: the Tight Power Bound, see section \ref{ssec:tightpowerbound}, and the Merging Lemma, see \ref{ssec:merginglemma}.
The remaining cases that cannot be solved via Tight Power Bound and Merging Lemma will be discussed in sections \ref{ssec:rascgeqac} and \ref{ssec:rascleqac}.
\subsection{Tight Power Bound}\label{ssec:tightpowerbound}
Following lemma gives one very useful inequality in three different forms. It gives a hint to why the power $\dots^{2\alpha-1}$ comes up in the RHS of \autoref{con:ana}.
\begin{lemma}[Tight Power Bound]\label{lmm:tightpowerbound}
	Let $x,y\geq0$.
	\begin{enumerate}[label=\environmentEnumerateLabel]
	\item 
		If  $a\in[1,2]$, $x\geq y$, then
		\begin{equation*}
			2^a x^{a-1}y \quad\leq\quad (x+y)^a-(x-y)^a \quad\leq\quad 2a x^{a-1}y \eqfs
		\end{equation*}
	\item 
		If  $a\in[1,2]$, then
		\begin{equation*}
			(x+y)^a-|x-y|^a \quad\leq\quad 2a\min(xy^{a-1},x^{a-1}y)\eqfs
		\end{equation*}
	\item 
		If  $a\in[1,2]$, $x\geq y$, then
		\begin{equation*}
			(x+y)^{a-1} (x-y) \quad\leq\quad x^a - y^a \quad\leq\quad a (x-y) \br{\frac{x+y}{2}}^{a-1}
			\eqfs
		\end{equation*}
	\end{enumerate}
\end{lemma}
Note that this result is slightly stronger than the application of the mean value theorem to the function $x \mapsto x^a$, which yields $x^a - y^a \leq a (x-y) z^{a-1}$ for all $x \geq y \geq 0$ and $a > 0$, where $z \in [y,x]$.
\begin{proof}
	Assume $x\geq y$.
	Set $z = \frac yx \in[0,1]$. Define
	\begin{equation*}
		f(z) = \frac{(1+z)^a-(1-z)^a}{z}
		\eqfs
	\end{equation*}
	If we can show $f(z) \leq 2a$, then
	\begin{align*}
		&&(1+z)^a-(1-z)^a &\leq 2a z\\
		\Rightarrow && (x+zx)^a-(x-zx)^a &\leq 2a x^a z\\
		\Rightarrow && (x+y)^a-(x-y)^a &\leq 2a x^{a-1} y
		\eqfs
	\end{align*}
	We have 
	\begin{equation*}
		f\pr(z) = \frac {g(z)}{z^2}
		\eqcm
	\end{equation*}
	where
	\begin{equation*}
		g(z) = az \br{(1+z)^{a-1}+(1-z)^{a-1}}-\br{(1+z)^{a}-(1-z)^{a}}
		\eqfs
	\end{equation*}
	We have
	\begin{equation*}
		g\pr(z) = az(a-1)\br{(1+z)^{a-2}-(1-z)^{a-2}} \leq 0\eqfs
	\end{equation*}
	Thus, $g(z)\leq g(0) = 0$.
	Thus, $f\pr(z)\leq0$.
	Thus, for all $z_0 \in [0,1]$,
	\begin{equation*}
		f(z_0) 
		\leq 
		\lim_{z\searrow 0}f(z) 
		\stackrel{\ms{L'H}}{=} 
		\lim_{z\searrow 0}\frac{a(1+z)^{a-1}+a(1-z)^{a-1}}{1} 
		= 
		2a
		\eqcm
	\end{equation*}
	where $\ms{L'H}$ indicates the use of L'Hospital's rule. 
	Furthermore, $f(z_0) \geq f(1) = 2^a$, which implies the lower bound. 
	This finishes the proof for (i). The other parts follow immediately.
\end{proof}
\subsection{Merging Lemma}\label{ssec:merginglemma}
In many cases (i.e., with additional assumption on $a,b,c,r$ or $s$), we prove the inequality of \autoref{con:ana} by applying first a merging lemma to the LHS to reduce the four summands to two summands of a specific form. Then we apply the Tight Power Bound. The Merging Lemma is discussed in this section.
\subsubsection{Simple Merging Lemma}
\begin{lemma}[Simple Merging Lemma]\label{lmm:merging:simple}
	Let $\alpha\in[\frac12,1]$, $b\geq0$, $a,c\in\R$.
	Then
	\begin{equation*}
		|a|^{2\alpha}-|c|^{2\alpha}-|a-b|^{2\alpha}+|c-b|^{2\alpha} \leq 
		2^{1-2\alpha} \bigg(
							(a-c + b)^{2\alpha}  -  |a-c - b|^{2\alpha}
						\bigg) \indOf{a-c>0}
		\eqfs
	\end{equation*}
\end{lemma}
\begin{proof}
	For $\tilde \alpha\geq 1$, the function $\R\to\R,\, x\mapsto |x|^{\tilde\alpha}-|x-1|^{\tilde\alpha}$ is increasing.
	We have $2\alpha\geq1$. Thus, if $a\leq c$, then
	\begin{equation*}
		|a|^{2\alpha}-|a-b|^{2\alpha}\leq |c|^{2\alpha}-|c-b|^{2\alpha} 
		\eqfs
	\end{equation*}
	This shows the inequality for the case $a\leq c$.
	
	Set $q:= a-b$ and define
	\begin{equation*}
		g(b) :=  |q+b|^{2\alpha}-|c|^{2\alpha}-|q|^{2\alpha}+|c-b|^{2\alpha} - 2\br{\br{\frac{q-c}{2}+b}^{2\alpha}-\br{\frac{q-c}{2}}^{2\alpha}}
		\eqfs
	\end{equation*}
	We have $g(0)=0$
	and
	\begin{equation*}
		\frac{g\pr(b)}{2\alpha} = \sgn(q+b)|q+b|^{2\alpha-1}-\sgn(c-b)|c-b|^{2\alpha-1} -
						2\br{\frac{q-c}{2}+b}^{2\alpha-1}
		\eqfs
	\end{equation*}
	\underline{Case 1: $\sgn(q+b) = +1$, $\sgn(c-b)=+1$}:
		\begin{equation*}
			\frac{g\pr(b)}{2\alpha} = (q+b)^{2\alpha-1}-(c-b)^{2\alpha-1} -
							2\br{\frac{q-c}{2}+b}^{2\alpha-1}
			\eqcm
		\end{equation*}
		\begin{equation*}
			(q+b)^{2\alpha-1}-(c-b)^{2\alpha-1} \leq \br{q+b-(c-b)}^{2\alpha-1}\leq 2\br{\frac{q-c}{2}+b}^{2\alpha-1}
			\eqfs
		\end{equation*}
	\underline{Case 2: $\sgn(q+b) = -1$, $\sgn(c-b)=-1$}:
		\begin{equation*}
			\frac{g\pr(b)}{2\alpha} = (b-c)^{2\alpha-1}-(-q-b)^{2\alpha-1}
							-2\br{\frac{q-c}{2}+b}^{2\alpha-1}
			\eqcm
		\end{equation*}
		\begin{equation*}
			(b-c)^{2\alpha-1}-(-q-b)^{2\alpha-1} \leq \br{b-c-(-q-b)}^{2\alpha-1}\leq 2\br{\frac{q-c}{2}+b}^{2\alpha-1}
			\eqfs
		\end{equation*}
	\underline{Case 3: $\sgn(q+b) = +1$, $\sgn(c-b)=-1$}:
		\begin{equation*}
			\frac{g\pr(b)}{2\alpha} = (q+b)^{2\alpha-1}+(b-c)^{2\alpha-1} -
									2\br{\frac{q-c}{2}+b}^{2\alpha-1}
									\eqcm
		\end{equation*}
		\begin{equation*}
			(q+b)^{2\alpha-1}+(b-c)^{2\alpha-1} \leq 2\br{\frac{q-c}{2}+b}^{2\alpha-1}
			\eqfs
		\end{equation*}
	\underline{Case 4: $\sgn(q+b) = -1$, $\sgn(c-b)=+1$}:
		\begin{equation*}
			\frac{g\pr(b)}{2\alpha} = -(-q-b)^{2\alpha-1}-(c-b)^{2\alpha-1} -
									2\br{\frac{q-c}{2}+b}^{2\alpha-1}
									\eqcm
		\end{equation*}
		\begin{equation*}
			-(-q-b)^{2\alpha-1}-(c-b)^{2\alpha-1} \leq 0
			\eqfs
		\end{equation*}
	\underline{Together:}
	In every case, we have $g\pr(b) \leq 0$ and $g(0)=0$. Thus, $g(b) \leq 0$.
\end{proof}
\subsubsection{$ra-sc$--Merging Lemma}
\begin{lemma}
	Let $\alpha\in[0,1]$.
	\theoremContentInNewLine
	\begin{enumerate}[label=\environmentEnumerateLabel]
	\item 
		Let $b,c\geq0$, $s\in[-1,1]$.
		Assume $2sc \leq b$.
		Then
		\begin{equation*}
			-c^{2\alpha}+(c^2 - 2 s c b+b^2)^\alpha \leq 
			-|sc|^{2\alpha}+|sc-b|^{2\alpha}
			\eqfs
		\end{equation*}
	\item
		Let $a,b\geq0$, $r\in[-1,1]$.
		Assume $2ra \geq b$.
		Then
		\begin{equation*}
			a^{2\alpha}-(a^2 - 2 r a b+b^2)^\alpha \leq 
			|ra|^{2\alpha}-|ra-b|^{2\alpha}
			\eqfs
		\end{equation*}
	\end{enumerate}
\end{lemma}
\begin{proof}
	The function $t \mapsto (t+1)^\alpha-t^\alpha$, $t\geq0$ is non-increasing for all $\alpha\in[0,1]$.
	We have $0\leq s^2c^2\leq c^2$ and $x := -2 s c b+b^2 \geq 0$. Thus,
	\begin{equation*}
		\br{c^2+x}^\alpha-\br{c^2}^\alpha \leq \br{(sc)^2+x}^\alpha-\br{(sc)^2}^\alpha
		\eqfs
	\end{equation*}
	Thus, 
	\begin{equation*}
		(c^2-2 s c b+b^2)^\alpha -c^{2\alpha}\leq |-sc+b|^{2\alpha}-|sc|^{2\alpha}
		\eqfs
	\end{equation*}
	For the second part, set $x:=2rab-b^2$, $y := a^2-x\geq0$, $\tilde y := |ra|^2-x\geq 0$. The condition $2ra\geq b$ implies $x\geq 0$.
	Thus, as before,
	\begin{equation*}
		a^{2\alpha}-(a^2 - 2 r a b+b^2)^\alpha 
		=
		(y+x)^\alpha-y^\alpha 
		\leq 
		(\tilde y+x)^\alpha-\tilde y^\alpha 
		=
		|ra|^{2\alpha}-|ra-b|^{2\alpha}
		\eqfs
	\end{equation*}
\end{proof}
\begin{lemma}[$ra-sc$--Merging Lemma]\label{lmm:rascMergingLemma}
	Let $\alpha\in[\frac12,1]$.
	Let $a,b,c\geq0$, $r,s\in[-1,1]$.
	\theoremContentInNewLine
	\begin{enumerate}[label=\environmentEnumerateLabel]
	\item 
		Assume $2ra \geq b$, $s\in\cb{-1,1}$.
		Then
		\begin{align*}
			&a^{2\alpha}-c^{2\alpha}-(a^2-2rab+b^2)^{\alpha} + (c^2 - 2 s c b + b^2)^\alpha
			\\& \leq 2^{1-2\alpha} \bigg((ra - s c + b)^{2\alpha}-\abs{ra - s c - b}^{2\alpha}\bigg) \ind_{ra-sc>0}
			\eqfs
		\end{align*}	
	\item 
		Assume $b \geq 2sc$, $r\in\cb{-1,1}$.
		Then
		\begin{align*}
			&a^{2\alpha}-c^{2\alpha}-(a^2-2rab+b^2)^{\alpha} + (c^2 - 2 s c b + b^2)^\alpha 
			\\&\leq 2^{1-2\alpha} \bigg((ra - s c + b)^{2\alpha}-\abs{ra - s c - b}^{2\alpha}\bigg)\ind_{ra-sc>0}
			\eqfs
		\end{align*}	
	\item 
		Assume $2ra \geq b \geq 2sc$.
		Then
		\begin{align*}
			&a^{2\alpha}-c^{2\alpha}-(a^2-2rab+b^2)^{\alpha} + (c^2 - 2 s c b + b^2)^\alpha 
			\\&\leq 2^{1-2\alpha} \br{(ra - s c + b)^{2\alpha}-\abs{ra - s c - b}^{2\alpha}}\ind_{ra-sc>0}
			\eqfs
		\end{align*}	
	\end{enumerate}
\end{lemma}
\begin{proof}
	The lemma above and the simple merging lemma imply
	\begin{align*} 
		&a^{2\alpha}-c^{2\alpha}-(a^2-2rab+b^2)^{\alpha} + (c^2 - 2 s c b + b^2)^\alpha 
		\\&\leq 
		(ra)^{2\alpha}-(sc)^{2\alpha}-(ra-b)^{2\alpha} + (sc-b)^{2\alpha}
		\\&\leq 
		2^{1-2\alpha} \br{(ra - s c + b)^{2\alpha}-\abs{ra - s c - b}^{2\alpha}}\ind_{ra-sc>0}
		\eqfs
	\end{align*}	
\end{proof}
%
%
%
%
%
\subsubsection{$a-sc$--Merging Lemma}
\autoref{lmm:rascMergingLemma} covers the case $\frac12 b \geq sc$.
The following lemma covers  $\frac12 b \leq sc$  under the additional restriction $sc \leq a-b$.
\begin{lemma}[$a-sc$--Merging Lemma]\label{lmm:merging:asc}
	Let $\alpha\in[\frac12,1]$.
	Let $a,b,c\geq0$, $s\in[-1,1]$.
	Assume $\frac12 b \leq sc\leq a-b$.
	Then
	\begin{equation*}
		a^{2\alpha}-c^{2\alpha}-(a-b)^{2\alpha}+(c^2 - 2 s c b + b^2)^\alpha \leq 2^{1-2\alpha} \bigg(
			(a - s c + b)^{2\alpha}  -  (a - s c - b)^{2\alpha}
		\bigg)
		\eqfs
	\end{equation*}
\end{lemma}
\begin{proof}
	Set $\delta:=a-b$.
	Define
	\begin{align*}
		f(\delta) 
		= &\,(\delta+b)^{2\alpha}-c^{2\alpha}-\delta^{2\alpha}+(c^2 - 2 s c b + b^2)^\alpha -
		\\&2 \br{
					\br{\frac{\delta - s c}{2} + b}^{2\alpha}  -  \br{\frac{\delta - s c}{2}}^{2\alpha}
				}
		\eqfs
	\end{align*}
	Then
	\begin{equation*}
		\frac{f\pr(\delta)}{2\alpha} =
		(\delta+b)^{2\alpha-1}-\delta^{2\alpha-1} 
		- \br{\frac{\delta - s c}{2} + b}^{2\alpha-1}
		+ \br{\frac{\delta - s c}{2}}^{2\alpha-1}
		\eqfs
	\end{equation*}
	We have
	\begin{align*}
		\delta+b &\geq  \delta\eqcm\\
		\frac{\delta - s c}{2} &\leq  \frac{\delta - s c}{2} + b\eqcm\\
		(\delta+b) + \frac{\delta - s c}{2} &= \delta +\br{\frac{\delta - s c}{2} + b}\eqfs
	\end{align*}
	Thus,
	\begin{equation*}
		(\delta+b)^{2\alpha-1}
		+ \br{\frac{\delta - s c}{2}}^{2\alpha-1}
		\leq
		\delta^{2\alpha-1} 
		+\br{\frac{\delta - s c}{2} + b}^{2\alpha-1}
		\eqfs
	\end{equation*}
	Thus, $f\pr(\delta) \leq 0$.\\ 
	The next lemma shows $f(sc) \leq 0$. Thus, $f(\delta) \leq 0$ for all $\delta \geq sc$.
\end{proof}
\begin{lemma}
	Let $x,a,b,c \geq 0$.
	Assume $b \leq 2x$, $x+b \geq c$, $x\leq c$.
	Then
	\begin{equation*}
		(x+b)^{2\alpha}+(c^2 - 2 x b + b^2)^\alpha \leq c^{2\alpha}+x^{2\alpha} + 2 b^{2\alpha} 
		\eqfs
	\end{equation*}	
\end{lemma}
\begin{proof}
	Define
	\begin{align*}
		g(x) &:=(x+b)^{2\alpha}+(c^2 - 2 x b + b^2)^\alpha - c^{2\alpha}-x^{2\alpha} - 2 b^{2\alpha} \eqcm
		\\
		h(x) &:= \frac{g\pr(x)}{2\alpha}
		=
		(x+b)^{2\alpha-1}-x^{2\alpha-1}-b(c^2 - 2 x b + b^2)^{\alpha -1}
		\eqfs
	\end{align*}
	We have
	\begin{equation*}
		h\pr(x)
		=
		(2\alpha-1)(x+b)^{2\alpha-2}-(2\alpha-1)x^{2\alpha-2}+2(\alpha -1) b^2(c^2 - 2 x b + b^2)^{\alpha -2}
		\eqfs
	\end{equation*}
	As $2\alpha-2 \leq 0$ and $2\alpha-1 \geq 0, (2\alpha-1)(x+b)^{2\alpha-2}-(2\alpha-1)x^{2\alpha-2}\leq 0$.
	As $\alpha-1\leq0$, $2(\alpha -1) b^2(c^2 - 2 x b + b^2)^{\alpha -2} \leq 0$.
	Thus, $h\pr(x)\leq0$. \\
	We have $x\geq x_{\ms{min}} := \max(\frac b2, c-b)$. For checking $h(x_{\ms{min}}) \leq 0$ and $g(x_{\ms{min}}) \leq 0$, we distinguish $x_{\ms{min}} = c-b$ and $x_{\ms{min}} = \frac b2$.\\
	\underline{Case 1, $c-b\leq\frac b2$:}\\
	If $c-b\leq\frac b2$, then $c\leq \frac32 b \leq (1+\sqrt{3})b$, $c^2 -2cb-2b^2\leq0$, and
	\begin{align*}
		h\brOf{c-b} 
		&= 
		c^{2\alpha-1}-(c-b)^{2\alpha-1}-b(c^2 - 2 (c-b) b + b^2)^{\alpha -1}
		\\&= 
		c^{2\alpha-1}-(c-b)^{2\alpha-1}-b(c^2 - 2 c b - b^2)^{\alpha -1}
		\\&\leq 
		b^{2\alpha-1}-b(c^2 - 2 c b - b^2)^{\alpha -1}
		\\&\leq 
		b \br{\br{b^2}^{\alpha-1}-\br{c^2 - 2 c b - b^2}^{\alpha -1}}
	\end{align*}
	And, thus, $h\brOf{c-b}  \leq 0$ as $c^2 -2cb-2b^2\leq0$.
	Furthermore,
	\begin{align*}
		g\brOf{c-b} 
		&=
		-(c-b)^{2\alpha}+(c^2 - 2 c b-b^2)^\alpha 
		- 2 b^{2\alpha} 
		\\&=
		-(c-b)^{2\alpha}+(c^2 - 2 c b-2b^2+b^2)^\alpha 
				- 2 b^{2\alpha} 
		\\&\leq 
		-(c-b)^{2\alpha}+b^{2\alpha }
		- 2 b^{2\alpha} 
		\\&=
		-(c-b)^{2\alpha}
		- b^{2\alpha} 
		\\&\leq
		0
		\eqfs
	\end{align*}
	Thus, $g(x)\leq 0$ for all valid $x$.\\
	\underline{Case 2, $c-b\geq\frac b2$:}\\
	If $c-b\geq\frac b2$, then $c\geq b$ and
	\begin{align*}
		h\brOf{\frac b2} 
		&= 
		\br{\frac32 b}^{2\alpha-1}-\br{\frac12 b}^{2\alpha-1}-b(c^2)^{\alpha -1}
		\\&=
		\br{\br{\frac32}^{2\alpha-1}-\br{\frac12}^{2\alpha-1}} b^{2\alpha-1}-b(c^2)^{\alpha -1}
		\\&\leq 
		b^{2\alpha-1}-b(c^2)^{\alpha -1}
		\\&\leq 
		b^{2\alpha-1}-b(b^2)^{\alpha -1}
		\\&\leq
		0\eqcm
	\end{align*}
	\begin{align*}
		g\brOf{\frac b2} 
		&=
		\br{\frac32b}^{2\alpha}-c^{2\alpha}-\br{\frac12b}^{2\alpha}+c^{2\alpha}
				- 2 b^{2\alpha} 
		\\&=
		\br{\br{\frac94}^\alpha-\br{\frac14}^\alpha-2} b^{2\alpha}
		\\&\leq 
		0
		\eqfs
	\end{align*}
	Thus, $g(x)\leq 0$ for all valid $x$.
\end{proof}
\subsection{Application of Tight Power Bound and Merging Lemma} \label{ssec:apply}
Whenever a Merging Lemma holds, we apply it as a first step and then use the Tight Power Bound, \autoref{lmm:tightpowerbound}, to obtain
\begin{align*}
	&a^{2\alpha}-c^{2\alpha}-(a^2-2rab+b^2)^{\alpha} + (c^2 - 2 s c b + b^2)^\alpha 
	\\&\leq 
	2^{1-2\alpha} \br{(ra - s c + b)^{2\alpha}-\abs{ra - s c - b}^{2\alpha}}
	\\& \leq 
	4\alpha 2^{1-2\alpha} (ra - s c)^{2\alpha-1} b
	\eqfs
\end{align*}	
In particular, we have finished the proof of \autoref{con:ana} in following cases:
\begin{itemize}
\item $ra\geq sc$ and $s, r \in \cb{-1,1}$: \autoref{lmm:merging:simple},
\item $2ra \geq b$ and $s\in\cb{-1,1}$; or $b \geq 2sc$ and $r\in\cb{-1,1}$; or  $2ra \geq b \geq 2sc$: \autoref{lmm:rascMergingLemma}, 
\item $\frac12 b \leq sc\leq a-b$ and $r=1$: \autoref{lmm:merging:asc}.
\end{itemize}
\subsection{The Case $ra-sc\geq |a-c|$} \label{ssec:rascgeqac}
\subsubsection{The Case $a\geq c$}
First we prove to simple lemmas, before we solve this case.
\begin{lemma}\label{lmm:f1}
	Let $a\geq b\geq 0 $, $d \geq c \geq 0$, and $\alpha\in[0,1]$.
	Then
	\begin{equation*}
		a^\alpha-b^\alpha-c^\alpha+d^\alpha \leq 2^{1-\alpha}(a-b-c+d)^{\alpha}
		\eqfs
	\end{equation*}
\end{lemma}
\begin{proof}
	As $a\geq b$, $d\geq c$, $\alpha\leq 1$,
	\begin{equation*}
		a^\alpha-b^\alpha + d^\alpha-c^\alpha \leq (a-b)^\alpha + (d-c)^\alpha \eqfs
	\end{equation*}
	Furthermore, by concavity of $x \mapsto x^\alpha$,
	\begin{equation*}
		(a-b)^\alpha + (d-c)^\alpha
		\leq
		2^{1-\alpha}(a-b+d-c)^\alpha
		\eqfs
	\end{equation*}
\end{proof}
\begin{lemma}\label{lmm:f2}
	Let $a\geq b\geq c \geq d\geq 0$, $a+d\geq b+c$, and $\alpha\in[0,1]$.
	Then
	\begin{equation*}
		a^\alpha-b^\alpha-c^\alpha+d^\alpha \leq (a-b-c+d)^{\alpha}
		\eqfs
	\end{equation*}
\end{lemma}
\begin{proof}
	Define $f(x,y) = x^\alpha+y^\alpha-(x+y)^\alpha$ for $x,y\geq0$.
	Then $\partial_xf(x,y) = \alpha(x^{\alpha-1}-(x+y)^{\alpha-1}) \geq 0$ and similarly $\partial_yf(x,y) \geq 0$.
	Set $\delta := a-b$ and $\epsilon := c-d$. The assumptions ensure $\delta\geq\epsilon\geq0$.
	Then, 
	\begin{equation*}
		f(b,\delta) \geq f(b,\epsilon) \geq f(d,\epsilon)
		\eqfs
	\end{equation*}
	Thus,
	\begin{align*}
		0 
		&\geq 
		f(d,\epsilon) - f(b,\delta) 
		\\&= 
		d^\alpha+\epsilon^\alpha-(d+\epsilon)^\alpha
		-b^\alpha-\delta^\alpha+(b+\delta)^\alpha
		\\&=
		d^\alpha+\epsilon^\alpha-c^\alpha
		-b^\alpha-\delta^\alpha+a^\alpha
		\eqfs
	\end{align*}
	With this we get
	\begin{align*}
		d^\alpha-c^\alpha-b^\alpha+a^\alpha
		&\leq
		\delta^\alpha-\epsilon^\alpha
		\\&\leq 
		(\delta-\epsilon)^\alpha
		\\&=
		(a-b-c+d)^\alpha
		\eqfs
	\end{align*}
\end{proof}%
For $a\geq c$, the remaining case is solved by following lemma.
\begin{lemma}\label{lmm:rascgeqacandageqc}
	Let $\alpha\in[0,1]$.
	Let $a,b,c\geq0$, $s\in[-1,1]$.
	Assume $\frac12 b \leq sc$, $sc \geq a-b$, and $a\geq c$.
	Then
	\begin{equation*}
		a^{2\alpha}-c^{2\alpha}-(a-b)^{2\alpha}+(c^2 - 2 s c b + b^2)^\alpha 
		\leq 
		2 b^\alpha (a-sc)^{\alpha}
		\leq 
		2 b (a-sc)^{2\alpha-1}
		\eqfs
	\end{equation*}
\end{lemma}
\begin{proof}
	Because $a \geq c$ and $\frac12 b \leq sc$, we have $a-b \geq a-2sc \geq a-2c \geq -c$. Hence, $a^2 \geq \max(c^2, (a-b)^2)$.
	Thus, applying either \autoref{lmm:f1} (if $c^2 - 2 s c b + b^2$ is larger then either $c^2$ or $(a-b)^2$) or \autoref{lmm:f2} yields
	\begin{align*}
		&a^{2\alpha}-c^{2\alpha}-(a-b)^{2\alpha}+(c^2 - 2 s c b + b^2)^\alpha 
		\\&\leq 
		2^{1-\alpha}
		\br{a^{2}-c^{2}-(a-b)^{2}+(c^2 - 2 s c b + b^2)}^\alpha 
		\\&=
		2 b^\alpha (a-sc)^{\alpha}
		\eqfs
	\end{align*}
	The condition $0 \leq a-sc \leq b$ implies
	\begin{equation*}
		2 b^\alpha (a-sc)^{\alpha}
		\leq 
		2 b (a-sc)^{2\alpha-1}
		\eqfs
	\end{equation*}
\end{proof}
\subsubsection{The Case $a\leq c$}
For the case $c\geq a$, we only need $ra - sc \geq c-a$	(for $r=1$), i.e., $sc \leq 2a - c$.
Assume $c\geq a$, $sc \geq a-b$, $\frac12b \leq sc$, and $sc \leq 2a - c$.
Then 
\begin{align*}
	c^2 &\geq c^2-2scb+b^2 \geq (a-b)^2\\
	c^2 &\geq a^2 \geq (a-b)^2
\end{align*}
We distinguish $\frac12 b \leq a-sc$ and $\frac12 b \geq a-sc$.
\begin{lemma}[$\frac12 b\leq a-sc$]\label{lmm:bleqasc}
	Let $\alpha\in[0,1]$.
	Let $a,b,c\geq0$, $s\in[-1,1]$.
	Assume $\frac12 b \leq sc$, $sc \geq a-b$, $c\geq a$, $sc \leq 2a - c$, and $\frac12 b\leq a-sc$.
	Then
	\begin{equation*}
		a^{2\alpha}-c^{2\alpha}-(a-b)^{2\alpha}+(c^2 - 2 s c b + b^2)^\alpha 
		\leq  
		2 b^\alpha (a-sc)^{\alpha}
		\eqfs
	\end{equation*}
\end{lemma}
\begin{proof}
	The conditions imply
	\begin{equation*}
		\max\brOf{\frac 12 b,\, a-b} \leq sc \leq \min\brOf{a-\frac12b,\, 2a-c}
		\eqfs
	\end{equation*}
	In particular, $\frac 12 b \leq a-\frac12b$, and $a-b \leq 2a-c$. Thus, $a+b \geq c \geq a \geq b$.
	
	Fix $a,b,c \geq 0$.
	Assume $c \geq a$. Let $x \in[0,a)$. Then $a-x>0$ and $c^2 - 2 x b + b^2 >0$.
	Define 
	\begin{equation*}
		f(x) := a^{2\alpha}-c^{2\alpha}-(a-b)^{2\alpha}+(c^2 - 2 x b + b^2)^\alpha  - 2 b^\alpha (a-x)^{\alpha}
		\eqfs
	\end{equation*}
	We have
	\begin{equation*}
		\frac{f\pr(x) }{2 \alpha b} = -(c^2 - 2 x b + b^2)^{\alpha-1}  + b^{\alpha-1} (a-x)^{\alpha-1}
		\eqfs
	\end{equation*}
	Furthermore,
	\begin{align*}
		&a-c\leq 0 \leq (c-b)^2
		\\\Rightarrow\qquad&
		2ab-2cb \leq 0 \leq c^2+b^2-2cb
		\\\Rightarrow\qquad&
		xb \leq 2ab-cb \leq c^2+b^2-cb \leq c^2+b^2 -ab
		\\\Rightarrow\qquad&
		ab -xb \leq c^2-2xb + b^2
		\\\Rightarrow\qquad&
		b^{\alpha-1} (a-x)^{\alpha-1} \geq (c^2 - 2 x b + b^2)^{\alpha-1}
		\\\Rightarrow\qquad&
		f\pr(x) \geq 0\eqfs
	\end{align*}
	We define
	\begin{equation*}
		g(c) := f(a-\frac12b) = a^{2\alpha}-c^{2\alpha}-(a-b)^{2\alpha}+(c^2 - 2 a b + 2 b^2)^\alpha  - 2^{1-\alpha} b^{2\alpha}
		\eqfs
	\end{equation*}
	Thus,
	\begin{equation*}
		\frac{g\pr(c)}{2\alpha c} = -c^{2\alpha-2}+\br{c^2 - 2 a b + 2 b^2}^{\alpha-1} \geq 0
		\eqfs
	\end{equation*}
	Define
	\begin{align*}
		h (a,b) &:= 
		g(a+b) 
		\\&= 
		a^{2\alpha}-(a+b)^{2\alpha}-(a-b)^{2\alpha}+(a^2  + 3 b^2)^\alpha  - 2^{1-\alpha} b^{2\alpha}
		\\&=
		a^{2\alpha}-(a+b)^{2\alpha}-(a-b)^{2\alpha}+(a^2  + 3 b^2)^\alpha  - 2\br{\frac{b^2}{2}}^{\alpha}
		\eqfs
	\end{align*}
	The next lemma shows $h(a,b) \leq 0$ for $a\geq b$. Thus, $g(c) \leq 0$. Thus, $f(x) \leq 0$.
\end{proof}
\begin{lemma}
	Let $\alpha\in[0, 1]$, $a,b\geq0$. Assume $a\geq b$.
	Then
	\begin{equation*}
		a^{2\alpha}+(a^2  + 3 b^2)^\alpha \leq (a+b)^{2\alpha} + (a-b)^{2\alpha} + 2\br{\frac{b^2}{2}}^{\alpha}
		\eqfs
	\end{equation*}
\end{lemma}
\begin{proof}
	Set $x =\frac{b}{a} \in [0,1]$. Define
	\begin{equation*}
		f(\alpha,x) = 1+\br{1+3x^2}^\alpha-\br{1+x}^{2\alpha}-\br{1-x}^{2\alpha}-2\br{\frac{x^2}2}^\alpha
		\eqfs
	\end{equation*}
	We have
	\begin{align*}
		\frac{\partial_\alpha f(\alpha, x)}{\alpha} &= \br{1+3x^2}^\alpha \log\brOf{1+3x^2}-\br{1+x}^{2\alpha}\log\brOf{(1+x)^2}
		\\&\hphantom{=}
		-\br{1-x}^{2\alpha}\log\brOf{(1-x)^2}-2\br{\frac{x^2}2}^\alpha\log\brOf{\frac{x^2}2}
		\\&=:g(x,\alpha)\eqcm
	\end{align*}
	\begin{align*}
		\frac{\partial_\alpha g(\alpha, x)}{\alpha} &=
		\br{1+3x^2}^\alpha \log\brOf{1+3x^2}^2-\br{1+x}^{2\alpha}\log\brOf{(1+x)^2}^2
		\\&\hphantom{=}
		-\br{1-x}^{2\alpha}\log\brOf{(1-x)^2}^2-2\br{\frac{x^2}2}^\alpha\log\brOf{\frac{x^2}2}^2
		\\&\leq
		\br{1+3x^2}^\alpha \log\brOf{1+3x^2}^2-\br{1+x}^{2\alpha}\log\brOf{(1+x)^2}^2
		\\&=:h(x,\alpha)
		\eqfs
	\end{align*}
	For $x\in[0,1]$, we have $1+3x^2 \leq (1+x)^2$. Thus,
	\begin{equation*}
		\br{\frac{1+3x^2}{(1+x)^2}}^\alpha \leq 1 \leq \br{\frac{\log\brOf{(1+x)^2}}{\log\brOf{1+3x^2}}}^2
		\eqfs
	\end{equation*}
	Thus,
	\begin{equation*}
		\br{1+3x^2}^\alpha \log\brOf{1+3x^2}^2 \leq \br{1+x}^{2\alpha}\log\brOf{(1+x)^2}^2
	\end{equation*}
	Thus, $h(x,\alpha) \leq 0$ and $\partial_\alpha g(\alpha,x) \leq 0$. Thus, $g(\alpha,x) \geq g(1,x)$ and
	\begin{align*}
		g(1,x) &= 
		\br{1+3x^2} \log\brOf{1+3x^2}-\br{1+x}^{2}\log\brOf{(1+x)^2}
		\\&\hphantom{=}
		-\br{1-x}^{2}\log\brOf{(1-x)^2}-2\br{\frac{x^2}2}\log\brOf{\frac{x^2}2}
		\\&=:\ell(x)
		\eqfs
	\end{align*}
	The next lemma shows $\ell(x) \geq 0$.
	Thus, $g(\alpha,x) \geq g(1,x) \geq 0$. Thus $\partial_\alpha f(\alpha,x) \geq 0$. Thus, $f(\alpha,x) \leq f(1,x)$ and
	\begin{equation*}
		f(1,x) = 1+\br{1+3x^2}-\br{1+x}^{2}-\br{1-x}^{2}-2\br{\frac{x^2}2} = 0
		\eqfs
	\end{equation*}
	Thus, $f(\alpha, x) \leq 0$.
\end{proof}
\begin{lemma}
	Let $x \in \R$.
	Define
	\begin{align*}
		f(x) &:= \br{1+3x^2} \log\brOf{1+3x^2}-\br{1+x}^{2}\log\brOf{(1+x)^2}
				\\&\hphantom{=}-\br{1-x}^{2}\log\brOf{(1-x)^2}-x^2\log\brOf{\frac{x^2}2}
				\eqfs
	\end{align*}
	Then
	\begin{equation*}
		f(x) \geq 0
		\eqfs
	\end{equation*}
\end{lemma}
\begin{proof}
	Let us first calculate some derivatives:
	\begin{align*}
		f(x) &= x^2 \log\brOf{\frac{2\br{1+3x^2}^3}{x^2\br{1-x^2}^2}} - 4 x \log\brOf{\frac{1+x}{1-x}} + \log\brOf{\frac{1+3x^2}{\br{1-x^2}^2}}
		\eqcm
		\\
		f\pr(x) &= 2 x \log\brOf{\frac{2\br{1+3x^2}^3}{x^2\br{1-x^2}^2}} - 4 \log\brOf{\frac{1+x}{1-x}}\eqcm
		\\
		f\prr(x) &= 2 \log\brOf{\frac{2\br{1+3x^2}^3}{x^2\br{1-x^2}^2}} - \frac{12}{1+3x^2}\eqcm
		\\
		f\prrr(x) &= \frac{4\br{9x^4+24x^2-1}}{x(1-x)(1+x)\br{3x^2+1}^2}\eqcm
		\\
		f^{(4)}(x) &= \frac{4\br{81x^8+324x^6-186x^4+36x^2+1}}{x^2\br{1-x^2}^2\br{3x^2+1}^3}
		\eqfs
	\end{align*}
	We consider the cases $x \in [0,\frac1{10}]$ and  $x \in [\frac1{10},1]$ separately, and start with the latter.
	For $x_0 \in (0,1)$, define
	\begin{equation*}
		g_{x_0}(x) := f(x_0) + f\pr(x_0)\br{x-x_0} + \frac12f\prr(x_0)\br{x-x_0}^2 + \frac16f\prrr(x_0)\br{x-x_0}^3
		\eqfs
	\end{equation*}
	Then the Taylor-Expansion for $x\in(0,1)$ is
	\begin{equation*}
		f(x) = g_{x_0}(x) + \frac{1}{24} f^{(4)}(\xi(x,x_0)) \br{x-x_0}^4
		\eqcm
	\end{equation*}
	with suitable $\xi(x,x_0)$.
	One can show that $81x^4+324x^3-186x^2+36x^1+1 > 0$ for all $x\geq 0$. In particular, $f^{(4)}(x) \geq 0$ for $x \in[0,1]$.
	Thus,
	\begin{equation*}
		f(x) \geq  g_{x_0}(x)
		\eqfs
	\end{equation*}
	We use $x_0 = \frac13$:
	\begin{align*}
		f\brOf{\frac13} &= \frac{10}{3} \log(3) - \frac{47}{9}\log(2)\eqcm
		\\
		f\pr\brOf{\frac13} &= 2\log(3) - \frac{10}{3}\log(2)\eqcm
		\\
		f\prr\brOf{\frac13} &= -9+2\log(2)+6\log(3)\eqcm
		\\
		f\prrr\brOf{\frac13} &= \frac{27}{2}
		\eqfs
	\end{align*}
	One can show that $g_{\frac13}(x) \geq 0$ for $x \geq \frac1{10}$. Thus, $f(x)\geq 0$ for $x \in [\frac1{10},1]$.
	
	The case $x \in [0,\frac1{10}]$ is left.
	One can show
	\begin{align*}
		 4 \log\brOf{\frac{1+x}{1-x}}  \leq 10x \leq 2 x \log\brOf{\frac{2\br{1+3x^2}^3}{x^2\br{1-x^2}^2}}
	\end{align*}
	for $x \in [0,\frac1{10}]$.
	This implies $f\pr(x) \geq 0$. Together with $f(0)=0$, this yields $f(x)\geq 0$ for $x \in [0,\frac1{10}]$.
\end{proof}
\begin{lemma}[$\frac12 b\geq a-sc$ and $a \leq b$]\label{lmm:bgeqascandaleqb}
	Let $\alpha\in[\frac12,1]$.
	Let $a,b,c\geq0$, $s\in[-1,1]$.
	Assume 
	\begin{equation*}
		0 \leq c - a \leq a -sc \leq \frac b2 \leq sc \leq 2a-c \leq a \leq c
	\end{equation*}
	and
	$a \leq b$.
	Then
	\begin{equation*}
		a^{2\alpha}-c^{2\alpha}-|a-b|^{2\alpha}+(c^2 - 2 s c b + b^2)^\alpha 
		\leq 
		2 b (a-sc)^{2\alpha-1}
		\eqfs
	\end{equation*}
\end{lemma}
\begin{proof} 
	Define
	\begin{equation*}
		f(a,b,y,w) := a^{2\alpha}-(y+a)^{2\alpha}-(b-a)^{2\alpha}+((y+a)^2 - 2 w b + b^2)^\alpha - 2 b (a-w)^{2\alpha-1}
		\eqfs
	\end{equation*}
	We have
	\begin{equation*}
		\partial_y f(a,b,y,w) = -2\alpha(y+a)^{2\alpha-1}+\alpha\br{2(y+a)}\br{(y+a)^2 - 2 w b + b^2}^{\alpha-1}
		\eqfs
	\end{equation*}
	Because of $\frac{b}2 \leq w$, we have
	\begin{equation*}
		(y+a)^2 \geq  (y+a)^2 - 2 w b + b^2
		\eqfs
	\end{equation*}
	Thus $\partial_y f(a,b,y,w) \geq 0$.
	Thus, for $y\in[0,a-w]$, we have $f(a,b,y,w) \leq f(a,b,a-w,w)$. We have
	\begin{align*}
		&f(a,b,a-w,w) 
		\\&= a^{2\alpha}-(2a-w)^{2\alpha}-(b-a)^{2\alpha}+((2a-w)^2 - 2 w b + b^2)^\alpha - 2 b (a-w)^{2\alpha-1} 
		\\&=: g(a,b,w)
		\eqcm
	\end{align*}
	\begin{align*}
		&\partial_b g(a,b,w) 
		\\&= -2\alpha(b-a)^{2\alpha-1}+2\alpha(b-w)((2a-w)^2 - 2 w b + b^2)^{\alpha-1} - 2 (a-w)^{2\alpha-1} 
		\\&\leq  2\alpha h(a,b,w)
		\eqcm
	\end{align*}
	\begin{equation*}
		h(a,b,w) = -(b-a)^{2\alpha-1}+(b-w)((2a-w)^2 - 2 w b + b^2)^{\alpha-1} -  (a-w)^{2\alpha-1}
		\eqfs
	\end{equation*}
	The conditions $0\leq a-w \leq \frac b2 \leq w$ and $a \leq b$ imply $w \leq a \leq b$. We have,
	\begin{align*}
		&\partial_a^2 h(a,b,w) \\= &-(2\alpha-1)(2\alpha-2)(b-a)^{2\alpha-3}\\&+4(\alpha-1)(\alpha-2)(2a-w)^2(b-w)((2a-w)^2 - 2 w b + b^2)^{\alpha-3} \\&-  (2\alpha-1)(2\alpha-2)(a-w)^{2\alpha-3} \\\geq&0
		\eqcm
	\end{align*}
	\begin{align*}
		h(w,b,w) 
		&=
		-(b-w)^{2\alpha-1}+(b-w)((2w-w)^2 - 2 w b + b^2)^{\alpha-1} -  (w-w)^{2\alpha-1}
		\\&=
		-(b-w)^{2\alpha-1}+(b-w)^{2\alpha-1} = 0
		\eqcm
	\end{align*}
	\begin{equation*}
		h(b,b,w) =-(b-b)^{2\alpha-1}+(b-b)((2b-w)^2 - 2 w b + b^2)^{\alpha-1} -  (b-b)^{2\alpha-1} \leq 0
		\eqfs
	\end{equation*}
	Thus, $h(a,b,w) \leq 0$ for all $a \in [w,b]$.
	Thus, $\partial_b g(a,b,w) \leq 0$.
	The conditions for $g$ are $0 \leq a-w \leq \frac b2 \leq w \leq a \leq b$.
	As $a \leq b$ and $\frac b2 \leq w$, we have $a \leq 2 w$ and thus $a \geq 2a-2w$.
	\begin{align*}
		g(a,a,w) 
		&= a^{2\alpha}-(2a-w)^{2\alpha}-(a-a)^{2\alpha}+((2a-w)^2 - 2 w a + a^2)^\alpha -
		\\&\qquad 2 a (a-w)^{2\alpha-1}
		\\&=
		a^{2\alpha}-(2a-w)^{2\alpha}+(5 a^2 - 6 a w + w^2)^\alpha - 2 a (a-w)^{2\alpha-1}
		\eqfs
	\end{align*}
	Set $w = a - y$, $y\in[0,a]$. We have
	\begin{align*}
		\ell(a,y) &:= a^{2\alpha}+(5 a^2 - 6 a (a-y) + (a-y)^2)^\alpha - (a+y)^{2\alpha} - 2 a y^{2\alpha-1}
		\\&= 
		a^{2\alpha}+(4 a y + y^2)^\alpha - (a+y)^{2\alpha} - 2 a y^{2\alpha-1}
		\eqfs
	\end{align*}
	We have $g(a,a,w) = \ell(a,y)$.
	Under the condition $0 \leq y \leq a$, we have $\ell(a,y) \leq 0$, cf next lemma.
\end{proof}
\begin{lemma}
		Let $\alpha\in[\frac12, 1]$, $a,y\geq0$.
		Assume $y \leq a$.
		Then
		\begin{equation*}
			a^{2\alpha} + (4 a y + y^2)^\alpha 
			\leq
			(a+y)^{2\alpha} + 2 a y^{2\alpha-1}
			\eqfs
		\end{equation*}
\end{lemma}
\begin{proof}
	For $y\leq a$ define
	\begin{equation*}
		g(a,y) := 2a^{2\alpha-1} + 4^\alpha y^\alpha a^{\alpha-1} - 2(a+y)^{2\alpha-1} - 2y^{2\alpha-1}
		\eqfs
	\end{equation*}
	We have
	\begin{equation*}
		\partial_a g(a,y) = 2(2\alpha-1)a^{2\alpha-2} + 4^\alpha(\alpha-1) y^\alpha a^{\alpha-2} - 2(2\alpha-1)(a+y)^{2\alpha-2}
		\eqcm
	\end{equation*}
	\begin{equation*}
		\partial_y \partial_a g(a,y) = 4^\alpha (\alpha-1) \alpha y^{\alpha-1} a^{\alpha-2} - 2(2\alpha-1)(2\alpha-2)(a+y)^{2\alpha-3}
		\eqfs
	\end{equation*}
	Set $z := \frac{y}{a} \in[0,1]$. Then
	\begin{equation*}
	 \frac{\partial_y \partial_a g(a,y)}{a^{2\alpha-3}} =  (\alpha-1) \br{4^\alpha \alpha z^{\alpha-1} - 4(2\alpha-1)(1+z)^{2\alpha-3}}
	 \eqfs
	\end{equation*}
	The function $z \mapsto \frac{\br{1+z}^{2\alpha-3}}{z^{\alpha-1}}$ is maximized in $[0,1]$ at $\frac{1-\alpha}{2-\alpha}$ with maximum 
	\begin{equation*}
		\frac{\br{1+\frac{1-\alpha}{2-\alpha}}^{2\alpha-3}}{\br{\frac{1-\alpha}{2-\alpha}}^{\alpha-1}}
		=
		\frac{\br{3-2\alpha}^{2\alpha-3} \br{2-\alpha}^{2-\alpha}}{\br{1-\alpha}^{\alpha-1}} 
		\eqfs
	\end{equation*}
	The function $\alpha\mapsto \frac{4^\alpha}{4(2\alpha-1)}$ is decreasing on $[\frac12, 1]$ and equal to $1$ at $\alpha=1$. 
	Thus, using \autoref{lmm:alphaFraction} below, 
	\begin{equation*}
		\frac{4^\alpha \alpha}{4(2\alpha-1)} \geq \alpha \geq \frac{\br{3-2\alpha}^{2\alpha-3} \br{2-\alpha}^{2-\alpha}}{\br{1-\alpha}^{\alpha-1}} 
		\eqfs
	\end{equation*}
	Thus,
	\begin{equation*}
		4^\alpha \alpha z^{\alpha-1} \geq 4(2\alpha-1) \br{1+z}^{2\alpha-3}
		\eqfs
	\end{equation*}
	Thus, $\partial_y \partial_a g(a,y) \leq 0$. Thus, $\partial_a g(a,y) \leq \partial_a g(a,0)$ and
	\begin{equation*}
		\partial_a  g(a,0) = 2(2\alpha-1)a^{2\alpha-2} - 2(2\alpha-1)(a)^{2\alpha-2} = 0
		\eqfs
	\end{equation*}
	Thus, $\partial_a g(a,y) \leq 0$. Thus, $g(a,y) \leq g(y,y)$, and
	\begin{equation*}
		g(y,y) = \br{2 + 4^\alpha - 2^{2\alpha} - 2}y^{2\alpha-1} = 0
		\eqfs
	\end{equation*}
	Thus,  $g(a,y) \leq 0$. We have
	\begin{align*}
		&a^{2\alpha} + (4 a y + y^2)^\alpha - (a+y)^{2\alpha} - 2 a y^{2\alpha-1}
		\\&\leq 
		a^{2\alpha} + (4 a)^\alpha y^\alpha +y^{2\alpha} - (a+y)^{2\alpha} - 2 a y^{2\alpha-1}
		\\&=:f(a,y)
		\eqcm
	\end{align*}
	\begin{equation*}
		\partial_a f(a,y) = 2\alpha a^{2\alpha-1} + 4^\alpha \alpha y^\alpha a^{\alpha-1} - 2\alpha(a+y)^{2\alpha-1} - 2y^{2\alpha-1}
		= \alpha g(a,y) \leq 0
		\eqfs
	\end{equation*}
	Thus,
	\begin{equation*}
		f(a,y) \leq f(y,y) = \br{1+4^\alpha+1-4^\alpha-2} y^{2\alpha} = 0
		\eqfs
	\end{equation*}
\end{proof}
\begin{lemma}\label{lmm:alphaFraction}
	Let $x\in[\frac12,1 ]$. Then
	\begin{equation*}
		\frac{\br{3-2x}^{2x-3} \br{2-x}^{2-x}}{\br{1-x}^{x-1}} \leq x
		\eqfs
	\end{equation*}
\end{lemma}
\begin{proof}
	Define 
	\begin{equation*}
		f(x) := x \log(2-x) + 3 \log(3-2x) - 2x\log(3-2x) - 2\log(2-x) + \log(x)
		\eqfs
	\end{equation*}
	Then 
	\begin{align*}
		f\pr(x) &= \frac1x - 2 \log(3 - 2 x) + \log(2 - x) - 1\eqcm\\
		f\prr(x) &= -\frac1{x^2} - \frac1{2 - x} + \frac4{3 - 2 x}\eqcm\\
		f\prrr(x) &= \frac2{x^3} - \frac1{(2 - x)^2} + \frac8{(3 - 2 x)^2}\eqfs
	\end{align*}
	As $8(2 - x)^2 \geq  (3 - 2 x)^2$ for $x \leq \frac52 - \frac1{\sqrt2}$, we have
	\begin{equation*}
		\frac1{(2 - x)^2} \leq \frac8{(3 - 2 x)^2}
	\end{equation*}
	and, therefore, $f\prrr(x) \geq 0$ for $x\in[\frac12,1]$. Hence, $f\pr$ is convex. We have $f\pr(1) = 0$. Thus, between $\frac12$ and $1$, $f\pr$ could be nonnegative, nonpositive or swap signs from positive to negative. In all cases, the minimum of $f$ is attained at the borders of the interval. As $f(1) = 0$ and $f(\frac12) = \frac12\log(2^5/3^3) \geq 0$, we have shown $f(x) \geq 0$ for all $x\in [\frac12, 1]$. Thus,
	\begin{equation*}
		x \log\brOf{\frac{2-x}{(3-2x)^2}} \geq \log\brOf{\frac{(2-x)^2}{x(3-2x)^3}}
		\eqfs
	\end{equation*}
	Applying the exponential function yields
	\begin{equation*}
		\br{\frac{2-x}{(3-2x)^2}}^x \geq \frac{(2-x)^2}{x(3-2x)^3}
		\eqfs
	\end{equation*}
	As $(1-x)^x \geq 1-x$, we obtain
	\begin{equation*}
		x(1-x)^x(3-2x)^3(2-x)^x \geq (1-x)(2-x)^2(3-2x)^{2x}\eqfs
	\end{equation*}
	Reordering the factors yields the desired inequality.
\end{proof}
\begin{lemma}[$\frac12 b\geq a-sc$ and $a \geq b$]\label{lmm:bgeqascandageqb}
	Let $\alpha\in[\frac12,1]$.
	Let $a,b,c\geq0$, $s\in[-1,1]$.
	Assume 
	\begin{equation*}
		0 \leq c - a \leq a -sc \leq \frac b2 \leq sc \leq 2a-c \leq a \leq c
	\end{equation*}
	and
	$a \geq b$.
	Then
	\begin{equation*}
		a^{2\alpha}-c^{2\alpha}-(a-b)^{2\alpha}+(c^2 - 2 s c b + b^2)^\alpha 
		\leq 
		2 b (a-sc)^{2\alpha-1}
		\eqfs
	\end{equation*}
\end{lemma}
\begin{proof}
	Define
	\begin{equation*}
		f(a,b,c,w) := a^{2\alpha}-c^{2\alpha}-(a-b)^{2\alpha}+(c^2 - 2 w b + b^2)^\alpha 
			- 2 b (a-w)^{2\alpha-1}
			\eqfs
	\end{equation*} 
	We have
	\begin{equation*}
		\partial_c f(a,b,c,w) =2\alpha\br{-c^{2\alpha-1}+c(c^2 - 2 w b + b^2)^{\alpha-1}}
		\eqfs
	\end{equation*} 
	Because of $2w \geq b$, we have $\partial_c f(a,b,c,w) \leq 0$.
	\begin{equation*}
		f(a,b,a,w) = -(a-b)^{2\alpha}+(a^2 - 2 w b + b^2)^\alpha - 2 b (a-w)^{2\alpha-1}
		\eqfs
	\end{equation*}
	Set $x := a-w$.
	The conditions for $x$ are $0 \leq x \leq \frac b2 \leq w$ and $b \leq x+w$. Define
	\begin{equation*}
		g(x,b,w) := -(x+w-b)^{2\alpha}+((x+w)^2 - 2 w b + b^2)^\alpha - 2 b x^{2\alpha-1}
		\eqfs
	\end{equation*}
	We have
	\begin{equation*}
		\partial_w g(x,b,w) =-2\alpha(x+w-b)^{2\alpha-1}+2\alpha(x+w-b)((x+w)^2 - 2 w b + b^2)^{\alpha-1} 
		\eqfs
	\end{equation*}
	We have
	\begin{equation*}
		(x+w-b)^2 - \br{(x+w)^2 - 2 w b + b^2} = -2 b x \leq 0
		\eqfs
	\end{equation*}
	Thus,
	\begin{equation*}
		-2\alpha(x+w-b)^{2\alpha-2} + 2\alpha\br{(x+w)^2 - 2 w b + b^2}^{\alpha-1} \leq 0
		\eqfs
	\end{equation*}
	As $a \geq b$ and thus $x+w-b \geq 0$, $\partial_w g(x,b,w) \leq 0$.
	We have $w \geq b-x \geq \frac b2$ and 
	\begin{equation*}
		g(x,b,b-x)
		=
		(2 x b)^\alpha - 2 b x^{2\alpha-1}
		\leq 0
	\end{equation*}
	because $x\leq b$.
\end{proof}
\subsection{The Case $|a-c| \geq ra-sc$}\label{ssec:rascleqac}
\begin{lemma}[Case 1.2]
	Let $\alpha\in[\frac12,1]$.
	Let $a,b,c\geq0$.
	Assume $a \geq c$ and $a+c\geq 2b$.
	Then 
	\begin{equation*}
		a^{2\alpha}-c^{2\alpha}-\br{(a-b)^2+4cb}^{\alpha}+(c+b)^{2\alpha} \leq 8\alpha 2^{-2\alpha} b (a-c)^{2\alpha-1}
		\eqfs
	\end{equation*}
\end{lemma}
This lemma follows from the next two lemmas, which split the proof of this inequality into two cases.
\begin{lemma}[Case 1.2, Merging]
	Let $\alpha\in[\frac12,1]$.
	Let $a,b,c\geq0$.
	Assume $a \geq b+c$.
	Then 
	\begin{equation*}
		a^{2\alpha}-c^{2\alpha}-\br{(a-b)^2+4cb}^{\alpha}+(c+b)^{2\alpha} 
		\leq 
		2^{1-2\alpha} \br{	(a-c + b)^{2\alpha}  -  (a-c - b)^{2\alpha}}
	\end{equation*}
\end{lemma}
\begin{proof}	
	Set $\delta := a-b \geq c \geq 0$ and define
	\begin{equation*}
		f(b) := (\delta+b)^{2\alpha}-c^{2\alpha}-\br{\delta^2+4cb}^{\alpha}+(c+b)^{2\alpha} -2^{1-2\alpha} \br{	(\delta-c + 2b)^{2\alpha}  -  (\delta-c)^{2\alpha}}
		\eqfs
	\end{equation*}
	Then
	\begin{equation*}
		\frac{f\pr(b)}{2\alpha} =
		(\delta+b)^{2\alpha-1}-2c\br{\delta^2+4cb}^{\alpha-1}+(c+b)^{2\alpha-1} -2^{2-2\alpha} (\delta-c + 2b)^{2\alpha-1}
		\eqcm
	\end{equation*}
	\begin{align*}
		\frac{f\pr(b)}{2\alpha} 
		\leq 
		(\delta+2b+c)^{2\alpha-1}- (\delta+2b-c)^{2\alpha-1} - 2c\br{\delta^2+4cb}^{\alpha-1} =: g(c)
		\eqcm
	\end{align*}
	\begin{align*}
		g\pr(c) 
		&= (2\alpha-1)(\delta+2b+c)^{2\alpha-2}+ (2\alpha-1)(\delta+2b-c)^{2\alpha-2} - 
		\\&\qquad 2\br{\delta^2+4cb}^{\alpha-1} -8cb(\alpha-1)\br{\delta^2+4cb}^{\alpha-2}
		\eqcm
	\end{align*}
	\begin{align*}
		&g\prr(c)\\ &= (2\alpha-1)(2\alpha-2)(\delta+2b+c)^{2\alpha-3}- 
		(2\alpha-1)(2\alpha-2)(\delta+2b-c)^{2\alpha-3}  - A\eqcm
	\end{align*}
	where
	\begin{align*}
		A := &\,2(4b)(\alpha-1)\br{\delta^2+4cb}^{\alpha-2} 
		8b(\alpha-1)\br{\delta^2+4cb}^{\alpha-2}
		+\\&\qquad
		8cb(4b)(\alpha-1)(\alpha-2)\br{\delta^2+4cb}^{\alpha-3}
		\\=&
		16b(\alpha-1)\br{\delta^2+4cb}^{\alpha-3} 
		\br{\delta^2+4cb + 2cb(\alpha-2)}
		\eqfs
	\end{align*}
	Thus, 	$g\prr(c) \geq 0$.
	We have $0 \leq c \leq \delta$ and
	\begin{equation*}
		g(0) =
		(\delta+2b)^{2\alpha-1}- (\delta+2b)^{2\alpha-1}-0=0
		\eqcm
	\end{equation*}
	\begin{equation*}
		g(\delta) = (2\delta+2b)^{2\alpha-1}- (2b)^{2\alpha-1} - 2\delta\br{\delta^2+4\delta b}^{\alpha-1} =: h(\delta)\eqcm
	\end{equation*}
	\begin{equation*}
		h\pr(\delta) = 2(2\alpha-1)(2\delta+2b)^{2\alpha-2}-2\br{\delta^2+4\delta b}^{\alpha-1}-2\delta(\alpha-1)(2\delta+4b)\br{\delta^2+4\delta b}^{\alpha-2}\eqfs
	\end{equation*}
	For $\alpha\in[\frac12,1]$, we have
	\begin{align*}
		&\br{\delta^2+4\delta b}^{\alpha-1}+\delta(\alpha-1)(2\delta+4b)\br{\delta^2+4\delta b}^{\alpha-2}
		\\&=
		\br{\delta^2+4\delta b}^{\alpha-2}\br{\delta^2+4\delta b+\delta (\alpha-1)(2\delta+4b)}
		\\& \geq 
		\br{\delta^2+4\delta b}^{\alpha-2}\br{\delta^2+4\delta b+ -\delta(1\delta+2b)}
		\geq0 
		\eqfs
	\end{align*}
	Thus, 
	\begin{equation*}
	h\pr(\delta) \leq 0
	\eqfs
	\end{equation*}
	We have
	\begin{equation*}
		h(0) = (2b)^{2\alpha-1}- (2b)^{2\alpha-1} - 0 = 0
		\eqfs
	\end{equation*}
	Thus, $h(\delta)\leq 0$, thus, $g(b=\delta) \leq 0$, thus $g(b) \leq 0$, thus $f\pr(b)\leq 0$
	\begin{equation*}
		f(0) = (\delta)^{2\alpha}-c^{2\alpha}-\br{\delta}^{2\alpha}+(c)^{2\alpha} -2^{1-2\alpha} \br{	(\delta-c)^{2\alpha}  -  (\delta-c)^{2\alpha}} =0
		\eqfs
	\end{equation*}
	Thus, $f(b) \leq 0$.
\end{proof}
\begin{lemma}[Case 1.2, $b\geq a-c$]
	Let $\alpha\in[\frac12,1]$.
	Let $a,b,c\geq0$.
	Assume $b \geq a-c \geq 2\max(0, b-c)$.
	Then 
	\begin{equation*}
		a^{2\alpha}-c^{2\alpha}-\br{(a-b)^2+4cb}^{\alpha}+(c+b)^{2\alpha} 
		\leq 
		2 b^{2\alpha-1} (a-c)
		\leq
		2 b (a-c)^{2\alpha-1}
		\eqfs
	\end{equation*}
\end{lemma}
\begin{proof}
	As $b\geq a-c$ and $2\alpha-1 \in[0,1]$, we have $b^{2\alpha-1} (a-c) \leq b (a-c)^{2\alpha-1}$.
	Define
	\begin{equation*}
		f(a) := a^{2\alpha}-c^{2\alpha}-\br{(a-b)^2+4cb}^{\alpha}+(c+b)^{2\alpha}  -2 b^{2\alpha-1} (a-c)
		\eqfs
	\end{equation*}
	We have
	\begin{align*}
		f\pr(a) &= 	2\alpha a^{2\alpha-1}-2\alpha (a-b)\br{(a-b)^2+4cb}^{\alpha-1}-2 b^{2\alpha-1}
		\\ {\scriptstyle a \geq b,c \text{ and } 2\alpha-2\leq 0} \qquad&\leq
		2\alpha a^{2\alpha-1}-2\alpha (a-b)\br{a+b}^{2\alpha-2}-2 b^{2\alpha-1}
		\\&=
		2\br{\alpha a^{2\alpha-1}- \alpha \frac{a-b}{a+b}\br{a+b}^{2\alpha-1} - b^{2\alpha-1}}
		\eqfs
	\end{align*}
	Set $x = \br{\frac{a-b}{a+b}}^{\frac{1}{2\alpha-1}} \leq 1$, $y = \alpha^{\frac{1}{2\alpha-1}} \leq 1$.
	Then 
	\begin{align*}
		\alpha a^{2\alpha-1}- \alpha \frac{a-b}{a+b}\br{a+b}^{2\alpha-1} - b^{2\alpha-1}
		&\leq
		(ya)^{2\alpha-1}- \br{xya+xyb}^{2\alpha-1} - b^{2\alpha-1}
		\\&\leq
		\br{ya-xya-xyb}^{2\alpha-1} - b^{2\alpha-1}
		\\&\leq
		\br{ya-xya-xyb-b}^{2\alpha-1}
		\\&=
		\br{(y-xy)a-(xy+1)b}^{2\alpha-1}
		\\&\leq
		\br{a-b}^{2\alpha-1}
		\\&\leq
		0\eqfs
	\end{align*}
	Thus, $f\pr(a) \leq 0$. Thus, only need to show $f(b) \leq 0$.
	Assume $b \geq c$. Then
	\begin{align*}
		f(b) 
		&= 
		b^{2\alpha}-c^{2\alpha}-\br{4cb}^{\alpha}+(c+b)^{2\alpha}  -2 b^{2\alpha-1} (b-c)
		\\&=
		-c^{2\alpha}-\br{4cb}^{\alpha}+(c+b)^{2\alpha}  - b^{2\alpha}+2 b^{2\alpha-1} c
		\\&\leq
		(c+b)^{2\alpha}- b^{2\alpha}-c^{2\alpha}-\br{4^\alpha-2}\br{cb}^{\alpha}
		\\&= 
		(c+b)^{2\alpha}-(b^\alpha-c^\alpha)^2-4^\alpha\br{cb}^{\alpha}
		\eqfs
	\end{align*}
	Thus, the next lemma implies $f(b) \leq 0$.
\end{proof}
\begin{lemma}\label{lmm:alpha_binom}
	Let $\alpha\in[\frac12,1]$, $x,y\geq0$.
	Then 
	\begin{equation*}
		(x+y)^{2\alpha}-(x^\alpha-y^\alpha)^2 \leq (4xy)^\alpha
		\eqfs
	\end{equation*}
\end{lemma}
We need two further lemmas before we prove this inequality.
\begin{lemma}\label{lmm:slogs}
	For $s\in[0,\frac12]$, we have
	\begin{equation*}
		\frac{1-s}{s} \leq \frac{\log(s)}{\log(1-s)}
		\eqfs
	\end{equation*}
\end{lemma}
\begin{proof}
	Define
	\begin{equation*}
		f(s) := s\log(s) - (1-s)\log(1-s)
		\eqfs
	\end{equation*}
	It hold
	\begin{equation*}
		f\prr(s) = \frac1s-\frac1{1-s} 
		\eqfs
	\end{equation*}
	Thus, $f\prr(s) \geq 0$ for $s\leq\frac12$.
	We have $f(0)=f(\frac12)=0$.
	Thus, $f(s) \leq 0$. Thus,
	\begin{equation*}
		s\log(s) \leq (1-s)\log(1-s)
		\eqfs
	\end{equation*}
	Because of $\log(1-s) \leq 0$, thus implies 
	\begin{equation*}
		\frac{1-s}{s} \leq \frac{\log(s)}{\log(1-s)}
		\eqfs
	\end{equation*}
\end{proof}
\begin{lemma}\label{lmm:abxfrac}
	Let $a,b\geq0$, $x\in[1,2]$.
	Define
	\begin{equation*}
		f(x) := \frac{a^{x}-b^{x}}{(a+b)^{x}}
		\eqfs
	\end{equation*}
	Assume $a\geq b$.
	Then $f\prr(x) \leq 0$.
	In particular,
	\begin{equation*}
		\inf_{x\in[1,2]} f(x) = f(1) = f(2) = \frac{a^2-b^2}{(a+b)^2} = \frac{a-b}{a+b}
		\eqfs
	\end{equation*}
\end{lemma}
\begin{proof}
	We have
	\begin{equation*}
		f\prr(x) = \br{a+b}^{-x} \br{a^x \log\brOf{\frac{a}{a+b}}^2 - b^x \log\brOf{\frac{b}{a+b}}^2}
		\eqfs
	\end{equation*}
	Set $s = \frac{b}{a+b}$. Then $1-s = \frac{a}{a+b}$. Then \autoref{lmm:slogs} implies
	\begin{equation*}
		\frac{a}{b} \leq \frac{\log\brOf{\frac{b}{a+b}}}{\log\brOf{\frac{a}{a+b}}}
		\eqfs
	\end{equation*}
	Thus,
	\begin{align*}
		\br{\frac{a}{b}}^x \leq \br{\frac{a}{b}}^2 \leq \br{\frac{\log\brOf{\frac{b}{a+b}}}{\log\brOf{\frac{a}{a+b}}}}^2
		\eqfs
	\end{align*}
	Thus,
	\begin{equation*}
		a^x \log\brOf{\frac{a}{a+b}}^2 \leq b^x \log\brOf{\frac{b}{a+b}}^2
		\eqfs
	\end{equation*}
	Thus, $f\prr(x) \leq 0$.
\end{proof}
\begin{proof}[of \autoref{lmm:alpha_binom}]
	For $z\geq1$ define
	\begin{equation*}
		f(z) := \br{z+2+z^{-1}}^{\alpha}-z^\alpha-z^{-\alpha}
		\eqfs
	\end{equation*}
	We will show that $f(z) \leq 4^\alpha-2$.
	This implies
	\begin{equation*}
		\frac{(z+1)^{2\alpha}-z^{2\alpha}-1}{z^\alpha} \leq  4^\alpha-2
		\eqfs
	\end{equation*}
	Thus,
	\begin{equation*}
		(z+1)^{2\alpha} \leq \br{ 4^\alpha-2}z^\alpha+z^{2\alpha}+1
		\eqfs
	\end{equation*}
	By setting $z = \frac xy$  for $x\geq y$, we obtain
	\begin{equation*}
		(x+y)^{2\alpha}-(x^\alpha-y^\alpha)^2 \leq (4xy)^\alpha
		\eqfs
	\end{equation*}
	The condition $x\geq y$ can be dropped because of symmetry.
	It remains to show that $f(z) \leq 4^\alpha-2$ is indeed true.
	We have $f(1) = 4^\alpha-2$. To finish the proof, we will show $f\pr(z) \leq 0$. Define
	\begin{equation*}
		g(z) := (z^2-1)(z+2)^{2\alpha}-(z+1)^2(z^{2\alpha}-1)
		\eqfs
	\end{equation*}
	Then
	\begin{equation*}
		f\pr(z) \frac{z^{\alpha+2}\br{z+2+z^{-1}}}{\alpha} = g(z)
		\eqfs
	\end{equation*}
	We show $g(z) \leq 0$, and therefore $f\pr(z)\leq0$, by applying \autoref{lmm:abxfrac} with $a=z$, $b=1$, and $x=2\alpha$:
	\begin{equation*}
		\frac{z^{x}-1^{x}}{(z+1)^{x}} \geq \frac{z^{2}-1^{2}}{(z+1)^{2}}
		\eqcm
	\end{equation*}
	which implies
	\begin{equation*}
		\br{z^{2\alpha}-1}(z+1)^{2} \geq \br{z^{2}-1}(z+1)^{2\alpha}
		\eqfs
	\end{equation*}
\end{proof}
According to \autoref{rmk:outline}, we have now finally finished to proof of \autoref{con:ana} and therefore of \autoref{thm:power_inequ}.
\end{appendices}
\phantomsection
\addcontentsline{toc}{section}{Index of Notations}
\printindex[inot]
\phantomsection
\addcontentsline{toc}{section}{Index of Assumptions}
\printindex[iass]
\phantomsection
\addcontentsline{toc}{section}{References}
\printbibliography
\end{document}